\documentclass[11pt]{article}
\usepackage{amsmath,amssymb,latexsym,graphicx,capt-of,lipsum,blindtext,wrapfig,listings,color,mathtools,amsfonts,amsthm,hyperref,bm}
\usepackage[tmargin=0.6in,bmargin=0.75in,
lmargin=2.1cm, rmargin=2.1cm]{geometry}
\usepackage[utf8]{inputenc}
\usepackage[english]{babel}
\usepackage[titletoc]{appendix}
\usepackage{verbatim}
\usepackage{authblk}

\newtheorem{thm}{Theorem}[section]
\newtheorem*{thm-non}{Theorem}

\newtheorem*{lem-non}{Lemma}

\newtheorem{clm}[thm]{Claim}

\DeclareMathOperator{\divergence}{div}
\DeclareMathOperator{\curl}{curl}
\DeclareMathOperator{\dist}{dist}
\DeclareMathOperator{\spn}{span}

\numberwithin{equation}{section}

\title{Existence of Large-Data Global Weak Solutions to Kinetic Models of Nonhomogeneous Dilute Polymeric Fluids}

\author{Chuhui He\thanks{chuhui.he@maths.ox.ac.uk} ~and Endre S\"uli\thanks{endre.suli@maths.ox.ac.uk}}

\affil{\small Mathematical Institute, University of Oxford, \\Woodstock Road, Oxford OX2 6GG, UK}

\date{}

\begin{document}
\maketitle

\begin{abstract}
We prove the existence of large-data global-in-time weak solutions to a general class of coupled bead-spring chain models with finitely extensible nonlinear elastic (FENE) type spring potentials for nonhomogeneous incompressible dilute polymeric fluids in a bounded domain in $\mathbb{R}^d$, $d=2$ or $3$. The class of models under consideration involves the Navier--Stokes system with variable density, where the viscosity coefficient depends on both the density and the polymer number density, coupled to a Fokker--Planck equation with a density-dependent drag coefficient. The proof is based on combining a truncation of the probability density function with a two-stage Galerkin approximation and weak compactness and compensated compactness techniques to pass to the limits in the sequence of Galerkin approximations and in the truncation level.
\end{abstract}

\section{Introduction}
The aim of this paper is to prove the existence of global-in-time large-data weak solutions to a Navier--Stokes--Fokker--Planck system that arises in models of nonhomogeneous dilute polymeric fluids. The paper extends the results presented in \cite{MR2981267} to a more general class of models by using a different proof, which is also considerably simpler than the original proof presented in \cite{MR2981267}. We assume that the solvent, occupying a bounded open Lipschitz domain $\Omega \subset \mathbb{R}^d$, $d = 2$ or $3$, with boundary $\partial \Omega$, is an incompressible, viscous, isothermal Newtonian fluid with variable density $\rho$ and dynamic viscosity $\mu=\mu(\rho,\varrho)$, where $\varrho$ is the variable polymer number density. Let $T \in \mathbb{R}_{> 0}$ denote the length of the time interval of interest and let $Q \coloneqq \Omega \times (0,T)$ denote the associated space-time domain. We consider the following system of equations:
\begin{align}
\label{eq1} \frac{\partial \rho}{\partial t} + \nabla_x \cdot (\bm{v} \rho) & = 0 \qquad \qquad \qquad \ \, \text{in $Q$}, \\
\label{eq2} \frac{\partial (\rho \bm{v})}{\partial t} + \nabla_x \cdot (\rho \bm{v} \otimes \bm{v}) - \nabla_x \cdot (\mu(\rho, \varrho)D(\bm{v})) + \nabla_x p &= \rho \bm{f} + \nabla_x \cdot \tau \qquad \text{in $Q$}, \\
\label{eq3} \nabla_x \cdot \bm{v} &=0 \qquad \qquad \qquad \ \, \text{in $Q$},
\end{align}
subject to the initial conditions
\begin{align}\label{inicond}
\begin{aligned}
\rho(\cdot, 0) &= \rho_0(\cdot) \quad \qquad \ \ \,  \text{in $\Omega$}, \\
(\rho \bm{v})(\cdot, 0) & = (\rho_0 \bm{v}_0)(\cdot) \qquad \text{in $\Omega$},
\end{aligned}
\end{align}
and the no-slip boundary condition
\begin{equation}\label{eq6}
\bm{v} = \bm{0} \qquad \text{on $\partial \Omega \times (0,T)$}.
\end{equation}
It is assumed that each of the equations above has been written in a nondimensional form; $\rho: Q \to \mathbb{R}$ denotes the solvent density, $\bm{v}: Q \to \mathbb{R}^d$ denotes the solvent velocity, $\tau: Q \to \mathbb{R}^{d \times d}$ denotes the elastic extra-stress tensor (i.e., the polymeric part of the Cauchy stress tensor), $\bm{f}: Q \to \mathbb{R}^d$ represents the density of the external body forces, and $p: Q \to \mathbb{R}$ denotes the pressure. Here, $D(\bm{v}) \coloneqq \frac{1}{2}(\nabla_x \bm{v} + (\nabla_x \bm{v})^{\rm T})$ is the symmetric part of the velocity gradient.

In a bead-spring chain model for dilute polymers, consisting of $K+1$  beads  coupled with $K$ elastic springs to represent a polymer chain, the elastic extra-stress tensor $\tau$ is defined by a version of the Kramers expression, depending on the probability density function $\psi$ of the (random) conformation $\bm{q} \coloneqq ((\bm{q}^1)^{\rm T}, \dots, (\bm{q}^K)^{\rm T})^{\rm T} \in \mathbb{R}^{d \times K}$ of the chain, with the column vector $\bm{q}^j \coloneqq (q_1^j, \dots, q_d^j)^{\rm T}$ representing the $d$-component conformation vector of the $j$th spring in the bead-spring chain. Let $D \coloneqq D^1 \times \cdots \times D^K \subset \mathbb{R}^{d \times K}$ be the domain of admissible conformation vectors.  Typically $D^j$ is the whole space $\mathbb{R}^d$ or a bounded open ball centred at the origin $\bm{0}$ in $\mathbb{R}^d$, for each $j = 1, \dots, K$. When $K=1$, the model is referred to as the dumbbell model. 
We focus on the finitely extensible nonlinear elastic (FENE) model where $D^j = B(0, b^{\frac{1}{2}}_j)$, a ball centred at the origin $\bm{0}$ in $\mathbb{R}^d$ and of radius $b^{\frac{1}{2}}_j$, with $b_j > 0$ for each $j \in \{ 1,\dots, K \}$. The smooth function $U^j : [0, \frac{b_j}{2} ) \to [0, \infty), j = 1, \dots, K$, is the spring potential satisfying $U^j(0) = 0$, $\lim_{s \to \frac{b_j}{2}-} U^j(s) = + \infty$. The extra-stress tensor $\tau$ is then defined by the formula:
\begin{equation}\label{eq7}
\tau(x,t) \coloneqq k \left( \sum_{j=1}^K \int_D \psi(x, \bm{q}, t) \bm{q}^j {\bm{q}^j}^{\rm T} (U^j)^\prime \left( \frac{1}{2} |\bm{q}^j|^2 \right)\mathrm{d}\bm{q} - K\varrho(x,t) I \right),
\end{equation}
with $I$ denoting the $d \times d$ identity matrix, $\,\mathrm{d}\bm{q} \coloneqq \,\mathrm{d}\bm{q}^1 \cdots \,\mathrm{d}\bm{q}^K$, and the density of polymer chains (referred to as \textit{polymer number density}) located at $x$ at time $t$ given by
\begin{equation*}
\varrho(x,t) \coloneqq \int_D \psi(x, \bm{q},t) \,\mathrm{d}\bm{q}.
\end{equation*}

We define the (normalized) partial Maxwellian $M^j$ with respect to the variable $\bm{q}^j$ by
\begin{equation*}
M^j(\bm{q}^j) = \frac{1}{Z^j} \mathrm{e}^{-U^j(\frac{1}{2} |\bm{q}^j|^2)}, \quad \text{where}\  Z^j \coloneqq \int_{D^j} \mathrm{e}^{-U^j(\frac{1}{2} |\bm{q}^j|^2)} \,\mathrm{d}\bm{q}^j,
\end{equation*}
where $\,\mathrm{d}\bm{q}^j \coloneqq \mathrm{d}q^j_1 \cdots \mathrm{d}q^j_d$, $j = 1, \dots, K$. The Maxwellian in the model is then defined by 
\begin{equation*}
M(\bm{q}) \coloneqq \prod_{j=1}^K M^j(\bm{q}^j) \qquad \forall \, \bm{q} \coloneqq (\bm{q}^1, \cdots, \bm{q}^K) \in D \coloneqq D^1 \times \cdots \times D^K.
\end{equation*}

Observe that, for $\bm{q} \in D$ and $j=1,\dots, K$,
\begin{equation*}
M(\bm{q}) \nabla_{\bm{q}^j} (M(\bm{q}))^{-1} = -(M(\bm{q}))^{-1} \nabla_{\bm{q}^j} M(\bm{q}) = \nabla_{\bm{q}^j} U^j\left(\frac{1}{2}|\bm{q}^j|^2\right) = (U^j)^\prime \left(\frac{1}{2} |\bm{q}^j|^2\right) \bm{q}^j,
\end{equation*}
and, by definition,
\begin{equation*}
\int_D M(\bm{q}) \,\mathrm{d}\bm{q} = 1.
\end{equation*}

In the above equalities $\nabla_{\bm{q}^j} \coloneqq (\partial/\partial q_1^j, \dots,\partial/\partial q_d^j)^{\rm T}$, for $j=1,\dots,K$. Then we define $\divergence_{\bm{q}^j} \coloneqq \nabla_{\bm{q}^j} \cdot$. For a general mapping $\bm{q} \in D \mapsto B(\bm{q}) \in \mathbb{R}^{d \times K}$, we define $\divergence_{\bm{q}} B \coloneqq \divergence_{q^1} B^1 + \cdots + \divergence_{q^K} B^K$, where $B^j$, $j =1, \dots, K$, denote the $d$-component column vectors of the matrix $B = B(\bm{q})$. We define the $d \times K$-component differential operator $\nabla_{\bm{q}} \coloneqq (\nabla_{\bm{q}^1}, \dots, \nabla_{\bm{q}^K})$. We shall assume the following properties: for $j = 1,\dots, K$ there exist constants $c_{ji} > 0$, $i = 1,2,3,4$, and $\gamma_j > 1$ such that the spring potential $U^j$ satisfies
\begin{align}\label{eq10} 
c_{j1}[\dist(\bm{q}^j, \partial D^j)]^{\gamma_j} \leq M^j(\bm{q}^j) \leq c_{j2}[\dist(\bm{q}^j, \partial D^j)]^{\gamma_j} \qquad \forall \, \bm{q}^j \in D^j, \\
 c_{j3} \leq [\dist(\bm{q}^j, \partial D^j)] (U^j)^\prime \left( \frac{1}{2} |\bm{q}^j|^2 \right) \leq c_{j4} \qquad \forall \, \bm{q}^j \in D^j.\label{eq11}
\end{align}

Since $[U^j(\frac{1}{2} |\bm{q}^j|^2)]^2 = (-\log M^j(\bm{q}^j) + C)^2$, it follows from (\ref{eq10}) and (\ref{eq11}) that
\begin{equation*}
\int_{D^j} \left[1 + \left[U^j \left(\frac{1}{2} |\bm{q}^j|^2 \right)\right]^2 + \left[(U^j)^\prime \left(\frac{1}{2} |\bm{q}^j|^2 \right) \right]^2 \right] M^j(\bm{q}^j) \,\mathrm{d}\bm{q}^j < \infty, \qquad j = 1,\dots, K.
\end{equation*}
The probability density function $\psi$ satisfies the following Fokker--Planck equation:
\begin{align}\begin{aligned}\label{eq15}
\frac{\partial \psi}{\partial t} + \nabla_x \cdot (\bm{v} \psi) &+ \sum_{j=1}^K \nabla_{\bm{q}^j} \cdot ((\nabla_x \bm{v}) \bm{q}^j \psi) \\
&\quad = \varepsilon \Delta_x \left( \frac{\psi}{\zeta(\rho)} \right) + \frac{1}{4\lambda} \sum_{i=1}^K \sum_{j=1}^K A_{ij} \nabla_{\bm{q}^j} \left(M \nabla_{\bm{q}^i} \left(\frac{\psi}{\zeta(\rho)M} \right) \right),
\end{aligned}
\end{align}
in $\mathcal{O} \times (0,T)$, with $\mathcal{O} \coloneqq \Omega \times D$. In the above equation $\zeta(\cdot) \in \mathbb{R}_{>0}$ is a density-dependent scaled drag coefficient. Let $\partial \bar{D}^j \coloneqq D^1 \times \cdots \times D^{j-1} \times \partial D^j \times D^{j+1} \times \cdots \times D^K$. We impose the following boundary and initial conditions, for all $j = 1,\dots, K$:
\begin{align}
\label{eq16} \left[ \frac{1}{4\lambda} \sum_{i=1}^K A_{ij}M \nabla_{\bm{q}^i} \left( \frac{\psi}{\zeta(\rho) M}\right) - (\nabla_x \bm{v})\bm{q}^j \psi \right] \cdot \frac{\bm{q}^j}{|\bm{q}^j|} &= 0 \qquad \qquad \quad \, \text{on $\Omega \times \partial \bar{D}^j \times (0,T)$}, \\
\label{eq17} \varepsilon \nabla_x \left( \frac{\psi}{\zeta(\rho)} \right) \cdot \bm{n} &=0 \qquad \qquad \quad \, \text{on $\partial \Omega \times D \times (0,T)$}, \\
\psi(x, \bm{q}, 0) &= \psi_0(x,\bm{q}) \qquad \  \text{in $\Omega \times D$},
\label{eq17a}
\end{align}
where $\bm{n}$ is the unit outward normal to $\partial \Omega$. In (\ref{eq7}), the dimensionless constant $k > 0$ is a constant multiple of the product of the Boltzmann constant $k_B$ and the absolute temperature $\kappa$. In (\ref{eq15}), the centre-of-mass diffusion coefficient $\varepsilon > 0$ is defined as $\varepsilon \coloneqq (l_0/L_0)^2/(4(K+1)\lambda)$ with $l_0 \coloneqq \sqrt{k_B \kappa / H}$ signifying the characteristic microscopic length-scale and $\lambda \coloneqq (\zeta_0/4H)(U_0/L_0)$, where $\zeta_0 > 0$ is a characteristic drag coefficient and $H > 0$ is a spring-constant. The dimensionless positive parameter $\lambda$ is called the Deborah number, which characterizes the elastic relaxation property of the fluid. The constant matrix $A = (A_{ij})_{i,j=1}^K$, called the Rouse matrix, is symmetric and positive definite.  Furthermore, we associate with $A$ the linear mapping $\mathbb{A} : \mathbb{R}^{d \times K} \to \mathbb{R}^{d \times K}$ defined, for any $B = (B_i^j)_{i=1,\dots,d}^{j=1,\dots,K} \in \mathbb{R}^{d \times K}$, by $(\mathbb{A}(B))_i^j \coloneqq \sum_{k=1}^K B_i^k A_{kj}$, and let $\mathbb{A}^j: \mathbb{R}^{d \times K} \to \mathbb{R}^d$ be the linear mapping defined by $(\mathbb{A}^j(B))_i \coloneqq (\mathbb{A}(B))_i^j$, for $i = 1,\dots, d$ and $j = 1,\dots, K$. By the positive definiteness of the Rouse matrix $A \in \mathbb{R}^{K \times K}_{sym}$, there exist positive constants $C_1$ and $C_2$ such that 
\begin{equation}\label{eq1.14}
C_1 |B|^2 \leq \mathbb{A}(B) : B \leq C_2 |B|^2 \qquad \forall \, B \in \mathbb{R}^{d \times K}.
\end{equation}
In the subsequent discussion we shall simply take $\varepsilon = 1$ and $\lambda = 1/4$ since none of the results depend on the specific values of these positive parameters.

Before embarking on the proof of our main result we provide a brief literature survey. Unless otherwise stated, the centre-of-mass diffusion term is absent from the model in the cited reference, i.e., $\varepsilon = 0$; also, unless otherwise stated, the density $\rho$ is assumed to be constant, and in all cited references only a simple dumbbell model is considered rather than a bead-spring chain model, i.e., $K=1$. 

In \cite{MR1084958}, Renardy proved a local existence and uniqueness result for a family of Hookean-type dumbbell models. Subsequently, E, Li \& Zhang \cite{MR2073140} and Li, Zhang \& Zhang \cite{MR2059152} revisited the question of local existence of solutions for dumbbell models, while 
Zhang and Zhang \cite{MR2221211} showed a local existence and uniqueness result for regular solutions to FENE-type dumbbell models. All of these papers required high regularity of the initial data. In \cite{MR2188682}, Constantin considered the Navier--Stokes equations coupled to nonlinear Fokker--Planck equations modelling the probability distribution of particles interacting with the fluid, and in \cite{MR2600741} Constantin and Seregin  proved the global regularity of solutions of the incompressible Navier--Stokes--Fokker--Planck system in $\mathbb{R}^2$, in the absence of boundaries.

In \cite{MR2340887}, Lions and Masmoudi proved the global existence of weak solutions for the corotational FENE dumbbell model and the Doi model (also called the rod model) using a propagation-of-compactness argument, i.e., the property that if one takes a sequence of weak solutions, which converges weakly and such that the initial data converge strongly, then the weak limit is also a solution. In \cite{MR2456183}, Masmoudi explored the FENE dumbbell model for a general class of potentials; he proved local existence in Sobolev spaces, global existence if the initial data are close to equilibrium, and global existence in two dimensions for the corotational FENE model.

In \cite{MR2149930}, Barrett, Schwab \& S\"{u}li showed the existence of global-in-time weak solutions to the coupled microscopic-macroscopic bead-spring chain model with constant density $\rho$ and in the absence of the centre-of-mass diffusion term, i.e., with $\varepsilon = 0$. The paper admitted a large class of potentials $U$, including the Hookean dumbbell model and general FENE-type dumbbell models in the general noncorotational case, however the velocity field $\bm{v}$ in the drift-term of the Fokker--Planck equation and the extra stress tensor had to be mollified.

Subsequently, in \cite{MR2819196} and \cite{MR2902849} Barrett and S\"{u}li managed to prove the existence of global-in-time weak solutions to general noncorotational Hookean-type bead-spring chain models and FENE-type bead-spring chain models respectively, with centre-of-mass diffusion $\varepsilon>0$, but without mollification and in the general case of $K \geq 1$ coupled beads in the bead-spring chain. This was achieved by introducing a cut-off parameter $L$, discretizing the resulting model in time, and then passing to the limit as $L \to \infty$ by requiring that the time step $\Delta t = o(L^{-1})$. The papers also provided rigorous proofs, for both FENE-type and Hookean-type models, of the convergence of weak solutions to their respective equilibria: $\bm{v}_\infty =\mathbf{0}$ and $\psi_\infty=M$, as $t \rightarrow \infty$.
A key contribution to the field has been Masmoudi's paper \cite{MR3010381}, which proved global existence of weak solutions to the FENE dumbbell model, in the absence of a centre-of-mass diffusion term.  

All of the papers cited above were concerned with homogeneous fluids (i.e., fluids with constant density). In this paper, we study a model similar to the one in \cite{MR2981267}; there, the existence of global weak solutions to FENE-type bead-spring chain models with variable density $\rho$, density-dependent dynamic viscosity $\mu(\rho)$ and density-dependent drag coefficient $\zeta(\rho)$ was shown in a bounded domain in $\mathbb{R}^d$, $d=2,3$. Here, in contrast with \cite{MR2981267}, we permit dependence of the dynamic viscosity on both the density and the polymer number density, i.e., $\mu=\mu(\rho,\varrho)$; this simple extension of the model considered in \cite{MR2981267} introduces nontrivial technical difficulties. Thus, instead of using a sequence of approximating problems based on time-discretization as in \cite{MR2981267}, here, motivated by the approach in \cite{MR3046297}, we use a Galerkin method to construct a sequence of spatially semi-discrete approximations. This approach shortens and simplifies the proof of existence of weak solutions compared to the proof in  \cite{MR2981267}; and, by admitting the dependence of the dynamic viscosity on both the density and the polymer number density, it also generalises the results of \cite{MR2981267} to a wider and physically more realistic class of models. 

The paper is structured as follows. In the next section, we shall derive a formal energy identity, which is at the heart of our proof. We shall introduce the necessary notation and the relevant function spaces, and we formulate our assumptions on the data. In Section 3, we state the main result; the rest of the paper is devoted to the proof of our main result, which guarantees the existence of large-data global weak solutions to the model under consideration.

In Section 4, we begin our proof by first introducing a truncation of the extra-stress tensor with truncation parameter $\ell$ in Subsection 4.1. In order to preserve the energy identity, we also truncate the probability density function in the drag term in the Fokker--Planck equation; we also truncate the initial condition accordingly. In Subsection 4.2, we perform a spatial Galerkin semidiscretization of the velocity field and the probability density function with parameters $n$ and $m$. Given a sufficiently smooth velocity field, i.e., $\bm{v} \in L^1(0,T; W^{1,1}(\Omega;\mathbb{R}^d))$, the proofs of existence and uniqueness of the weak solution to the transport equation \eqref{eq1} can be found in \cite{MR2986590}, for example. By rewriting the truncated Navier--Stokes equation and the truncated Fokker--Planck equation in non-conservative form, we then use Schauder's fixed point theorem to prove the existence of solutions to our partially Galerkin discretized system in Subsection 4.3. In Subsection 4.4, we derive $n$-independent a priori estimates, which allow us to pass to the limit as $n \to \infty$. In Subsection 4.5, we first prove the boundedness of the sequence of approximate densities $\rho^m$ and the nonnegativity of the Galerkin approximations $\hat{\psi}^m$. Then we derive an $m$-independent a priori estimate. We use a Nikolski\u{\i} norm estimate and the Aubin--Lions Lemma to deduce strong convergence of $(\bm{v}^m)_{m=1}^\infty$. To deduce the strong convergence of $(\hat{\psi}^m)_{m=1}^\infty$, we first apply the Dunford--Pettis Theorem to deduce its weak convergence in $L^1_{loc}(\mathcal{O} \times (0,T))$. Then we use the Div-Curl Lemma to show that $(\hat{\psi}^m)_{m=1}^\infty$ converges almost everywhere in $\mathcal{O} \times (0,T)$, followed by Vitali's Convergence Theorem, which gives the strong convergence of the sequence $(\hat{\psi}^m )_{m=1}^\infty$ in $L^1(0,T; L^1(\mathcal{O}))$ as $m \to \infty$. Finally, in Subsection 4.6, we derive $\ell$-independent estimates and apply similar techniques as in Subsection 4.5 to pass to the limit as $\ell \to \infty$. We also show the weak attainment of the initial conditions.

\section{Energy identity, notational conventions, and assumptions on the data}

We begin by stating a formal energy identity for the system of nonlinear partial differential equations under consideration, which plays a key role in our proof of existence of solutions. We shall then also introduce the necessary notational conventions used throughout the paper, and, motivated by the structure of the formal energy identity, we shall state our assumptions on the data under which the existence result is subsequently proved. 

\subsection{The formal energy identity}
We shall derive a formal energy identity under the assumption that $\bm{v}$, $\rho$ and $\psi$ are sufficiently smooth, and, at least for our purposes in this introductory section, $\rho$ is nonnegative, and $\psi$ and $\zeta$ are positive. By taking the $L^2(\Omega)$ inner product of (\ref{eq1}) with $\frac{1}{2}|\bm{v}|^2$,
and taking the $L^2(\Omega; \mathbb{R}^d)$ inner product of (\ref{eq2}) with $\bm{v}$, we deduce upon partial integration and noting the homogeneous Dirichlet boundary condition on $\bm{v}$ and the divergence-free property of $\bm{v}$ and (\ref{eq7}) that
\begin{equation}\label{eq19}
\begin{split}
\frac{\mathrm{d}}{\,\mathrm{d}t} \int_\Omega \frac{1}{2} \rho |\bm{v}|^2 \,\mathrm{d}x + \int_\Omega \mu(\rho,\varrho) |D(\bm{v})|^2 \,\mathrm{d}x &= \int_\Omega \rho \bm{f} \cdot \bm{v} \,\mathrm{d}x - \int_\Omega \tau : \nabla_x \bm{v} \,\mathrm{d}x \\
&= \int_\Omega \rho \bm{f} \cdot \bm{v} \,\mathrm{d}x - k \sum_{j=1}^K \int_{\Omega \times D}\! \psi\, (U^j)^\prime \!\left( \frac{1}{2} |\bm{q}^j|^2 \right) \bm{q}^j \bm{q}^{j^{\rm T}}\! :\! \nabla_x \bm{v} \,\mathrm{d}\bm{q}\,\mathrm{d}x.
\end{split}
\end{equation}
Next, we multiply the Fokker--Planck equation (\ref{eq15}) by $\log \left(\frac{\psi}{\zeta(\rho)M}\right) + 1$, and integrate over $\Omega \times D$. Using integrations by parts and the boundary conditions (\ref{eq6}), (\ref{eq16}) and (\ref{eq17}) in the second and third term on the left-hand side and the two terms on the right-hand side we get
\begin{multline}\label{eq20}
\frac{\mathrm{d}}{\,\mathrm{d}t} \int_{\Omega \times D} M \zeta(\rho) \mathcal{F}\left( \frac{\psi}{\zeta(\rho)M} \right) \,\mathrm{d}\bm{q}\,\mathrm{d}x - \sum_{j=1}^K \int_{\Omega \times D} \psi (U^j)^\prime \left( \frac{1}{2} |\bm{q}^j|^2 \right) \bm{q}^j \bm{q}^{j^{\rm T}} : \nabla_x \bm{v} \,\mathrm{d}\bm{q}\,\mathrm{d}x \\
+ 4\int_{\Omega \times D} M \left| \nabla_x \sqrt{\frac{\psi}{\zeta(\rho) M}}\right|^2 \,\mathrm{d}\bm{q}\,\mathrm{d}x + 4\int_{\Omega \times D} M \mathbb{A}\left(\nabla_{\bm{q}} \sqrt{\frac{\psi}{\zeta(\rho)M}} \right) \colon \left(\nabla_{\bm{q}} \sqrt{\frac{\psi}{\zeta(\rho)M}} \right)\mathrm{d}\bm{q}\,\mathrm{d}x = 0,
\end{multline}
where $\mathcal{F}(s) = s\log s + 1$ for $s>0$ and $\mathcal{F}(0) \coloneqq \lim_{s \to 0+} \mathcal{F}(s) = 1$. We now multiply (\ref{eq20}) by $k$ and add the resulting identity to (\ref{eq19}) to deduce that
\begin{equation*}
\begin{split}
&\frac{\mathrm{d}}{\,\mathrm{d}t} \int_\Omega \left[  \frac{1}{2} \rho |\bm{v}|^2 + k\int_D M \zeta(\rho) \mathcal{F}\left( \frac{\psi}{\zeta(\rho)M} \right)\mathrm{d}\bm{q} \right] \,\mathrm{d}x +  \int_\Omega \mu(\rho,\varrho) |D(\bm{v})|^2 \,\mathrm{d}x \\
&\quad+ 4k \int_{\Omega \times D} M \left| \nabla_x \sqrt{\frac{\psi}{\zeta(\rho) M}}\right|^2 \,\mathrm{d}\bm{q}\,\mathrm{d}x + 4k \int_{\Omega \times D} M \mathbb{A}\left( \nabla_{\bm{q}} \sqrt{\frac{\psi}{\zeta(\rho)M}} \right) \colon \left( \nabla_{\bm{q}} \sqrt{\frac{\psi}{\zeta(\rho)M}} \right) \,\mathrm{d}\bm{q}\,\mathrm{d}x \\
&\qquad = \int_\Omega \rho \bm{f} \cdot \bm{v} \,\mathrm{d}x,
\end{split}
\end{equation*}
which is the formal energy identity that is essential to our proof of existence of weak solutions to the coupled Navier--Stokes--Fokker--Planck system under consideration.

\subsection{Preliminaries, notational conventions, and assumptions on the data}
First we shall summarise the definitions of Lebesgue spaces, Sobolev spaces and Bochner spaces. Let $O$ be a measurable set in $\mathbb{R}^d$ and $p \in [1,\infty)$. The standard Lebesgue space of $p$-integrable functions is denoted by $(L^p(O), \| \cdot \|_p)$. When $p = \infty$, $(L^\infty(O), \| \cdot \|_{\infty})$ denotes the space of essentially bounded functions. For $s \in \mathbb{N}$, let $(W^{s,p}(O), \| \cdot \|_{W^{s,p}(O)})$ be the standard Sobolev spaces and denote by $|\cdot|_{W^{s,p}(O)}$ the corresponding semi-norm. Since we shall need to work with Maxwellian-weighted spaces, we define, for a nonnegative weight-function $N \in L^\infty_{loc}(O)$,
\begin{align}
\nonumber L^p_N(O) &\coloneqq \{ f \in L^p_{loc}(O) : \int_O N(z) | f(z) |^p \,\mathrm{d}z < \infty \}, \\
\nonumber W^{1,p}_N(O) &\coloneqq \{ f \in W^{1,p}_{loc}(O) : \int_O N(z) (| f(z) |^p + |\nabla_z f(z)|^p) \,\mathrm{d}z < \infty \}.
\end{align}
For any pair of functions $f$, $g$, with $f \in L^p(O)$ and $g \in L^{p^\prime}(O)$, where $1/p + 1/p^\prime = 1$ and $p, p^\prime \in [1, \infty]$, we set
\begin{equation*}
(f,g)_O \coloneqq \int_O f(z) g(z) \, \mathrm{d}z.
\end{equation*}
Note that we set $1^\prime := \infty$ and $\infty^\prime := 1$. We adopt analogous notation for vector-valued and tensor-valued functions. If $O= \Omega$, we omit the subscript $\Omega$ from the inner product $(f,g)_\Omega$ for simplicity. For a general Banach space $(X, \| \cdot \|_X)$, the dual space consisting of all continuous linear functionals on $X$ is denoted by $X^\prime$ and the dual pairing is denoted by $\langle f, g \rangle_X$ if $f \in X^\prime$ and $g \in X$. \par
Let $\Omega \subset \mathbb{R}^d$ be a bounded open Lipschitz domain, with $d = 2,3$. $C(\overline{\Omega})$ denotes the set of all continuous real-valued functions on $\overline{\Omega}$. Let $C^\infty(\Omega)$ be the set of all smooth functions on $\Omega$ and let $C^\infty_0(\Omega)$ be the set of all functions in $C^\infty(\Omega)$ that are compactly supported in $\Omega$. Then we define the following function spaces:
\begin{align}
\nonumber W^{1,p}_{\bm{n}}(\Omega; \mathbb{R}^d) &\coloneqq \overline{ \{ \bm{v} \in C^\infty(\Omega; \mathbb{R}^d): \bm{v} \cdot \bm{n} = 0\ \text{on}\ \partial \Omega \} }^{\| \cdot \|_{W^{1,p}(\Omega)} }, \\
\nonumber W^{1,p}_{\bm{n}, \divergence}(\Omega; \mathbb{R}^d) &\coloneqq \overline{ \{ \bm{v} \in C^\infty(\Omega; \mathbb{R}^d): \bm{v} \cdot \bm{n} = 0\ \text{on}\ \partial \Omega, \divergence \bm{v} = 0\ \text{in}\ \Omega \} }^{\| \cdot \|_{W^{1,p}(\Omega)} }, \\
\nonumber W^{1,p}_{0, \divergence}(\Omega; \mathbb{R}^d) &\coloneqq \overline{ \{ \bm{v} \in C^\infty_0(\Omega; \mathbb{R}^d): \divergence \bm{v} = 0\ \text{in}\ \Omega \} }^{\| \cdot \|_{W^{1,p}(\Omega)} }, \\
\nonumber L^2_{0, \divergence}(\Omega; \mathbb{R}^d) &\coloneqq \overline{W^{1,2}_{\bm{n},\divergence}(\Omega;\mathbb{R}^d)}^{\| \cdot \|_2}.
\end{align}
\par
We denote by $L^p(0,T; X)$ the standard Bochner space of $p$-integrable $X$-valued functions defined on the interval $(0,T)$. If $X$ is separable and reflexive and $p \in (1, \infty)$, then $L^p(0,T; X)$ is separable and reflexive and $(L^p(0,T; X))^\prime = L^{p^\prime}(0,T; X^\prime)$. \par
For later purposes, we recall the well-known Gagliardo--Nirenberg inequality. Let $r \in [2, \infty)$ if $d = 2$, and $r \in [2, 6]$ if $d = 3$ and $\theta = d(\frac{1}{2} - \frac{1}{r})$. Then, there exists a constant $C = C(\Omega, r, d)$ such that, for all $f \in W^{1,2}(\Omega)$:
\begin{equation}\label{eq1.25}
\| f \|_{L^r(\Omega)} \leq C\| f \|^{1- \theta}_{L^2(\Omega)} \| f \|^\theta_{W^{1,2}(\Omega)}.
\end{equation}
We also recall the following version of Korn's inequality: for all $\bm{w} \in W^{1,2}_0(\Omega; \mathbb{R}^d)$, we have
\begin{equation}\label{eq1.26}
\int_\Omega |D(\bm{w})|^2 \, \mathrm{d}x \geq c_0 \| \bm{w} \|^2_{W^{1,2}(\Omega; \mathbb{R}^d)},
\end{equation}
where $c_0 > 0$ is a constant. \par
Next, we need to make a few assumptions. First we introduce the notation:
\begin{equation*}
\hat{\psi} = \frac{\psi}{\zeta(\rho)M}, \quad \hat{\psi_0} = \frac{\psi_0}{\zeta(\rho_0)M}.
\end{equation*}
We assume that $\partial \Omega \in C^{0,1}$. For the Maxwellian $M$ we assume that
\begin{equation}\label{eq22}
M \in C_0(\overline{D}) \cap C^{0,1}_{loc}(D) \cap W^{1,1}_0(D), \quad M \geq 0, \quad M^{-1} \in C_{loc}(D).
\end{equation}
For spring potentials that satisfy \eqref{eq10} and \eqref{eq11} (which is the case for typical FENE-type spring potentials) the properties listed in \eqref{eq22} are automatically hold, because \eqref{eq10} and \eqref{eq11} imply \eqref{eq22}. 

For the initial density $\rho_0$ we assume that
\begin{equation}\label{eq23}
\rho_0 \in [\rho_{\min}, \rho_{\max}], \ \text{with $\rho_{\min}> 0$}.
\end{equation}
For the initial velocity $v_0$ we assume that
\begin{equation}\label{eq-inivel}
\bm{v}_0 \in L^2_{0,\divergence}(\Omega; \mathbb{R}^d).
\end{equation}
For the initial value $\psi_0$ of the probability density function we assume that
\begin{equation}\label{eq25}
\begin{split}
&\psi_0 \geq 0 \quad \text{a.e. on $\Omega \times D$},\quad \mathcal{F}(\hat{\psi}_0) \in L^1_M(\Omega \times D),
\\
0 \leq \varrho_0(x):=& \,\int_D \psi_0(\cdot,\bm{q}) \,\mathrm{d}\bm{q} \leq \varrho_{\max}\quad \text{a.e. on $\Omega$},  \quad \int_{\Omega \times D} \psi_0(x,\bm{q}) \,\mathrm{d}\bm{q}\,\mathrm{d}x = 1.
\end{split}
\end{equation}
We shall further assume that
\begin{equation}\label{eq26}
\mu \in C^1([\rho_{\min}, \rho_{\max}] \times [0,\infty), [\mu_{\min}, \mu_{\max}]), \quad \zeta \in C^1([\rho_{\min}, \rho_{\max}], [\zeta_{\min}, \zeta_{\max}]), \ \text{with $\mu_{\min}, \zeta_{\min}>0$}.
\end{equation}
Finally, we assume that
\begin{equation}\label{finL2}
\bm{f} \in L^2(0,T; L^2(\Omega;\mathbb{R}^d)).
\end{equation}

\section{The main result}
In this subsection we state our main result,  which we shall prove in the next sections.
\begin{thm}\label{thm1}
Let $\Omega \subset \mathbb{R}^d$, $d \in \{ 2,3 \}$, be a bounded open Lipschitz domain. Let $K \in \mathbb{N}$ be arbitrary and let $D^i \subset \mathbb{R}^d$, $i = 1, \dots, K$, be bounded open balls centred at the origin. Suppose that $\bm{f} \in L^2(0,T; L^2(\Omega; \mathbb{R}^d))$. Assume that the map $\mathbb{A}: B \in \mathbb{R}^{d \times K} \mapsto \mathbb{A}(B) \in \mathbb{R}^{d \times K}$ is linear and satisfies (\ref{eq1.14}), the Maxwellian $M : D \to \mathbb{R}$ satisfies (\ref{eq22}), $\mu(\cdot,\cdot)$ and $\zeta(\cdot)$ satisfy (\ref{eq26}), and the initial data $(\rho_0, \bm{v}_0, \psi_0)$ satisfy (\ref{eq23})--(\ref{eq25}). Then, there exist functions $(\rho, \bm{v}, \tau, \psi)$, 
such that $\psi(x,\bm{q},t) =  \zeta(\rho(x,t)) M(\bm{q})  \hat\psi(x,\bm{q},t) $, with
\begin{align*}
\rho &\in L^\infty(\Omega \times (0,T)) \cap C([0,T]; L^p(\Omega)), \ \text{where $p \in [1,\infty)$}, \\
\bm{v} &\in L^\infty(0,T; L^2_{0,\divergence}(\Omega; \mathbb{R}^d)) \cap L^2(0,T; W^{1,2}_{0,\divergence}(\Omega; \mathbb{R}^d)), \\
\tau &\in L^2(0,T; L^2(\Omega; \mathbb{R}^{d \times d})), \\
\hat{\psi} &\in L^\infty(\Omega \times (0,T); L^1_M(D)) \cap L^2(0,T; W^{1,1}_M(\mathcal{O})), \ \hat{\psi} \geq 0 \ \text{a.e. in $\mathcal{O} \times (0,T)$}, 
\end{align*}
where $\mathcal{O}:=\Omega \times D$, and where the triple $(\rho, \bm{v}, \hat{\psi})$ satisfies the following coupled system of nonlinear partial differential equations in weak form: 
\begin{equation}\label{eqdensity}
\int_0^T [ \langle \partial_t \rho, \eta \rangle - (\bm{v}\rho, \nabla_x \eta)] \,\mathrm{d}t = 0,\quad \text{for all $\eta \in L^1(0,T; W^{1, \frac{q}{q-1}}(\Omega))$}, 
\end{equation}
where $q \in (2,\infty)$ when $d=2$ and $q \in [3,6]$ when $d=3$,
\begin{align}\label{eqNS}
\begin{aligned}
&\int_0^T \langle \partial_t (\rho \bm{v}), \bm{w} \rangle \,\mathrm{d}t + \int_0^T [-(\rho \bm{v} \otimes \bm{v}, \nabla_x \bm{w}) + (\mu(\rho,\varrho)D(\bm{v}), \nabla_x \bm{w})] \,\mathrm{d}t \\
&= \int_0^T [-(\tau, \nabla_x \bm{w}) +  (\rho\bm{f}, \bm{w})] \,\mathrm{d}t, \quad \text{for all $\bm{w} \in L^s(0,T; W^{1,s}_{0,\divergence}(\Omega; \mathbb{R}^d))$ \,\,with $s>2$}, 
\end{aligned}
\end{align}
and
\begin{equation}\label{eqFP}
\begin{split}
&\int_0^T \left\langle \partial_t(M \zeta(\rho)\hat{\psi}), \varphi \right\rangle_{\mathcal{O}} -  \left(M \zeta(\rho) \bm{v} \hat{\psi}, \nabla_x \varphi \right)_{\mathcal{O}} -  \left(M \zeta(\rho) \hat{\psi}(\nabla_x \bm{v}) \bm{q}, \nabla_{\bm{q}} \varphi \right)_{\mathcal{O}} \,\mathrm{d}t \\
&\quad  + \int_0^T (M \nabla_x \hat{\psi}, \nabla_x \varphi)_{\mathcal{O}} + \left( M \mathbb{A}(\nabla_{\bm{q}} \hat{\psi}), \nabla_{\bm{q}} \varphi \right)_{\mathcal{O}} \,\mathrm{d}t = 0, \quad  \text{for all $\varphi \in L^\infty(0,T; W^{1,\infty}(\mathcal{O}))$}.
\end{split}
\end{equation}
The polymer number density is given by
\begin{equation}\label{eqPND}
\varrho(x,t) = \zeta(\rho) \int_D M(\bm{q})\hat\psi(x,\bm{q},t)\,\mathrm{d}\bm{q} \quad \text{for a.e. $(x,t) \in \Omega \times (0,T)$}
\end{equation}
and the extra-stress tensor $\tau$ is given by
\begin{equation}\label{eqtau}
\tau(x,t) = k \sum_{j=1}^K \int_D M\zeta(\rho) \nabla_{\bm{q}^j} \hat{\psi}(x,\bm{q},t) \otimes \bm{q}^j \,\mathrm{d}\bm{q} \quad \text{for a.e. $(x,t) \in \Omega \times (0,T)$}.
\end{equation}
The following weak continuity properties hold:
\begin{align}\label{weakccont}
\begin{aligned} 
t &\mapsto \int_\Omega \rho(x,t)\bm{v}(x,t)\cdot \bm{u}\,  \mathrm{d}x \in C([0,T])\quad\mbox{for any $\bm{u} \in W^{1,s}_{0,\divergence}(\Omega;\mathbb{R}^d)$},\\
t &\mapsto \int_{\mathcal{O}}M(\bm{q}) \zeta(\rho(x,t))(t)\hat{\psi}(x,\bm{q},t)\,\phi(x,\bm{q})\,\mathrm{d}x\,\mathrm{d}\bm{q}\in C([0,T]) \quad \mbox{for any $\phi \in W^{1,\infty}(\mathcal{O})$},
\end{aligned}
\end{align}
and the initial data are attained in the following sense: 
\begin{align}\label{eqinicond}
\begin{aligned}
\lim_{t \to 0+} ((\rho\bm{v})(t), \bm{u}) &= (\rho_0 \bm{v}_0, \bm{u}) \qquad \qquad \quad \text{for all $\bm{u} \in W^{1,s}_{0,\divergence}(\Omega; \mathbb{R}^d)$ where $s>2$}, \\
\lim_{t \to 0+} ( M (\zeta(\rho)\hat{\psi})(t), \phi)_{\mathcal{O}} &= (M\zeta(\rho_0) \hat{\psi}_0, \phi )_{\mathcal{O}} \qquad \text{for all $\phi \in W^{1,\infty}(\mathcal{O})$}.
\end{aligned}
\end{align}
Moreover, for a.e. $t \in (0,T)$, the following energy inequality holds:
\begin{equation}\label{energy}
\begin{split}
&k\int_{\mathcal{O}} M \zeta(\rho(\cdot, t)) \mathcal{F}(\hat{\psi}(\cdot,t)) \,\mathrm{d}x\,\mathrm{d}\bm{q} + \frac{1}{2}\int_\Omega \rho(\cdot,t) |\bm{v}(\cdot,t)|^2 \,\mathrm{d}x \\
&\quad + \int_0^t \int_\Omega \mu(\rho,\varrho) |D(\bm{v})|^2 \,\mathrm{d}x\,\mathrm{d}s + 4kC_1\int_0^t \int_{\mathcal{O}} M \left| \nabla_{x,\bm{q}} \sqrt{\hat{\psi}} \right|^2 \,\mathrm{d}x\,\mathrm{d}\bm{q}\,\mathrm{d}s \\
&\leq k\int_{\mathcal{O}} M \zeta(\rho_0) \mathcal{F}(\hat{\psi}_0) \,\mathrm{d}x\,\mathrm{d}\bm{q} + \frac{1}{2}\int_\Omega \rho_0 |\bm{v}_0|^2 \,\mathrm{d}x + \int_0^t (\rho \bm{f}, \bm{v}) \,\mathrm{d}s,
\end{split}
\end{equation}
where $\mathcal{F}(s) = s\log s + 1$ for $s>0$ and $\mathcal{F}(0) \coloneqq \lim_{s \to 0+} \mathcal{F}(s) = 1$.
\end{thm}
We will prove this result by constructing a sequence of approximations to the problem under consideration. We shall
 then pass to the limits in the various approximation paramaters --- first in the dimensions $n$ and $m$ of the Galerkin subspaces, and then in the parameter $\ell$ in the truncation process that we shall next introduce, to deduce the convergence of the sequence of approximations to a global-in-time weak solution of the problem.

\section{Proof of existence}
\subsection{The first level of approximation: truncation}
To approximate our original Navier--Stokes--Fokker--Planck system, we begin by introducing a smooth nonnegative function $\Gamma \in C_0^\infty((-2,2))$, such that $\Gamma(s) = 1$ for all $s \in [-1,1]$ and for an arbitrary $\ell \in \mathbb{N}$  we define $\Gamma_\ell(s) \coloneqq \Gamma(\frac{s}{\ell})$. The primitive function to $\Gamma_\ell$ is the truncation function defined by
\begin{equation*}
T_\ell(s) \coloneqq \int_0^s \Gamma_\ell(r)\, \mathrm{d}r.
\end{equation*}
We define the $\ell$-th approximation of  (\ref{eq1}) as follows:
\begin{equation}\label{eqrhol}
\frac{\partial \rho^{\ell}}{\partial t} + \nabla_x \cdot (\bm{v}^{\ell} \rho^{\ell}) = 0 \qquad  \text{in $Q$},
\end{equation}
subject to the following initial condition:
\begin{equation}\label{eqrholini}
\rho^{\ell}(\cdot, 0) = \rho_0(\cdot) \qquad \text{in $\Omega$}.
\end{equation}
We define the $\ell$-th approximation of (\ref{eq2}) and (\ref{eq3}) as follows:
\begin{align}
\begin{aligned}
\label{eq29} \frac{\partial (\rho^{\ell} \bm{v}^{\ell})}{\partial t} +  \nabla_x \cdot (\rho^{\ell} \bm{v}^{\ell} \otimes \bm{v}^{\ell}) - \nabla_x \cdot (\mu(\rho^{\ell},\varrho^{\ell})D(\bm{v}^{\ell})) + \nabla_x p^{\ell} 
&= \rho^{\ell} \bm{f} + \nabla_x \cdot \tau^{\ell} \qquad \quad  \text{in $Q$}, \\
\nabla_x \cdot \bm{v}^{\ell} &= 0 \qquad \qquad \qquad \quad \quad \  \text{in $Q$},
\end{aligned}
\end{align}
with initial and boundary conditions given by
\begin{align}\label{eqvinibd}
\begin{aligned}
\bm{v}^{\ell}(\cdot, 0) & = \bm{v}_0(\cdot) \qquad \qquad \text{in $\Omega$}, \\
\bm{v}^{\ell} &= \bm{0} \qquad \qquad \quad \, \,  \text{on $\partial \Omega \times (0,T)$}.
\end{aligned}
\end{align}
The $\ell$-th approximation $\varrho^\ell$ of the polymer number density $\varrho$ is defined by 
\begin{align}\label{pnd}
 \varrho^\ell(x,t):= \zeta(\rho^\ell(x,t)) \int_D M(\bm{q}) \hat\psi^\ell(x,\bm{q},t)\,\mathrm{d}\bm{q},
\end{align}
and the $\ell$-th approximation $\tau^\ell$ of the polymeric extra stress tensor $\tau$ is given by
\begin{equation}\label{taul}
\tau^{\ell}(x,t) \coloneqq k \sum_{j=1}^K \int_D M(\bm{q}) \zeta(\rho^{\ell}(x,t)) \nabla_{\bm{q}^j} T_\ell(\hat{\psi}^\ell(x,\bm{q},t)) \otimes \bm{q}^j \,\mathrm{d}\bm{q},
\end{equation}
where $\hat\psi^\ell$ is the solution of the initial-boundary-value problem \eqref{FPL}--\eqref{FPLIC} stated below.

By partial integration we have that
\begin{equation}\label{taulpi}
\tau^{\ell}(x,t) = -k \int_D \left[ KM(\bm{q})\zeta(\rho^{\ell})T_\ell(\hat{\psi}^\ell(x, \bm{q},t))I + \sum_{j=1}^K \zeta(\rho^{\ell}) T_\ell(\hat{\psi}^\ell(x, \bm{q},t))\nabla_{\bm{q}^j} M(\bm{q}) \otimes \bm{q}^j \right] \mathrm{d}\bm{q}
\end{equation}
on observing that the boundary term on $\partial D$ vanishes since $M = 0$ on $\partial D$. We shall also modify the Fokker--Planck equation (\ref{eq15}). We first set 
\begin{equation*}
\Lambda_\ell(s) \coloneqq s\Gamma_\ell(s).
\end{equation*}
The $\ell$-th approximation of (\ref{eq15}) is given by
\begin{align}\label{FPL}
\begin{aligned}
\frac{\partial (M\zeta(\rho^{\ell}) \hat{\psi}^{\ell})}{\partial t} 
+ \nabla_x \cdot (M \zeta(\rho^{\ell}) \hat{\psi}^{\ell} \bm{v}^{\ell})  &+  \divergence_{\bm{q}} (M \zeta(\rho^{\ell}) \Lambda_\ell(\hat{\psi}^{\ell})(\nabla_x \bm{v}^{\ell}) \bm{q}) \\
&- \Delta_x (M\hat{\psi}^{\ell}) - \divergence_{\bm{q}} \mathbb{A}\left(M \nabla_{\bm{q}} \hat{\psi}^{\ell} \right) = 0
\end{aligned}
\end{align}
on $\mathcal{O} \times (0,T)$, where $\mathcal{O} \coloneqq \Omega \times D$, and is supplemented by the following boundary conditions:
\begin{align}
\left[ \mathbb{A}^j(M \nabla_{\bm{q}} \hat{\psi}^{\ell}) - M \zeta(\rho^{\ell}) \Lambda_\ell( \hat{\psi}^{\ell}) (\nabla_x \bm{v}^{\ell})\bm{q}^j \right] \cdot \bm{n}^j &= 0 \qquad \qquad \quad \, \text{on $\Omega \times \partial \bar{D}^j \times (0,T)$}, \label{FPLBC1}\\
M\nabla_x \hat{\psi}^{\ell} \cdot \bm{n} &=0 \qquad \qquad \quad \, \text{on $\partial \Omega \times D \times (0,T)$},\label{FPLBC2}
\end{align}
for all $j=1,\dots, K$. We also truncate the initial condition for $\hat{\psi}^{\ell}$ as follows:
\begin{equation}\label{FPLIC}
\hat{\psi}^{\ell}(x, \bm{q},0) = T_\ell(\hat{\psi}_0(x,\bm{q}))\qquad \mbox{for $(x,\bm{q}) \in \Omega \times D$}.
\end{equation}

For simplicity, we shall omit the superscript $\ell$ temporarily in the following discussions; we shall reinstate it later and will then pass to the limit $\ell \rightarrow \infty$. We need to show first, however, that this approximating problem has a solution $(\rho^\ell, \bm{v}^\ell, \hat\psi^\ell, \varrho^\ell)$ for each $\ell \geq 1$; we shall do so by constructing a two-level Galerkin approximation to it and passing to the limits in the sequences of Galerkin approximations.

\subsection{A two-stage Galerkin approximation}
First, we define an approximate Maxwellian $M^m$ by fixing a sequence of positive functions $(\overline{M}^m)_{m \in \mathbb{N}} \subset C^{0,1}_0(\overline{D})$ such that for each compact set $\kappa \subset D$ the following  holds:
\begin{equation}\label{eq3.14}
\lim_{m \to \infty} \| \overline{M}^m - M \|_{C(\overline{D}) \cap W^{1,1}_0(D)} + \| (\overline{M}^m)^{-1} - M^{-1} \|_{C(\kappa)} = 0.
\end{equation}
Then, the approximate Maxwellian $M^m$ is defined by
\begin{equation*}
M^m \coloneqq \overline{M}^m + \frac{1}{m}, \quad \text{for $m = 1, 2, \dots$}.
\end{equation*}
Let $(\bm{f}^m)_{m = 1}^\infty$ be a sequence of functions in $C([0,T]; L^2(\Omega; \mathbb{R}^d))$ converging to $\bm{f}$ in $L^2(0,T; L^2(\Omega; \mathbb{R}^d))$. Let further $(\rho^m_0)_{m=1}^\infty$ be a sequence of functions in $C^1(\overline\Omega)$ such that $\rho^m_0 \in [\rho_{\min}, \rho_{\max}]$, with $\rho_{\min} >0$, which converges to $\rho_0$ in $L^1(\Omega)$; such a sequence can be constructed by extending $\rho_0$ from $\Omega$ to $\mathbb{R}^d$ by $0$ and convolving the resulting function, still denoted by $\rho_0$,  with $\theta^m$, where $\theta^m(x):= m^d \theta(mx)$, $\theta \in C^\infty_0(\mathbb{R}^d)$, $\theta \geq 0$, and $\int_{\mathbb{R}^d} \theta(x)\,\mathrm{d}x = 1$, and  observing that 
\[ \rho^m_0(x) - \rho_{\min} = \int_{\mathbb{R}^d}(\rho_0(x-y) - \rho_{\min})\, \theta^m(y)\,\mathrm{d}y \geq 0,\]
and
\[ \rho^m_0(x) - \rho_{\max} = \int_{\mathbb{R}^d}(\rho_0(x-y) - \rho_{\max})\, \theta^m(y)\,\mathrm{d}y \leq 0.\]

Next, we shall introduce the Galerkin basis functions for the velocity field and the probability density function respectively. The Hilbert space $W^{1,2}_{0,\divergence} \cap W^{d+1,2}(\Omega; \mathbb{R}^d)$, equipped with the inner product of $W^{d+1,2}(\Omega; \mathbb{R}^d)$, is compactly and densely embedded in the Hilbert space $L^2_{0, \divergence}(\Omega; \mathbb{R}^d)$. Hence, by the Hilbert--Schmidt Theorem, there exists a sequence $(\bm{w}_i)_{i=1}^\infty$ of eigenfunctions in $W^{1,2}_{0,\divergence} \cap W^{d+1,2}(\Omega; \mathbb{R}^d)$ whose linear span is dense in $L^2_{0, \divergence}(\Omega; \mathbb{R}^d)$, such that the $\bm{w}_i$, $i = 1,2,\dots$, are orthogonal in the inner product of $W^{d+1,2}(\Omega; \mathbb{R}^d)$ and orthonormal in the inner product of $L^2(\Omega; \mathbb{R}^d)$. By Sobolev embedding, it follows that $\bm{w}_i \in C^1(\overline\Omega;\mathbb{R}^d)$ for all $i=1,2, \ldots$. Similarly, for each $m \in \mathbb{N}$ we find a sequence $(\varphi_i^m)_{i=1}^\infty$ of eigenfunctions in $W^{(K+1)d+1,2}(\mathcal{O})$ that are orthogonal in $W^{1, 2}_{M^m}(\mathcal{O})$ and orthonormal in $L^2_{M^m}(\mathcal{O})$. As $\mathcal{O} = \Omega \times D \subset \mathbb{R}^{(K+1)d}$, by Sobolev embedding we deduce that $\varphi_i^m \in C^1(\mathcal{O})$.

Finally, we fix $m,n \in \mathbb{N}$ and look for $(\rho^{m,n}, \bm{v}^{m,n}, \hat{\psi}^{m,n})$, where $\bm{v}^{m,n}$ and $\hat{\psi}^{m,n}$ are of the form
\begin{align}\label{eq-vmn}
\bm{v}^{m,n}(x,t) &\coloneqq \sum_{i=1}^m c^{m,n}_i(t) \bm{w}_i(x),\\\label{eq-psimn}
\hat{\psi}^{m,n}(x,\bm{q}, t) &\coloneqq \sum_{i=1}^n d^{m,n}_i(t) \varphi^m_i(x,\bm{q}),
\end{align}
that solve
\begin{align}
\label{eq48} &(\partial_t \rho^{m,n}, \eta) - (\bm{v}^{m,n}\rho^{m,n}, \nabla_x \eta) = 0\quad \text{for all $\eta \in C^{0,1}(\overline\Omega)$ and a.e. $t \in (0,T)$},\\
\begin{split}
\label{eq47} &( \partial_t (\rho^{m,n} \bm{v}^{m,n}), \bm{w}_i) - (\rho^{m,n} \bm{v}^{m,n} \otimes \bm{v}^{m,n}, \nabla_x \bm{w}_i) + (\mu(\rho^{m,n},\varrho^{m,n})D(\bm{v}^{m,n}), \nabla_x \bm{w}_i)\\
&\quad = -(\tau^{m,n}, \nabla_x \bm{w}_i) + (\rho^{m,n}\bm{f}^m, \bm{w}_i) \quad \text{for all $i = 1,\dots,m$ and a.e. $t \in (0,T)$}, 
\end{split}
\\
\begin{split}
\label{eq49} &\left( \partial_t(M^m \zeta(\rho^{m,n})\hat{\psi}^{m,n}), \varphi^m_i \right)_{\mathcal{O}} 
-  \left(M^m \zeta(\rho^{m,n}) \bm{v}^{m,n} \hat{\psi}^{m,n}, \nabla_x \varphi^m_i\right)_{\mathcal{O}} \\
&\quad - \left(M \zeta(\rho^{m,n}) \Lambda_\ell(\hat{\psi}^{m,n})(\nabla_x \bm{v}^{m,n}) \bm{q}, \nabla_{\bm{q}} \varphi^m_i \right)_{\mathcal{O}} + (M^m \nabla_x \hat{\psi}^{m,n}, \nabla_x \varphi^m_i)_{\mathcal{O}} \\
&\quad + \left( M^m \mathbb{A}(\nabla_{\bm{q}} \hat{\psi}^{m,n}), \nabla_{\bm{q}} \varphi^m_i\right)_{\mathcal{O}} = 0 \quad  \text{for all $i=1,\dots,n$ and a.e. $t \in (0,T)$}.
\end{split}
\end{align}
Here $\varrho^{m,n}$ is defined by
\begin{align}
\label{eq49a}
&\varrho^{m,n}(x,t):=\zeta(\rho^{m,n}(x,t)) \int_D M^m(\bm{q})\,[\hat\psi^{m,n}(x,\bm{q},t)]_{+} \,\mathrm{d}\bm{q} \quad \mbox{for a.e. $(x,t) \in Q$},
\end{align}
where, for a real number $s$, $[s]_+:=\max(0,s)$, and the expression $\tau^{m,n}$ is defined as follows:
\begin{equation}\label{eq49.1}
\tau^{m,n} \coloneqq  -k \int_D \left[ K M \zeta(\rho^{m,n})T_\ell(\hat{\psi}^{m,n})I + \sum_{j=1}^K \zeta(\rho^{m,n}) T_\ell(\hat{\psi}^{m,n})\nabla_{\bm{q}^j} M \otimes \bm{q}^j \right] \,\mathrm{d}\bm{q} \quad \text{a.e. in $Q$}.
\end{equation}
The initial data are given by
\begin{alignat}{2}\label{eqlinicond}
\begin{aligned}
\rho^{m,n}(x,0) &= \rho_0^m(x) \qquad &&\mbox{on $\Omega$}, \\
\bm{v}^{m,n}(x,0) &= \bm{v}^m_0(x) \coloneqq \sum_{i=1}^m (\bm{v}_0, \bm{w}_i) \bm{w}_i(x) \qquad &&\mbox{on $\Omega$}, \\
\hat{\psi}^{m,n}(x,\bm{q},0) &= \hat{\psi}^{m,n}_0(x,\bm{q}) \coloneqq \sum_{i=1}^n (T_\ell(\hat{\psi}^m_0), \varphi^m_i)_{\mathcal{O}}\, \varphi^m_i(x, \bm{q}) \qquad &&\mbox{on $\Omega\times D$}, 
\end{aligned}
\end{alignat}
where
\begin{equation}\label{eq1.66}
\hat{\psi}^m_0 \coloneqq \hat{\psi}_0 \frac{M}{M^m}.
\end{equation}
Note that in the third term on the left-hand side of (\ref{eq49}) and in (\ref{eq49.1}) the Maxwellian $M$ has, intentionally, not been replaced by the approximate Maxwellian $M^m$. Note also that we do not perform Galerkin discretizations of the density $\rho$ and of the polymer number density $\varrho$ in the above system (\ref{eq48})--(\ref{eq1.66}), so at this point we have no guarantee that solutions to this system exist. Our aim in the next section is therefore to show that solutions to this partially Galerkin-discretized system do in fact exist. Having done so, we shall pass to the limit $n \rightarrow \infty$, then to the limit $m \rightarrow \infty$, and finally we shall let $\ell \rightarrow \infty$. 

\subsection{Existence of solutions to the partially Galerkin-discretized system}
In this subsection, we will first show that solutions $(\rho^{m,n}, \bm{v}^{m,n}, \hat{\psi}^{m,n})$ exist for the system (\ref{eq48})--(\ref{eq1.66}) using Schauder's fixed-point theorem.  

For any integers $m,n \geq 1$, let $V^m =\spn \{\bm{w}_1, \ldots, \bm{w}_m\}$ and $X^n =\spn \{\varphi^m_1, \ldots, \varphi^m_n \}$ be the finite-dimensional Galerkin approximation spaces under consideration. 
Let $\bm{u}^{m,n} \in C([0,T]; V^m)$ and $\xi^{m,n} \in C([0,T]; X^n)$ be given. First let us consider the following transport problem:
\begin{equation}\label{eq56}
\frac{\partial \rho^{m,n}}{\partial t} + \divergence_x(\rho^{m,n} \bm{u}^{m,n}) = 0,
\end{equation}
subject to the initial condition
\begin{equation}\label{eq56a}
\rho^{m,n}(0) = \rho_{0}^m.
\end{equation}
Since $\bm{u}^{m,n} \in L^1(0,T; W^{1,1}(\Omega;\mathbb{R}^d))$ and $\divergence_x \bm{u}^{m,n} = 0$, one can show, by following the arguments in Chapter VI in \cite{MR2986590}, that there exists a unique solution $\rho^{m,n}$ to the transport problem which satisfies
\begin{equation}\label{eq58}
0<\rho_{\min} \leq \rho^{m,n}  \leq \rho_{\max}
\end{equation}
and, in addition, $\rho^{m,n} \in C([0,T]; L^p(\Omega))$ for $1 \leq p < \infty$. 
We also define
\[ \lambda^{m,n}(x,t):= \zeta(\rho^{m,n}(x,t)) \int_D M^m(\bm{q})\,[\xi^{m,n}(x,\bm{q},t)]_+\, \mathrm{d}\bm{q}.\]

Now that we have built $\rho^{m,n}$ and $\lambda^{m,n}$, we seek $\bm{v}^{m,n} \in C^1([0,T]; V^m)$ and $\hat{\psi}^{m,n} \in C^1([0,T]; X^n)$ satisfying
\begin{align}\label{eq60}
\begin{split}
&\int_\Omega \rho^{m,n} \left( \frac{\partial \bm{v}^{m,n}}{\partial t} + (\bm{u}^{m,n} \cdot \nabla_x)\bm{v}^{m,n} \right) \cdot \bm{w} \,\mathrm{d}x + \int_\Omega \mu(\rho^{m,n},\lambda^{m,n}) D(\bm{v}^{m,n}) : D(\bm{w})\,\mathrm{d}x \\
&\quad= \int_\Omega \rho^{m,n} \bm{f}^m \cdot \bm{w} \,\mathrm{d}x - \int_\Omega \tau^{m,n} : \nabla_x \bm{w} \,\mathrm{d}x,
\end{split}
\\
\label{eq84}
\begin{split}
&\int_{\mathcal{O}} M^m \zeta(\rho^{m,n}) \frac{\partial \hat{\psi}^{m,n}}{\partial t} \varphi \,\mathrm{d}x\,\mathrm{d}\bm{q} + \int_{\mathcal{O}} M^m \zeta(\rho^{m,n}) (\nabla_x \hat{\psi}^{m,n} \cdot \bm{u}^{m,n}) \varphi \,\mathrm{d}x\,\mathrm{d}\bm{q} + \int_{\mathcal{O}} M^m \nabla_x \hat{\psi}^{m,n} \cdot \nabla_x \varphi \,\mathrm{d}x\,\mathrm{d}\bm{q} \\
&\qquad + \int_{\mathcal{O}} M^m \mathbb{A}(\nabla_{\bm{q}} \hat{\psi}^{m,n}) : \nabla_{\bm{q}} \varphi \,\mathrm{d}x\,\mathrm{d}\bm{q} = \int_{\mathcal{O}} M \zeta(\rho^{m,n}) \Lambda_\ell(\xi^{m,n}) (\nabla_x \bm{u}^{m,n}) \bm{q} : \nabla_{\bm{q}} \varphi \,\mathrm{d}x\,\mathrm{d}\bm{q}
\end{split}
\end{align}
for any $\bm{w} \in V^m$ and $\varphi \in X^n$, and $\tau^{m,n}$ is given by
\begin{equation}\label{eq59}
\tau^{m,n} =  -k \int_D \left[ K M \zeta(\rho^{m,n})T_\ell(\xi^{m,n})I + \sum_{j=1}^K \zeta(\rho^{m,n}) T_\ell(\xi^{m,n})\nabla_{\bm{q}^j} M \otimes \bm{q}^j \right] \,\mathrm{d}\bm{q}.
\end{equation}
First we note that since by hypothesis $\zeta$ is a continuous function of its argument and $\rho^{m,n}$ is bounded, $\zeta(\rho^{m,n})$ is bounded above. Thanks to the presence of the truncation function $T_\ell$ we deduce that
\begin{equation}\label{eq62}
|\tau^{m,n}| \leq C(\ell, M, \zeta_{\max}).
\end{equation}
Also note that since we have replaced the unknown advection vector field $\bm{v}^{m,n}$ in the convective term by the known vector field $\bm{u}^{m,n}$, (\ref{eq60}), (\ref{eq84}) is now a system of linear ordinary differential equations. 
\begin{clm}\label{clm2.1}
Problem (\ref{eq60}) has a unique solution $\bm{v}^{m,n} \in C^1([0,T]; V^m)$ subject to the initial condition 
\begin{equation*}
\bm{v}^{m,n}(\cdot, 0) = \bm{v}_{0,m} \coloneqq \sum_{i=1}^m (\bm{v}_0, \bm{w}_i) \bm{w}_i.
\end{equation*}
\end{clm}
\begin{proof}
We shall seek the solution in the form $\bm{v}^{m,n}= \sum_{i=1}^m \alpha^{m,n}_i(t)\bm{w}_i$. Now (\ref{eq60}) can be re-written as 
\begin{equation}\label{eq63}
M(t)\frac{\mathrm{d} \bm{\alpha}}{\,\mathrm{d}t}(t) = A(t) \bm{\alpha}(t) + B(t),
\end{equation}
where $\bm{\alpha}(t) = (\alpha^{m,n}_1(t),\dots, \alpha^{m,n}_m(t))^{\rm T} \in \mathbb{R}^m$ is the unknown vector. In the above equation
\begin{align}\nonumber
(M(t))_{ij} &:= \int_\Omega \rho^{m,n} \bm{w}_i \cdot \bm{w}_j \,\mathrm{d}x, \\\nonumber
(A(t))_{ij} &:= - \left(\int_\Omega \rho^{m,n} (\bm{u}^{m,n} \cdot \nabla_x )\bm{w}_i \cdot \bm{w}_j \,\mathrm{d}x + \int_\Omega \mu(\rho^{m,n},\lambda^{m,n}) D(\bm{w}_i) : D(\bm{w}_j) \,\mathrm{d}x \right), \\\nonumber
(B(t))_j &:= \int_\Omega \rho^{m,n} \bm{f}^m \cdot \bm{w}_j \,\mathrm{d}x - \int_\Omega \tau^{m,n} : \nabla_x \bm{w}_j \,\mathrm{d}x,
\end{align}
where $i,j \in \{1,\dots, m \}$ and $\tau^{m,n}$ is a function of $\xi^{m,n}$, as given in (\ref{eq59}). Note that $M(t)$, $A(t)$ and $B(t)$ are continuous with respect to $t$. Since $M(t)$ is the Gram matrix of the basis $\{ \bm{w}_i \}$ with respect to the inner product
\begin{equation*}
(f,g) \mapsto \langle f,g \rangle_{\rho^{m,n}} \coloneqq \int_\Omega \rho^{m,n}(t,x) f(x) g(x) \,\mathrm{d}x,
\end{equation*}
it follows that $M(t)$ is invertible for all $t \in [0,T]$. By the Cauchy--Lipschitz Theorem the initial-value problem (\ref{eq63}) has a unique global solution. As a consequence, (\ref{eq60}) has a unique solution $\bm{v}^{m,n }\in C^1([0,T]; V^m)$ subject to the initial condition $\bm{v}^{m,n}(0) = \bm{v}_{0,m}$.
\end{proof}
\begin{clm}\label{clm2.1.1}
Problem (\ref{eq84}) has a unique solution $\hat{\psi}^{m,n} \in C^1([0,T]; X^n)$ subject to the initial condition
\begin{equation*}
\hat{\psi}^{m,n}(\cdot, \cdot, 0) = \hat{\psi}_{0,n} \coloneqq \sum_{i=1}^n (T_\ell(\hat{\psi}_0^m), \varphi^m_i)_{\mathcal{O}}\, \varphi_i^m.
\end{equation*}
\end{clm}
\begin{proof}
Note that (\ref{eq84}) is a system of linear ordinary differential equations. Thus, by writing $\hat{\psi}^{m,n} = \sum_{i=1}^n \beta^{m,n}_i(t)\varphi^m_i$, we deduce that (\ref{eq84}) has a unique solution $\hat{\psi}^{m,n} \in C^1([0,T]; X^n)$ subject to the initial condition $\hat{\psi}^{m,n}(0) = \hat{\psi}_{0,n}$. The detailed argument is similar to the proof of Claim \ref{clm2.1}.
\end{proof}
Let $\|\cdot\|_{V^m}$ and  $\|\cdot \|_{X^n}$ be norms on $V^m$ and $X^n$, respectively. Since $V^m$ and $X^n$ are finite-dimensional linear spaces, and all norms on finite-dimensional linear spaces are equivalent, the precise choice of these norms is of no relevance in the discussion that follows. 
\begin{clm}\label{clm2.2}
Let 
\begin{equation*}
\mathcal{K} \coloneqq \left\{ \bm{v} \in C^1([0,T]; V^m); \sup_{t \in [0,T]} \| \bm{v}(t) \|_2 \leq C, \quad \sup_{t \in [0,T]} \left\| \frac{\partial \bm{v}}{\partial t} \right\|_{V^m} \leq C(m,n) \right\} \subset C([0,T]; V^m)
\end{equation*}
and
\begin{equation*}
\mathcal{S} \coloneqq \left\{ \hat{\psi} \in C^1([0,T]; X^n); \sup_{t \in [0,T]} \| \hat{\psi}(t) \|_{L^2_{M^m}(\mathcal{O})} \leq C, \quad \sup_{t \in [0,T]} \left\| \frac{\partial \hat{\psi}}{\partial t} \right\|_{X^n} \leq C(m,n) \right\} \subset C([0,T]; X^n).
\end{equation*}
Let $\Theta: \overline{\mathcal{K} \times \mathcal{S}} \to \overline{\mathcal{K} \times \mathcal{S}}$ denote the map that takes the pair $(\bm{u}^{m,n}, \xi^{m,n})$ to $(\bm{v}^{m,n}, \hat{\psi}^{m,n}) =: \Theta(\bm{u}^{m,n}, \xi^{m,n})$ via the procedure (\ref{eq60}) and (\ref{eq84}); then, for $C$ and $C(m,n)$ sufficiently large, the mapping $\Theta$ has a fixed point in $\overline{\mathcal{K} \times \mathcal{S}}$.
\end{clm}
\begin{proof}
To show that $\Theta$ has a fixed point we apply Schauder's fixed-point theorem. First we note that $\overline{\mathcal{K} \times \mathcal{S}}$ is obviously non-empty and it is easy to show that $\overline{\mathcal{K} \times \mathcal{S}}$ is convex. Then, it remains to show that: (i) $\Theta$ maps $\mathcal{K} \times \mathcal{S}$ into itself; (ii) $\mathcal{K} \times \mathcal{S}$ is relatively compact in $C([0,T]; V^m) \times C([0,T]; X^n)$; then $\overline{\mathcal{K} \times \mathcal{S}}$ is compact in $C([0,T]; V^m) \times C([0,T]; X^n)$; (iii) $\Theta: \overline{\mathcal{K} \times \mathcal{S}} \to \overline{\mathcal{K} \times \mathcal{S}}$ is continuous. \par
We start by showing suitable energy estimates. We observe that, for any $\bm{w} \in C([0,T]; V^m)$, we can take $\bm{w}(s)$ as a test function in (\ref{eq60}) and integrate with respect to time over $[0,t]$, where $t \in (0,T]$. We obtain
\begin{multline}\label{eq69}
\int_0^t \int_\Omega \rho^{m,n} \left( \frac{\partial \bm{v}^{m,n}}{\partial t} + (\bm{u}^{m,n} \cdot \nabla_x)\bm{v}^{m,n} \right) \cdot \bm{w} \,\mathrm{d}x \,\mathrm{d}s +\int_0^t \int_\Omega \mu(\rho^{m,n},\lambda^{m,n}) D(\bm{v}^{m,n}) : D(\bm{w})\,\mathrm{d}x \,\mathrm{d}s \\
= \int_0^t\int_\Omega \rho^{m,n} \bm{f}^m \cdot \bm{w} \,\mathrm{d}x \,\mathrm{d}s- \int_0^t \int_\Omega \tau^{m,n} : \nabla_x \bm{w} \,\mathrm{d}x\,\mathrm{d}s.
\end{multline}
Similarly, for any $\varphi \in C([0,T]; X^n)$, we can take $\varphi(s)$ as a test function in (\ref{eq84}) and integrate with respect to time over $[0,t]$, where $t \in (0,T]$. We obtain
\begin{equation}\label{eq69.1}
\begin{split}
&\int_0^t \int_{\mathcal{O}} M^m \zeta(\rho^{m,n}) \frac{\partial \hat{\psi}^{m,n}}{\partial t} \varphi \,\mathrm{d}x\,\mathrm{d}\bm{q} \,\mathrm{d}s + \int_0^t \int_{\mathcal{O}} M^m \zeta(\rho^{m,n}) (\nabla_x \hat{\psi}^{m,n} \cdot \bm{u}^{m,n}) \varphi \,\mathrm{d}x\,\mathrm{d}\bm{q} \,\mathrm{d}s \\
&\quad + \int_0^t \int_{\mathcal{O}} M^m \nabla_x \hat{\psi}^{m,n} \cdot \nabla_x \varphi \,\mathrm{d}x\,\mathrm{d}\bm{q}\,\mathrm{d}s +\int_0^t \int_{\mathcal{O}} M^m \mathbb{A}(\nabla_{\bm{q}} \hat{\psi}^{m,n}) : \nabla_{\bm{q}} \varphi \,\mathrm{d}x\,\mathrm{d}\bm{q}\,\mathrm{d}s \\
&= \int_0^t \int_{\mathcal{O}} M \zeta(\rho^{m,n}) \Lambda_\ell(\xi^{m,n}) (\nabla_x \bm{u}^{m,n}) \bm{q} : \nabla_{\bm{q}} \varphi \,\mathrm{d}x\,\mathrm{d}\bm{q}\,\mathrm{d}s.
\end{split}
\end{equation}
Taking the test function $\bm{w} = \bm{v}^{m,n}$ in (\ref{eq69}) we deduce that
\begin{equation}\label{eq70}
\begin{split}
&\frac{1}{2} \int_\Omega \rho^{m,n}(t) |\bm{v}^{m,n}(t)|^2 \,\mathrm{d}x - \frac{1}{2} \int_0^t \int_\Omega \frac{\partial \rho^{m,n}}{\partial t} |\bm{v}^{m,n}|^2 \,\mathrm{d}x \,\mathrm{d}s + \frac{1}{2} \int_0^t \int_\Omega \rho^{m,n} \bm{u}^{m,n} \cdot \nabla_x (|\bm{v}^{m,n}|^2) \,\mathrm{d}x\,\mathrm{d}s \\
&\quad + \int_0^t \int_\Omega \mu(\rho^{m,n},\lambda^{m,n}) |D(\bm{v}^{m,n})|^2\,\mathrm{d}x \,\mathrm{d}s \\
&= \frac{1}{2} \int_\Omega \rho_{0}^m |\bm{v}_{0,m}|^2 \,\mathrm{d}x + \int_0^t\int_\Omega \rho^{m,n} \bm{f}^m \cdot \bm{v}^{m,n} \,\mathrm{d}x \,\mathrm{d}s - \int_0^t \int_\Omega \tau^{m,n} : \nabla_x \bm{v}^{m,n} \,\mathrm{d}x\,\mathrm{d}s.
\end{split}
\end{equation}
Since $\bm{v}^{m,n} \in C^1([0,T]; V^m)$, we can test the transport equation (\ref{eq56}) with $|\bm{v}^{m,n}|^2$ and we see that the second and third term in the above identity add up to $0$. From the bounds (\ref{eq58}) and (\ref{eq62}), the assumption (\ref{eq26}), Young's inequality and Korn's inequality (\ref{eq1.26}), we deduce from (\ref{eq70}) that
\begin{equation}\label{eq2.70}
\begin{split}
\int_\Omega \rho^{m,n}(t) |\bm{v}^{m,n}(t)|^2 \,\mathrm{d}x + \mu_{\min} \int_0^t \int_\Omega |D(\bm{v}^{m,n})|^2 \,\mathrm{d}x\,\mathrm{d}s &\leq \rho_{\max}\int_\Omega |\bm{v}_{0,m}|^2 \,\mathrm{d}x + \rho_{\max} \int_0^t \int_\Omega |\bm{f}^m |^2 \,\mathrm{d}x\,\mathrm{d}s \\
&\quad + \rho_{\max} \int_0^t\int_\Omega |\bm{v}^{m,n} |^2 \,\mathrm{d}x\,\mathrm{d}s \\
&\quad + \frac{1}{\mu_{\min}c_0} \int_0^t \int_\Omega |\tau^{m,n}|^2 \,\mathrm{d}x\,\mathrm{d}s \\
&\leq C + C\int_0^t\int_\Omega  |\bm{v}^{m,n}|^2 \,\mathrm{d}x\,\mathrm{d}s,
\end{split}
\end{equation}
where $C$ is a constant depending on the data $\ell, \bm{f}, \bm{v}_0, \mu_{\min}, M, \zeta_{\max}$. On noting that $\rho^{m,n} \geq \rho_{\min}>0$, we have in particular that
\begin{equation*}
\int_\Omega |\bm{v}^{m,n}(t)|^2 \,\mathrm{d}x \leq C+C \int_0^t \int_\Omega |\bm{v}^{m,n}|^2 \,\mathrm{d}x\,\mathrm{d}s.
\end{equation*}
By Gronwall's inequality we have that
\begin{equation}\label{eq73}
\sup_{t \in [0,T]} \| \bm{v}^{m,n}(t) \|_2^2 + \mu_{\min} \int_0^T \int_\Omega |D(\bm{v}^{m,n})|^2 \,\mathrm{d}x\,\mathrm{d}t \leq C.
\end{equation}
Similarly, by taking the test function $\varphi = \hat{\psi}^{m,n}$ in (\ref{eq69.1}) we get
\begin{multline}\label{eq88}
\hspace{-3mm}\frac{1}{2} \int_0^t \int_{\mathcal{O}} M^m \zeta(\rho^{m,n}) \left( \frac{\partial (\hat{\psi}^{m,n})^2}{\partial t} + \bm{u}^{m,n} \cdot \nabla_x (\hat{\psi}^{m,n})^2 \right) \,\mathrm{d}x\,\mathrm{d}\bm{q}\,\mathrm{d}s + \int_0^t \int_{\mathcal{O}} M^m |\nabla_x \hat{\psi}^{m,n}|^2 \,\mathrm{d}x\,\mathrm{d}\bm{q}\,\mathrm{d}s \\
+ \int_0^t \int_{\mathcal{O}} M^m \mathbb{A}(\nabla_{\bm{q}} \hat{\psi}^{m,n}) : \nabla_{\bm{q}} \hat{\psi}^{m,n} \,\mathrm{d}x\,\mathrm{d}\bm{q}\,\mathrm{d}s \\
= \int_0^t \int_{\mathcal{O}} M \zeta(\rho^{m,n}) \Lambda_\ell(\xi^{m,n}) (\nabla_x \bm{u}^{m,n}) \bm{q} : \nabla_{\bm{q}} \hat{\psi}^{m,n} \,\mathrm{d}x\,\mathrm{d}\bm{q}\,\mathrm{d}s.
\end{multline}
Since $\zeta$ is a $C^1$ function of the density, by the renormalization property we have that 
\begin{equation}\label{eq85}
\int_0^t \int_{\mathcal{O}} \zeta(\rho^{m,n}) \left(\frac{\partial \phi}{\partial t} + \bm{u}^{m,n} \cdot \nabla_x \phi \right)\mathrm{d}x\,\mathrm{d}\bm{q}\,\mathrm{d}s - \int_{\mathcal{O}} \zeta(\rho^{m,n}(t))\phi(t) \,\mathrm{d}x\,\mathrm{d}\bm{q} + \int_\mathcal{O} \zeta(\rho_{0}^m) \phi(0) \,\mathrm{d}x\,\mathrm{d}\bm{q} = 0
\end{equation}
for any $\phi \in C^{0,1}([0,T] \times \overline{\mathcal{O}})$. As $\hat\phi^{m,n}\in
C^1([0,T]; W^{(K+1)d+1,2}(\mathcal{O}))\hookrightarrow
C^1([0,T]; C^1(\overline{\mathcal{O}}))$, 
it follows that $|\hat{\psi}^{m,n}|^2 \in C^1([0,T]; C^1(\overline{\mathcal{O}}))$. Thanks to the assumed smoothness of $M$ (and thereby also of $M^m$), we can take the test function $\phi = M^m |\hat{\psi}^{m,n}|^2$ in (\ref{eq85}) to get that
\begin{multline}\label{eq89}
\int_0^t \int_{\mathcal{O}} M^m \zeta(\rho^{m,n}) \left( \frac{\partial (\hat{\psi}^{m,n})^2}{\partial t} + \bm{u}^{m,n} \cdot \nabla_x (\hat{\psi}^{m,n})^2 \right) \,\mathrm{d}x\,\mathrm{d}\bm{q}\,\mathrm{d}s \\
= \int_{\mathcal{O}} M^m \zeta(\rho^{m,n}(t)) |\hat{\psi}^{m,n}(t)|^2 \,\mathrm{d}x\,\mathrm{d}\bm{q} - \int_{\mathcal{O}} M^m \zeta(\rho_{0}^m) |\hat{\psi}_{0,n}|^2 \,\mathrm{d}x\,\mathrm{d}\bm{q}.
\end{multline}
By multiplying (\ref{eq89}) by $1/2$ and subtracting the resulting equation from (\ref{eq88}) we obtain
\begin{equation*}
\begin{split}
&\frac{1}{2}\int_{\mathcal{O}} M^m \zeta(\rho^{m,n}(t)) |\hat{\psi}^{m,n}(t)|^2 \,\mathrm{d}x\,\mathrm{d}\bm{q} + \int_0^t \int_{\mathcal{O}} M^m |\nabla_x \hat{\psi}^{m,n}|^2 \,\mathrm{d}x\,\mathrm{d}\bm{q}\,\mathrm{d}s \\
&\quad + \int_0^t \int_{\mathcal{O}} M^m \mathbb{A}(\nabla_{\bm{q}} \hat{\psi}^{m,n}) : \nabla_{\bm{q}} \hat{\psi}^{m,n} \,\mathrm{d}x\,\mathrm{d}\bm{q}\,\mathrm{d}s \\
&= \frac{1}{2} \int_{\mathcal{O}} M^m \zeta(\rho_{0}^m) |\hat{\psi}_{0,n}|^2 \,\mathrm{d}x\,\mathrm{d}\bm{q} + \int_0^t \int_{\mathcal{O}} M \zeta(\rho^{m,n}) \Lambda_\ell(\xi^{m,n}) (\nabla_x \bm{u}^{m,n}) \bm{q} : \nabla_{\bm{q}} \hat{\psi}^{m,n} \,\mathrm{d}x\,\mathrm{d}\bm{q}\,\mathrm{d}s.
\end{split}
\end{equation*}
On noting (\ref{eq1.14}), (\ref{eq26}) and the presence of the truncation function $\Lambda_{\ell}(\cdot)$ we can apply H\"{o}lder's inequality and Young's inequality to get that
\begin{equation}\label{eq92.1}
\begin{split}
&\zeta_{\min} \int_{\mathcal{O}} M^m |\hat{\psi}^{m,n}(t)|^2 \,\mathrm{d}x\,\mathrm{d}\bm{q} + 2 \int_0^t \int_{\mathcal{O}} M^m |\nabla_x \hat{\psi}^{m,n}|^2 \,\mathrm{d}x\,\mathrm{d}\bm{q}\,\mathrm{d}s + C_1\int_0^t \int_{\mathcal{O}} M^m |\nabla_{\bm{q}} \hat{\psi}^{m,n}|^2 \,\mathrm{d}x\,\mathrm{d}\bm{q}\,\mathrm{d}s \\
& \leq \zeta_{\max} \int_{\mathcal{O}} M^m |\hat{\psi}_{0,n}|^2 \,\mathrm{d}x\,\mathrm{d}\bm{q} +\frac{1}{C_1} \int_0^t \int_{\mathcal{O}} M |\zeta(\rho^{m,n})|^2 |\Lambda_\ell(\xi^{m,n})|^2 |\nabla_x \bm{u}^{m,n}|^2 |\bm{q}|^2 \,\mathrm{d}x\,\mathrm{d}\bm{q}\,\mathrm{d}s \\
&\leq C(\ell, M, \zeta_{\max}, \hat{\psi}_0).
\end{split}
\end{equation}
From the estimates (\ref{eq73}) and (\ref{eq92.1}) we deduce that
\begin{align}
\label{eq1.96} \sup_{t \in [0,T]} \| \bm{v}^{m,n}(t) \|_2 &\leq C, \\
\label{eq1.97} \sup_{t \in [0,T]} \| \hat{\psi}^{m,n}(t) \|_{L^2_{M^m}(\mathcal{O})} &\leq C.
\end{align}
Next, we shall derive estimates for the norms of $\partial_t \bm{v}^{m,n}$ and $\partial_t \hat{\psi}^{m,n}$. We take the test function $\bm{w} = \partial_t \bm{v}^{m,n}$ in (\ref{eq60}). Since $V^m$ is finite-dimensional, all norms on $V^m$ are equivalent (with constants depending on $m$). We obtain, using H\"{o}lder's inequality, that
\begin{equation}\label{eq74}
\begin{split}
\int_\Omega \rho^{m,n} \left| \frac{\partial \bm{v}^{m,n}}{\partial t} \right|^2 \,\mathrm{d}x &\leq C\|\bm{f}^m \|_2 \left\| \frac{\partial \bm{v}^{m,n}}{\partial t} \right\|_2 + C \| \bm{u}^{m,n} \|_3 \| \nabla_x \bm{v}^{m,n} \|_2 \left\| \frac{\partial \bm{v}^{m,n}}{\partial t} \right\|_6 \\
&\quad + C\left\| \frac{\partial \bm{v}^{m,n}}{\partial t} \right\|_{W^{1,2}(\Omega)} + C \| \nabla_x \bm{v}^{m,n} \|_2 \left\| \frac{\partial \bm{v}^{m,n}}{\partial t} \right\|_{W^{1,2}(\Omega)} \\
&\leq C(m,n)\left\| \frac{\partial \bm{v}^{m,n}}{\partial t} \right\|_2.
\end{split}
\end{equation}
On noting that $\rho^{m,n} \geq \rho_{\min}$ it follows that
\begin{equation*}
\left\| \frac{\partial \bm{v}^{m,n}}{\partial t} \right\|_2 \leq \frac{C(m,n)}{\rho_{\min}}.
\end{equation*}
Since all norms on $V^m$ are equivalent, we have that
\begin{equation}\label{eq76}
\sup_{t \in [0,T]} \left\| \frac{\partial \bm{v}^{m,n}}{\partial t} \right\|_{V^m} \leq C(m,n).
\end{equation}
Now let us take the test function $\varphi = \partial_t \hat{\psi}^{m,n}$ in (\ref{eq84}). Since $X^n$ is finite-dimensional, all norms on $X^n$ are equivalent (with constants depending on $n$). We obtain, using H\"{o}lder's inequality, that
\begin{equation}\label{eq93}
\begin{split}
\int_{\mathcal{O}} M^m \zeta(\rho^{m,n}) \left| \frac{\partial \hat{\psi}^{m,n}}{\partial t} \right|^2 \,\mathrm{d}x\,\mathrm{d}\bm{q} &\leq C(M, \zeta_{\max}) \| \nabla_x \hat{\psi}^{m,n} \|_{L^2(\mathcal{O})} \| \bm{v}^{m,n} \|_{L^3(\Omega)} \left\| \frac{\partial \hat{\psi}^{m,n}}{\partial t} \right\|_{L^6(\mathcal{O})} \\
&\quad + C(M) \| \nabla_{x,\bm{q}} \hat{\psi}^{m,n} \|_{L^2(\mathcal{O})} \left\| \frac{\partial \hat{\psi}^{m,n}}{\partial t} \right\|_{W^{1,2}(\mathcal{O})} \\
&\quad + C(\ell, M,\zeta_{\max}) \| \nabla_x \bm{u}^{m,n} \|_{L^2(\Omega)} \left\| \frac{\partial \hat{\psi}^{m,n}}{\partial t} \right\|_{W^{1,2}(\mathcal{O})} \\
&\leq C(n) \left\| \frac{\partial \hat{\psi}^{m,n}}{\partial t} \right\|_{L^2(\mathcal{O})}.
\end{split}
\end{equation}
Since $\zeta(\rho^{m,n}) \geq \zeta_{\min}$ and $M^m \geq 1/m$, we have
\begin{equation*}
\left\| \frac{\partial \hat{\psi}^{m,n}}{\partial t} \right\|_2 \leq \frac{mC(n)}{\zeta_{\min}}.
\end{equation*}
Since all norms on $X^n$ are equivalent, we have that
\begin{equation}\label{eq94}
\sup_{t \in [0,T]} \left\| \frac{\partial \hat{\psi}^{m,n}}{\partial t} \right\|_{X^n} \leq C(m,n).
\end{equation}
From the estimates (\ref{eq76}) and (\ref{eq94}) we obtain that
\begin{align}
\label{eq1.104} \sup_{t \in [0,T]} \left\| \frac{\partial \bm{v}^{m,n}}{\partial t} \right\|_{V^m} &\leq C(m,n), \\
\label{eq1.105} \sup_{t \in [0,T]} \left\| \frac{\partial \hat{\psi}^{m,n}}{\partial t} \right\|_{X^n} &\leq C(m,n).
\end{align}
From (\ref{eq1.96}), (\ref{eq1.97}), (\ref{eq1.104}) and (\ref{eq1.105}) we see that if we take $(\bm{u}^{m,n}, \xi^{m,n}) \in \mathcal{K} \times \mathcal{S}$, then $\Theta(\bm{u}^{m,n}, \xi^{m,n}) = (\bm{v}^{m,n}, \hat{\psi}^{m,n})$ still belongs to $\mathcal{K} \times \mathcal{S}$. Therefore, we have shown that (i) $\Theta$ maps $\mathcal{K} \times \mathcal{S}$ into itself. \par
(ii) From (\ref{eq1.96}) we obtain that, for any $t \in [0,T]$, $\| \bm{v}^{m,n}(t) \|_{V^m} \leq C(m)$, thanks to the fact that all norms are equivalent in finite-dimensional spaces. Then, the subset $\mathcal{K}(t) \coloneqq \{ \bm{v}(t); \bm{v} \in \mathcal{K} \}$ is relatively compact in $V^m$. Also, since $\bm{v} \in C^1([0,T]; V^m)$ and $\bm{v}$ satisfies (\ref{eq1.104}), we deduce that
\begin{equation*}
\| \bm{v}(t_1) - \bm{v}(t_2) \|_{V^m} \leq C(m,n)|t_1 - t_2|.
\end{equation*}
Then, for all $t_1 \in [0,T]$ and for all $\varepsilon > 0$, there exists a $\delta = \varepsilon / 2C(m,n) > 0$ such that 
\begin{equation*}
\| \bm{v}(t_1) - \bm{v}(t_2) \|_{V^{m}} < \varepsilon
\end{equation*}
for all $t_2 \in [0,T]$ with $|t_1 - t_2| < \delta$ and for all $\bm{v} \in \mathcal{K}$. Therefore, by the Arzel\`{a}--Ascoli Theorem, we deduce that $\mathcal{K}$ is relatively compact in $C([0,T]; V^m)$. \par
Since $M^m \geq 1/m$, we have from (\ref{eq1.97}) that, for any $t \in [0,T]$, $\| \hat{\psi}^{m,n}(t) \|_{L^2(\mathcal{O})} \leq C(m)$. This then gives $\| \hat{\psi}^{m,n}(t) \|_{X^n} \leq C(m,n)$, since all norms are equivalent in finite-dimensional spaces. Thus the subset $\mathcal{S}(t) \coloneqq \{ \hat{\psi}(t): \hat{\psi} \in \mathcal{S} \}$ is relatively compact in $X^n$. Also, since $\hat{\psi} \in C^1([0,T]; X^n)$ and $\hat{\psi}$ satisfies (\ref{eq1.105}), we deduce that 
\begin{equation*}
\| \hat{\psi}(t_3) - \hat{\psi}(t_4) \|_{X^n} \leq C(m,n) | t_3 - t_4 |.
\end{equation*}
Then, for all $t_3 \in [0,T]$ and for all $\varepsilon^{\prime} > 0$, there exists a $\delta^{\prime} = \varepsilon^{\prime} / 2C(m,n) > 0$ such that
\begin{equation*}
\| \hat{\psi}(t_3) - \hat{\psi}(t_4) \|_{X^n} < \varepsilon^{\prime}
\end{equation*}
for all $t_4 \in [0,T]$ with $| t_3 - t_4 | < \delta^{\prime}$ and for all $\hat{\psi} \in \mathcal{S}$. Therefore, by the Arzel\'{a}--Ascoli Theorem, $\mathcal{S}$ is relatively compact in $C([0,T]; X^n)$. Hence, we have shown that $\mathcal{K} \times \mathcal{S}$ is relatively compact in $C([0,T]; V^m) \times C([0,T]; X^n)$. \par
(iii) Next we will show that $\Theta$ is continuous; to this end it suffices to show that $\Theta$ is sequentially continuous. Let $( \bm{u}^{m,n}_{(r)})_{r=1}^\infty$ be a sequence in $C([0,T]; V^m)$ which converges to some $\bm{u}^{m,n}$ in $C([0,T]; V^m)$ as $r \to \infty$ and let $( \xi^{m,n}_{(r)})_{r=1}^\infty$ be a sequence in $C([0,T]; X^n)$ which converges to some $\xi^{m,n}$ in $C([0,T]; X^n)$ as $r \to \infty$. To show that $\Theta$ is sequentially continuous we shall show that $(\bm{v}^{m,n}_{(r)}, \hat{\psi}^{m,n}_{(r)}) = \Theta(\bm{u}^{m,n}_{(r)}, \xi^{m,n}_{(r)})$ converges to $\Theta(\bm{u}^{m,n}, \xi^{m,n})$ as $r \to \infty$ and $\Theta(\bm{u}^{m,n}, \xi^{m,n}) = (\bm{v}^{m,n}, \hat{\psi}^{m,n})$. 

For any $r$, let $\rho^{m,n}_{(r)}$ be the unique solution to the transport equation (\ref{eq56}) corresponding to the velocity field $\bm{u}_{(r)}^{m,n}$. Let $\rho^{m,n}$ be the unique solution to (\ref{eq56}) corresponding to the velocity field $\bm{u}^{m,n}$.
All of these transport problems are associated with the same initial datum $\rho_{0}^m \in W^{d+1,2}(\Omega)$ satisfying \eqref{eq23} (and converging to $\rho_0$ in $L^1(\Omega)$ as $m \rightarrow \infty$; here though, $m\geq 1$ and $n\geq 1$ are fixed, and we are interested, instead, in the limit $r \rightarrow \infty$).

We further define $\lambda^{m,n}_{(r)}$ and $\lambda^{m,n}$ by
\begin{align*}
\lambda^{m,n}_{(r)}(x,t):=\zeta(\rho^{m,n}_{(r)}(x,t))\int_D M^m(\bm{q})\,[\xi^{m,n}_{(r)}(x,\bm{q},t)]_+\,\mathrm{d}\bm{q}, \\
\lambda^{m,n}(x,t):=\zeta(\rho^{m,n}(x,t))\int_D M^m(\bm{q})\,[\xi^{m,n}(x,\bm{q},t)]_+\,\mathrm{d}\bm{q}.
\end{align*}
By Theorem VI.1.9 in \cite{MR2986590} we deduce that, as $r \to \infty$,
\begin{equation}\label{eq79}
\rho_{(r)}^{m,n} \longrightarrow \rho^{m,n} \qquad \text{strongly in $C([0,T]; L^p(\Omega))$ for any $p \in [1,\infty)$}.
\end{equation}
From the assumptions (\ref{eq26}) on $\zeta$ we then deduce that, as $r \to \infty$,
\begin{alignat}{2}
\label{eq81} \zeta(\rho_{(r)}^{m,n}) &\longrightarrow \zeta(\rho^{m,n}) \qquad &&\text{strongly in $C([0,T]; L^p(\Omega))$ for any $p \in [1,\infty)$}.
\end{alignat}
Hence, and thanks to the global Lipschitz continuity of the mapping $s \in \mathbb{R} \mapsto [s]_+ \in \mathbb{R}_{\geq 0}$, we have that
\begin{equation}\label{eq79a}
\lambda_{(r)}^{m,n} \longrightarrow \lambda^{m,n} \qquad \text{strongly in $C([0,T]; L^p(\Omega))$ for any $p \in [1,\infty)$}.
\end{equation}
From the assumptions (\ref{eq26}) on $\mu$ we then deduce that, as $r \to \infty$,
\begin{alignat}{2}
\label{eq80} \mu(\rho_{(r)}^{m,n},\lambda^{m,n}_{(r)}) &\longrightarrow \mu(\rho^{m,n},\lambda^{m,n}) \qquad &&\text{strongly in $L^{\infty}(0,T; L^p(\Omega))$ for any $p \in [1,\infty)$}.
\end{alignat}
For any $r$, we take the test function in (\ref{eq69}) and (\ref{eq69.1}) to be $\bm{v}_{(r)}^{m,n}$ and $\hat{\psi}^{m,n}_{(r)}$ respectively and perform a similar procedure as in (\ref{eq69})--(\ref{eq92.1}). Then, we deduce that
\begin{align}
\label{eq82.1} \sup_r \| \bm{v}_{(r)}^{m,n} \|_{C([0,T]; L^2(\Omega; \mathbb{R}^d))} &\leq C, \\
\label{eq82.2} \sup_r \| \nabla_x \bm{v}_{(r)}^{m,n} \|_{L^2(0,T; L^2(\Omega; \mathbb{R}^{d \times d}))} &\leq C, \\
\label{eq101} \sup_r \| \hat{\psi}_{(r)}^{m,n} \|_{C([0,T]; L^2_{M^m}(\mathcal{O}))} &\leq C,\\
\label{eq102} \sup_r \| \nabla_{x,\bm{q}} \hat{\psi}_{(r)}^{m,n} \|_{L^2(0,T; L^2_{M^m}(\mathcal{O}))} &\leq C.
\end{align}
Taking the test function in (\ref{eq60}) to be $\partial_t \bm{v}_{(r)}^{m,n}$ we perform a similar argument as in (\ref{eq74}), (\ref{eq76}) to deduce that
\begin{equation*}
\sup_r \left\| \frac{\partial \bm{v}_{(r)}^{m,n}}{\partial t} \right\|_{C([0,T]; V^m)} \leq C(m,n).
\end{equation*}
Similarly, taking the test function in (\ref{eq84}) to be $\partial_t \hat{\psi}_{(r)}^{m,n}$, as in (\ref{eq93}), (\ref{eq94}), we get
\begin{equation*}
\sup_r \left\| \frac{\partial \hat{\psi}_{(r)}^{m,n}}{\partial t} \right\|_{C([0,T]; X^n)} < C(m,n).
\end{equation*}
Now since $\bm{v}_{(r)}^{m,n} \in \mathcal{K}$ for all $r$ and $\mathcal{K}$ is relatively compact in $C([0,T]; V^m)$, there exists a subsequence (not relabelled) such that, as $r \to \infty$,
\begin{equation}\label{eq83}
\bm{v}_{(r)}^{m,n} \longrightarrow \bm{v}^{\prime} \qquad \text{strongly in $C([0,T]; V^m)$}.
\end{equation}
From the bounds (\ref{eq82.1}) and (\ref{eq82.2}) we deduce the following weak convergence:
\begin{equation}\label{eq84aa}
\bm{v}_{(r)}^{m,n} \rightharpoonup \bm{v}^\prime \qquad \text{weakly in $L^2(0,T; W^{1,2}(\Omega; \mathbb{R}^d))$}.
\end{equation}
With the convergence results (\ref{eq79}), (\ref{eq79a}),  (\ref{eq80}), (\ref{eq81}), (\ref{eq83}) and $(\ref{eq84aa})$ we can pass to the limit as $r \to \infty$ in the equation satisfied by $\bm{v}_{(r)}^{m,n}$ and the limit $\bm{v}^\prime$ satisfies (\ref{eq60}). However, with $\rho^{m,n}$, $\bm{u}^{m,n}$ and $\xi^{m,n}$ given, (\ref{eq60}) can be solved with a unique solution $\bm{v}^{m,n}$ as proved in Claim \ref{clm2.1}, which then implies that $\bm{v}^\prime \equiv \bm{v}^{m,n}$. \par
Meanwhile, since $\hat{\psi}_{(r)}^{m,n} \in \mathcal{S}$ for all $r$ and $\mathcal{S}$ is relatively compact in $C([0,T]; X^n)$, there exists a subsequence (not relabelled) such that, as $r \to \infty$,
\begin{equation}\label{eq104}
\hat{\psi}_{(r)}^{m,n} \longrightarrow \hat{\psi}^\prime \qquad \text{strongly in $C([0,T]; X^n)$}.
\end{equation}
From the bounds (\ref{eq101}) and (\ref{eq102}) we deduce the following weak convergence:
\begin{align}\label{eq105a}
\sqrt{M^m} \nabla_x \hat{\psi}_{(r)}^{m,n} &\rightharpoonup \sqrt{M^m} \nabla_x \hat{\psi}^\prime \qquad \text{weakly in $L^2(0,T; L^2(\mathcal{O}))$}, \\
\label{eq106} \sqrt{M^m} \nabla_{\bm{q}} \hat{\psi}_{(r)}^{m,n} &\rightharpoonup \sqrt{M^m} \nabla_{\bm{q}} \hat{\psi}^\prime \qquad \text{weakly in $L^2(0,T; L^2(\mathcal{O}))$}.
\end{align}
With the convergence results (\ref{eq81}), (\ref{eq104}), (\ref{eq105a}) and (\ref{eq106}) we can pass to the limit as $r \to \infty$ in the equation satisfied by $\hat{\psi}_{(r)}^{m,n}$ and the limit $\hat{\psi}^\prime$ satisfies (\ref{eq84}). However, with $\rho^{m,n}$, $\bm{u}^{m,n}$ and $\xi^{m,n}$ given, (\ref{eq84}) can be solved uniquely with $\hat{\psi}^{m,n}$ as proved in Claim \ref{clm2.1.1}, which implies that $\hat{\psi}^\prime \equiv \hat{\psi}^{m,n}$. Hence, we have shown that the mapping $\Theta$ is continuous. \par
Finally, by Schauder's fixed-point theorem, we deduce that $\Theta: \overline{\mathcal{K} \times \mathcal{S}} \to \overline{\mathcal{K} \times \mathcal{S}}$ has a fixed point $(\bm{v}^{m,n}, \hat{\psi}^{m,n})$ in $\overline{\mathcal{K} \times \mathcal{S}}$. Thus, also, 
\[ \varrho^{m,n}(x,t)=\zeta(\rho^{m,n}(x,t))\int_D M^m(\bm{q})\, [\hat\psi^{m,n}(x,\bm{q},t)]_{+}\,\mathrm{d}\bm{q}.\]
and
\begin{align*}
\tau^{m,n}(x,t) &=  -k \int_D \bigg[ K M(\bm{q}) \zeta(\rho^{m,n}(x,t))T_\ell(\hat\psi^{m,n}(x,\bm{q},t))I \\
&\qquad +  \sum_{j=1}^K \zeta(\rho^{m,n}(x,t)) T_\ell(\hat\psi^{m,n}(x,\bm{q},t)\nabla_{\bm{q}^j} M(\bm{q}) \otimes \bm{q}^j \bigg] \,\mathrm{d}\bm{q}.
\end{align*}
That completes the proof of the existence of a solution $(\rho^{m,n}, \bm{v}^{m,n}, \hat{\psi}^{m,n})$ to the partially Galerkin discretized system (\ref{eq48})--(\ref{eq1.66}).
\end{proof}

In the following sections we shall derive uniform bounds independent of the parameters $n$, $m$ and $\ell$, and then use those to successively pass to the limits with $n, m, \ell \to \infty$.

\subsection{Passage to the limit with $n$}
The goal of this section is to pass to the limit as $n \to \infty$. To achieve this we shall first derive uniform estimates independent of $n$. We note that from the definition (\ref{eq49.1}) of $\tau^{m,n}$ we deduce that
\begin{equation}\label{eq161}
|\tau^{m,n}| \leq C(\ell, \zeta_{\max}, M),
\end{equation}
which implies that there exists a subsequence (not relabelled) such that
\begin{equation}\label{eq171}
\tau^{m,n} \rightharpoonup \tau^m \qquad \text{weak* in $L^{\infty}(\Omega \times (0,T); \mathbb{R}^{d \times d})$}.
\end{equation}
\par
By Theorem VI.1.6 in \cite{MR2986590} we find that
\begin{equation*}
\sup_{t \in (0,T)} \| \rho^{m,n}(t) \|_{L^\infty(\Omega)} \leq \| \rho_0^m \|_{L^\infty(\Omega)} \leq \rho_{\max}.
\end{equation*}
Moreover, Proposition VI.1.8 in \cite{MR2986590} gives
\begin{equation}\label{eq163}
\rho^{m,n} \geq \rho_{\min}.
\end{equation}
We deduce the existence of subsequences (not relabelled) such that, as $n \to \infty$,
\begin{align}
\label{eq173} \rho^{m,n} &\rightharpoonup \rho^m \qquad \ \, \, \text{weak* in $L^\infty(\Omega \times (0,T))$}, \\
\label{eq174} (\rho^{m,n})^2 &\rightharpoonup \vartheta^m \qquad \ \,  \text{weak* in $L^\infty(\Omega \times (0,T))$}.
\end{align}
By the renormalization property, for any $\beta \in C^1(\mathbb{R})$, $\beta(\rho^{m,n})$ satisfies
\begin{equation*}
\frac{\partial \beta(\rho^{m,n})}{\partial t} + \divergence_x (\beta(\rho^{m,n}) \bm{v}^{m,n}) = 0
\end{equation*}
in the distributional sense. It follows from the Sobolev embedding theorem that
\begin{equation}\label{eq108}
\begin{split}
\left\| \frac{\partial \beta(\rho^{m,n})}{\partial t} \right\|_{L^2(0,T; W^{-1,p}(\Omega))} &\leq \| \beta(\rho^{m,n}) \bm{v}^{m,n} \|_{L^2(0,T; L^p(\Omega;\mathbb{R}^d))}\\
& \leq C \| \bm{v}^{m,n} \|_{L^2(0,T; L^p(\Omega; \mathbb{R}^d))} \\
&\leq  C \| \bm{v}^{m,n} \|_{L^2(0,T; W^{1,2}(\Omega; \mathbb{R}^d))} \\
&\leq C
\end{split}
\end{equation}
for any $p \in [1,\infty)$ when $d=2$ and $p \in [1,6]$ when $d=3$.
\par
We multiply the $i$-th equation in (\ref{eq47}) by $c^{m,n}_i(t)$ and sum with respect to $i = 1,\dots, m$ to deduce the following identity
\begin{multline}\nonumber
\frac{1}{2}\frac{\mathrm{d}}{\,\mathrm{d}t} \int_\Omega \rho^{m,n}|\bm{v}^{m,n}|^2 \,\mathrm{d}x + \frac{1}{2} \int_\Omega \frac{\partial \rho^{m,n}}{\partial t} |\bm{v}^{m,n}|^2 \,\mathrm{d}x - \frac{1}{2}\int_\Omega \rho^{m,n} \bm{v}^{m,n} \cdot \nabla_x (|\bm{v}^{m,n}|^2) \,\mathrm{d}x \\
+ \int_\Omega \mu(\rho^{m,n},\varrho^{m,n}) |D(\bm{v}^{m,n})|^2 \,\mathrm{d}x = - \int_\Omega \tau^{m,n} : D(\bm{v}^{m,n}) \,\mathrm{d}x+ (\rho^{m,n} \bm{f}^m, \bm{v}^{m,n}).
\end{multline}
The second term and the third term in the above equality add up to $0$, since we can take the test function $\eta = |\bm{v}^{m,n}|^2$ in (\ref{eq48}). Using (\ref{eq26}), (\ref{eq161}), Young's inequality and Gronwall's inequality, we find that
\begin{equation*}
\sup_{t \in (0,T)} \| \sqrt{\rho^{m,n}(t)} \bm{v}^{m,n}(t) \|^2_{L^2(\Omega; \mathbb{R}^d)} + \mu_{\min} \int_0^T \int_\Omega |D(\bm{v}^{m,n})|^2 \,\mathrm{d}x\,\mathrm{d}t \leq C(\ell, M, \bm{v}_0, \bm{f}, \mu_{\min}, \rho_{\max}).
\end{equation*}
Using (\ref{eq163}) and Korn's inequality (\ref{eq1.26}) we deduce that
\begin{equation*}
\sup_{t \in (0,T)} \| \bm{v}^{m,n}(t) \|^2_{L^2(\Omega; \mathbb{R}^d)} + c_0 \mu_{\min} \int_0^T \| \bm{v}^{m,n} \|^2_{W^{1,2}(\Omega; \mathbb{R}^d)} \, \mathrm{d}t \leq C(\ell),
\end{equation*}
which implies, by the orthogonality of the basis, that
\begin{equation}\label{eq176}
\sup_{t \in (0,T);\, i = 1, \dots,m} |c^{m,n}_i(t)| \leq C(m,\ell).
\end{equation}
From the definition of $\bm{v}^{m,n}$ we see that
\begin{equation}
\label{eq169} \frac{\partial \bm{v}^{m,n}}{\partial t}(x,t) = \sum_{j=1}^m \frac{\mathrm{d}c^{m,n}_j(t)}{\,\mathrm{d}t} \bm{w}_j(x).
\end{equation}
Substituting (\ref{eq169}) into (\ref{eq47}) we have that, for a fixed $i \in \{ 1,\dots, m \}$,
\begin{equation*}
\begin{split}
\rho_{\min} \left| \frac{\mathrm{d}c^{m,n}_i(t)}{\,\mathrm{d}t} \right| &\leq |(\partial_t \rho^{m,n} \bm{v}^{m,n},\bm{w}_i)|  + |(\rho^{m,n} \bm{v}^{m,n} \otimes \bm{v}^{m,n}, \nabla_x \bm{w}_i)| + |(\mu(\rho^{m,n},\varrho^{m,n})D(\bm{v}^{m,n}), \nabla_x \bm{w}_i)| \\
&\quad + |(\tau^{m,n}, \nabla_x \bm{w}_i)| + |(\rho^{m,n}\bm{f}^m, \bm{w}_i)| \\
&\leq \| \partial_t \rho^{m,n} \|_{W^{-1,2}(\Omega)} \| \bm{v}^{m,n} \|_{W^{1,2}(\Omega; \mathbb{R}^{d})} \| \bm{w}_i \|_{W^{1,\infty}(\Omega; \mathbb{R}^d)} \\
&\quad + \rho_{\max} \| \bm{v}^{m,n} \|^2_{L^2(\Omega; \mathbb{R}^d)} \| \nabla_x \bm{w}_i \|_{L^\infty(\Omega; \mathbb{R}^{d \times d})} \\
&\quad + \mu_{\max} \| D(\bm{v}^{m,n}) \|_{L^2(\Omega; \mathbb{R}^{d \times d})} \| \nabla_x \bm{w}_i \|_{L^\infty(\Omega; \mathbb{R}^{d \times d})} + C(\ell)\| \nabla_x \bm{w}_i \|_{L^\infty(\Omega; \mathbb{R}^{d \times d})} \\
&\quad + C(\bm{f}, \rho_{\max}) \| \nabla_x \bm{w}_i \|_{L^\infty(\Omega; \mathbb{R}^{d \times d})}   \\
&\leq C(m,\ell) \| \bm{w}_i \|_{W^{1,\infty}(\Omega; \mathbb{R}^d)}.
\end{split}
\end{equation*}
Since $W^{d+1,2}(\Omega) \hookrightarrow W^{1,\infty}(\Omega)$, we have that $\| \bm{w}_i \|_{W^{1,\infty}(\Omega; \mathbb{R}^d)} \leq C(m)$. Therefore,
\begin{equation}\label{eq179}
\sup_{t \in (0,T);\, i =1,\dots,m} \left| \frac{\mathrm{d}c^{m,n}_i(t)}{\,\mathrm{d}t} \right| \leq C(m,\ell).
\end{equation}
By the definition of $\bm{v}^{m,n}$ it is easy to see that
\begin{align}
\begin{aligned}
\| \bm{v}^{m,n} \|_{L^\infty(\Omega; \mathbb{R}^d)} &\leq C(m,\ell), \\
\label{eq177} \| \nabla_x \bm{v}^{m,n} \|_{L^\infty(\Omega; \mathbb{R}^{d \times d})} &\leq C(m,\ell).
\end{aligned}
\end{align}
The uniform bounds (\ref{eq176}) and (\ref{eq179}) imply the following convergence results:
\begin{align}
\label{eq186} c_i^{m,n} &\rightharpoonup c_i^m \qquad \text{weak* in $W^{1,\infty}((0,T))$},\\
\label{eq187} c_i^{m,n} &\to c_i^m \qquad \text{strongly in $C([0,T])$},
\end{align}
which then imply that
\begin{equation}\label{eq188}
\bm{v}^{m,n} \to \bm{v}^m \qquad \text{strongly in $C([0,T]; W^{1,2}_{0,\divergence}(\Omega; \mathbb{R}^d))$}.
\end{equation}
Noting (\ref{eq173}) and (\ref{eq174}) we deduce that
\begin{align}
\label{eq189} \rho^{m,n} \bm{v}^{m,n} &\rightharpoonup \rho^{m} \bm{v}^m \qquad \ \, \text{weakly in $L^p(\Omega \times (0,T) )$}, \\
\label{eq189.1} (\rho^{m,n})^2 \bm{v}^{m,n} &\rightharpoonup \vartheta^m \bm{v}^m \qquad \ \, \text{weakly in $L^p(\Omega \times (0,T))$},
\end{align}
where $p \in [1,\infty)$ when $d=2$ and $p \in [1,6]$ when $d=3$. From the estimate (\ref{eq108}) we deduce that
\begin{align}
\label{eq191} \frac{\partial \rho^{m,n}}{\partial t} &\rightharpoonup \frac{\partial \rho^m}{\partial t} \qquad \ \text{weakly in $L^2(0,T; W^{-1,p}(\Omega))$}, \\
\label{eq191.1} \frac{\partial (\rho^{m,n})^2}{\partial t} &\rightharpoonup \frac{\partial \vartheta^m}{\partial t} \qquad \ \text{weakly in $L^2(0,T; W^{-1,p}(\Omega))$},
\end{align}
where $p \in (1,\infty)$ when $d=2$ and $p \in (1,6]$ when $d=3$. With the convergence results (\ref{eq191}) and (\ref{eq189}) we pass to the limit as $n \to \infty$ and deduce that $\rho^{m}$ is the (unique) weak solution of 
\begin{equation}\label{eq1.125}
\frac{\partial \rho^m}{\partial t} + \divergence_x (\rho^m \bm{v}^m) = 0
\end{equation}
with the initial condition
\begin{equation}\label{rhomnini}
\rho^m(0) = \rho_0^m.
\end{equation}
By the renormalization property, if we define $P^{m,n} = (\rho^{m,n})^2$, then $P^{m,n}$ solves the following initial-value problem:
\begin{align*}
\frac{\partial P^{m,n}}{\partial t} + \divergence_x (P^{m,n} \bm{v}^{m,n}) &= 0, \\
P^{m,n}(0) &= (\rho_0^m)^2.
\end{align*}
Similarly, $P^m = (\rho^m)^2$ solves the following problem:
\begin{align*}
\frac{\partial P^m}{\partial t} + \divergence_x (P^m \bm{v}^m) &= 0, \\
P^m(0) &= (\rho_0^m)^2.
\end{align*}
With the convergence results (\ref{eq191.1}) and (\ref{eq189.1}) we deduce that $\vartheta^m$ also solves the following initial-value problem
\begin{align}
\begin{aligned}
\label{eq1.131} \frac{\partial \vartheta^m}{\partial t} + \divergence_x (\vartheta^m \bm{v}^m) &= 0, \\
\vartheta^m(0) &= (\rho_0^m)^2.
\end{aligned}
\end{align}
However, since (\ref{eq1.131}) is of the same form as (\ref{eq1.125}), the solution to (\ref{eq1.131}) is unique. Hence, $\vartheta^m = (\rho^m)^2$. Then, the convergence result (\ref{eq174}) gives
\begin{equation*}
\int_0^T \int_\Omega |\rho^{m,n}|^2 \,\mathrm{d}x\,\mathrm{d}t \to \int_0^T \int_\Omega |\rho^m|^2 \,\mathrm{d}x\,\mathrm{d}t,
\end{equation*}
which then implies that
\begin{equation*}
\rho^{m,n} \to \rho^m \qquad \text{strongly in $L^2(\Omega \times (0,T))$}.
\end{equation*}
It then follows from (\ref{eq173}) that
\begin{equation}\label{eq194}
\rho^{m,n} \to \rho^m \qquad \text{strongly in $L^p(\Omega \times (0,T))$},
\end{equation}
for any $p \in [1,\infty)$. With the convergence result (\ref{eq188}) for $\bm{v}^{m,n}$ we can perform a similar argument as in Theorem VI.1.9 in \cite{MR2986590} and strengthen the above convergence to get
\begin{equation}\label{rhomn-strong}
\rho^{m,n} \to \rho^m \qquad \text{strongly in $C([0,T]; L^p(\Omega))$},
\end{equation}
for any $p \in [1,\infty)$. Thanks to the assumption (\ref{eq26}) on $\zeta$ we then have that
\begin{align}
\label{eq196} \zeta(\rho^{m,n}) &\to \zeta(\rho^m) \qquad \text{strongly in $C([0,T]; L^p(\Omega))$}.
\end{align}

We multiply the $i$-th equation in (\ref{eq49}) by $d_i^{m,n}$ and sum with respect to $i = 1,\dots,n$ to deduce the following identity
\begin{equation}\label{eq178}
\begin{split}
&\frac{1}{2} \frac{\mathrm{d}}{\,\mathrm{d}t} \int_{\mathcal{O}} M^m \zeta(\rho^{m,n}) (\hat{\psi}^{m,n})^2 \,\mathrm{d}x\,\mathrm{d}\bm{q} + \frac{1}{2} \int_{\mathcal{O}} M^m \frac{\partial \zeta(\rho^{m,n})}{\partial t} (\hat{\psi}^{m,n})^2 \,\mathrm{d}x\,\mathrm{d}\bm{q} \\
&- \frac{1}{2} \int_{\mathcal{O}} M^m \zeta(\rho^{m,n}) \bm{v}^{m,n} \cdot \nabla_x (\hat{\psi}^{m,n})^2 \,\mathrm{d}x\,\mathrm{d}\bm{q} + \int_{\mathcal{O}} M^m |\nabla_x \hat{\psi}^{m,n}|^2 \,\mathrm{d}x\,\mathrm{d}\bm{q} \\
&+ \int_{\mathcal{O}} M^m \mathbb{A}(\nabla_{\bm{q}} \hat{\psi}^{m,n}) : \nabla_{\bm{q}} \hat{\psi}^{m,n} \,\mathrm{d}x\,\mathrm{d}\bm{q} =  \left(M^m \zeta(\rho^{m,n}) \Lambda_\ell(\hat{\psi}^{m,n})(\nabla_x \bm{v}^{m,n}) \bm{q}, \nabla_{\bm{q}} \hat{\psi}^{m,n} \right)_{\mathcal{O}}.
\end{split}
\end{equation}
Note that $\zeta(\rho^{m,n})$ is a renormalized solution in the sense that
\begin{equation*}
\left\langle \partial_t \zeta(\rho^{m,n}), \eta \right\rangle - \left( \bm{v}^{m,n} \zeta(\rho^{m,n}), \nabla_x \eta \right) = 0 \qquad \forall \, \eta \in C^{0,1}([0,T]; C^{0,1}(\overline{\Omega})).
\end{equation*}
Taking the test function $\eta = \int_D M^m (\hat{\psi}^{m,n})^2 \,\mathrm{d}\bm{q}$ we see that the second term and the third term in (\ref{eq178}) add up to $0$. By Young's inequality, the definition of $\Lambda_\ell$, (\ref{eq26}) and (\ref{eq177}), we deduce that
\begin{equation}\label{eq180}
\begin{split}
 \left(M^m \zeta(\rho^{m,n}) \Lambda_\ell(\hat{\psi}^{m,n})(\nabla_x \bm{v}^{m,n}) \bm{q}, \nabla_{\bm{q}} \hat{\psi}^{m,n} \right)_{\mathcal{O}} &\leq \frac{C_1}{2} \int_{\mathcal{O}} M^m |\nabla_{\bm{q}} \hat{\psi}^{m,n}|^2 \,\mathrm{d}x\,\mathrm{d}\bm{q} \\
 &\quad + \frac{C(\ell, \zeta_{\max})}{2} \| \nabla_x \bm{v}^{m,n} \|^2_{\infty} \int_{\mathcal{O}} M^m \zeta(\rho^{m,n})(\hat{\psi}^{m,n})^2 \,\mathrm{d}x\,\mathrm{d}\bm{q} \\
&\leq \frac{C_1}{2} \int_{\mathcal{O}} M^m |\nabla_{\bm{q}} \hat{\psi}^{m,n}|^2 \,\mathrm{d}x\,\mathrm{d}\bm{q} \\
&\quad + C(m,\ell) \int_{\mathcal{O}} M^m \zeta(\rho^{m,n})(\hat{\psi}^{m,n})^2 \,\mathrm{d}x\,\mathrm{d}\bm{q}.
\end{split}
\end{equation}
Inserting (\ref{eq180}) into (\ref{eq178}) and using (\ref{eq1.14}) we deduce that
\begin{equation}\label{eq200}
\frac{\mathrm{d}}{\,\mathrm{d}t} \int_{\mathcal{O}} M^m \zeta(\rho^{m,n}) (\hat{\psi}^{m,n})^2 \,\mathrm{d}x\,\mathrm{d}\bm{q} + \int_{\mathcal{O}} M^m |\nabla_{x,\bm{q}} \hat{\psi}^{m,n}|^2 \,\mathrm{d}x\,\mathrm{d}\bm{q} \leq C(m,\ell) \int_{\mathcal{O}} M^m \zeta(\rho^{m,n}) (\hat{\psi}^{m,n})^2 \,\mathrm{d}x\,\mathrm{d}\bm{q}.
\end{equation}
Gronwall's inequality gives that 
\begin{equation*}
\sup_{t \in (0,T)} \| \sqrt{\zeta(\rho^{m,n}(t))} \hat{\psi}^{m,n}(t) \|^2_{L^2_{M^m}(\mathcal{O})} + \int_0^T \int_{\mathcal{O}} M^m |\nabla_{x,\bm{q}} \hat{\psi}^{m,n}|^2 \,\mathrm{d}x\,\mathrm{d}\bm{q}\,\mathrm{d}t \leq C(m,\ell).
\end{equation*}
By noting that $M^m \geq 1/m$ and (\ref{eq26}) we deduce that
\begin{equation}\label{eq201}
\sup_{t \in (0,T)} \| \hat{\psi}^{m,n}(t) \|^2_{L^2(\mathcal{O})} + \int_0^T \int_{\mathcal{O}} |\nabla_{x,\bm{q}} \hat{\psi}^{m,n}|^2 \,\mathrm{d}x\,\mathrm{d}\bm{q}\,\mathrm{d}t \leq C(m,\ell).
\end{equation}
Since $M^m$ is Lipschitz continuous, we can further deduce that
\begin{equation*}
\int_0^T \| M^m \hat{\psi}^{m,n} \|^2_{W^{1,2}(\mathcal{O})} \,\mathrm{d}t \leq C(m,\ell).
\end{equation*}
Then, using (\ref{eq49}) and a standard calculation, we obtain that 
\begin{equation}\label{eq226}
\int_0^T \| \partial_t(M^m \zeta(\rho^{m,n}) \hat{\psi}^{m,n})  \|^2_{(W^{1,2}(\mathcal{O}))'} \,\mathrm{d}t \leq C(m,\ell).
\end{equation}

Next, we shall focus on the fractional time derivative of $\hat{\psi}^{m,n}$. 
Integrating (\ref{eq49}) with respect to time over $(s, s+h)$, with $s < T-h$, we have
\begin{equation*}
\begin{split}
&\int_{\mathcal{O}} M^m[(\zeta(\rho^{m,n})\hat{\psi}^{m,n})(s+h) - (\zeta(\rho^{m,n}) \hat{\psi}^{m,n})(s)]\varphi^m_i \,\mathrm{d}x\,\mathrm{d}\bm{q} \\
&= \int_s^{s+h} \int_{\mathcal{O}} M^m \zeta(\rho^{m,n}) \hat{\psi}^{m,n} \bm{v}^{m,n} \cdot \nabla_x \varphi^m_i \,\mathrm{d}x\,\mathrm{d}\bm{q}\,\mathrm{d}t \\
&\quad - \int_s^{s+h} \int_{\mathcal{O}} M^m \nabla_x \hat{\psi}^{m,n} : \nabla_x \varphi^m_i \,\mathrm{d}x\,\mathrm{d}\bm{q}\,\mathrm{d}t \\
&\quad - \int_s^{s+h} \int_{\mathcal{O}} M^m \mathbb{A}(\nabla_{\bm{q}} \hat{\psi}^{m,n}) : \nabla_{\bm{q}} \varphi^m_i \,\mathrm{d}x\,\mathrm{d}\bm{q}\,\mathrm{d}t \\
&\quad + \int_s^{s+h} \int_{\mathcal{O}} M\zeta(\rho^{m,n}) \Lambda_\ell(\hat{\psi}^{m,n}) (\nabla_x \bm{v}^{m,n}) \bm{q}:\nabla_{\bm{q}} \varphi^m_i \,\mathrm{d}x\,\mathrm{d}\bm{q}\,\mathrm{d}t \\
&\eqqcolon U_1 + U_2 + U_3 + U_4.
\end{split}
\end{equation*}
For $U_1$,  we have by H\"{o}lder's inequality and the Gagliardo--Nirenberg inequality (\ref{eq1.25}) that
\begin{equation*}
\begin{split}
|U_1| &\leq C(M, \zeta_{\max}) \| \nabla_x \varphi^m_i \|_{L^2(\mathcal{O})} \int_s^{s+h} \| \hat{\psi}^{m,n} \|_{L^q(\Omega; L^2(D))} \| \bm{v}^{m,n} \|_{L^{\frac{2q}{q-2}}(\Omega)} \,\mathrm{d}t \\
&\leq C(M, \zeta_{\max}) \| \nabla_x \varphi^m_i \|_{L^2(\mathcal{O})} \int_s^{s+h} \| \hat{\psi}^{m,n} \|_{L^q(\Omega; L^2(D))} \| \bm{v}^{m,n} \|^{1-\frac{\mathrm{d}}{q}}_{L^2(\Omega)} \| \bm{v}^{m,n} \|^{\frac{\mathrm{d}}{q}}_{W^{1,2}(\Omega)} \,\mathrm{d}t \\
&\leq C(M, \zeta_{\max}) \| \nabla_x \varphi^m_i \|_{L^2(\mathcal{O})} \| \bm{v}^{m,n} \|^{1-\frac{\mathrm{d}}{q}}_{L^\infty(0,T; L^2(\Omega))}  \left( \int_s^{s+h} \| \bm{v}^{m,n} \|^2_{W^{1,2}(\Omega)} \,\mathrm{d}t \right)^{\frac{\mathrm{d}}{2q}} \\
&\quad \times \left( \int_s^{s+h} \| \hat{\psi}^{m,n} \|^{\frac{2q}{2q-d}}_{L^q(\Omega; L^2(D))} \,\mathrm{d}t \right)^{\frac{2q-d}{2q}} \\
&\leq C(M,\zeta_{\max}) h^{\frac{q-d}{2q}} \| \nabla_x \varphi^m_i \|_{L^2(\mathcal{O})} \| \bm{v}^{m,n} \|^{1 - \frac{\mathrm{d}}{q}}_{L^\infty(0,T; L^2(\Omega))} \| \bm{v}^{m,n} \|^{\frac{\mathrm{d}}{q}}_{L^2(s,s+h; W^{1,2}(\Omega))} \\
&\quad \times \left(\int_s^{s+h} \| \hat{\psi}^{m,n} \|^2_{L^q(\Omega; L^2(D))} \,\mathrm{d}t \right)^{\frac{1}{2}}\\
&\leq Ch^{\frac{q-d}{2q}} \| \nabla_{x,\bm{q}} \varphi^m_i \|_{L^2(\mathcal{O})},
\end{split}
\end{equation*}
where $q \in (2,\infty)$ when $d=2$ and $q \in (3,6]$ when $d=3$. For $U_2$, we have that
\begin{equation*}
|U_2| \leq C(M) h^{\frac{1}{2}} \| \nabla_x \hat{\psi}^{m,n} \|_{L^2(s,s+h; L^2(\mathcal{O}))} \| \nabla_x \varphi^m_i \|_{L^2(\mathcal{O})} \leq Ch^{\frac{1}{2}} \| \nabla_{x,\bm{q}} \varphi^m_i \|_{L^2(\mathcal{O})}.
\end{equation*}
Similarly as above, for $U_3$ we have that
\begin{equation*}
|U_3| \leq Ch^{\frac{1}{2}} \| \nabla_{x, \bm{q}} \varphi^m_i \|_{L^2(\mathcal{O})}.
\end{equation*}
For $U_4$, thanks to the presence of the truncation function $\Lambda_\ell$, we obtain that
\begin{equation*}
|U_4| \leq Ch^{\frac{1}{2}} \| \nabla_{\bm{q}} \varphi^m_i \|_{L^2(\mathcal{O})} \| \nabla_x \bm{v}^{m,n} \|_{L^2(s,s+h; L^2(\Omega))} \leq Ch^{\frac{1}{2}} \| \nabla_{x, \bm{q}} \varphi^{m}_i \|_{L^2(\mathcal{O})}.
\end{equation*}
Hence,
\begin{equation}\label{eq132}
\left| \int_{\mathcal{O}} M^m[(\zeta(\rho^{m,n})\hat{\psi}^{m,n})(s+h) - (\zeta(\rho^{m,n}) \hat{\psi}^{m,n})(s)]\varphi^{m}_i \,\mathrm{d}x\,\mathrm{d}\bm{q} \right| \leq C(h^{\frac{1}{2}} + h^{\frac{q-d}{2q}}) \| \nabla_{x, \bm{q}} \varphi^m_i \|_{L^2(\mathcal{O})},
\end{equation}
where $q \in (2,\infty)$ when $d=2$ and $q \in (3,6]$ when $d=3$. By the renormalization property $\zeta(\rho^{m,n})$ satisfies
\begin{equation*}
\int_0^T \int_\Omega \left( \frac{ \partial \zeta(\rho^{m,n})}{\partial t} \eta - \zeta(\rho^{m,n}) \bm{v}^{m,n} \cdot \nabla_x \eta \right) \,\mathrm{d}x\,\mathrm{d}t = 0.
\end{equation*}
Taking $\eta = \chi_{[s,s+h]} \int_D M^m \hat{\psi}^{m,n}(s)\varphi_i^m \,\mathrm{d}\bm{q}$ in the above equation we obtain that
\begin{equation*}
\int_{\mathcal{O}} M^m [\zeta(\rho^{m,n})(s+h) - \zeta(\rho^{m,n})(s)] \hat{\psi}^{m,n} \varphi_i^m \,\mathrm{d}x\,\mathrm{d}\bm{q} = \int_s^{s+h} \int_{\mathcal{O}} M^m \zeta(\rho^{m,n})\bm{v}^{m,n} \cdot \nabla_x(\hat{\psi}^{m,n}(s) \varphi^m_i) \,\mathrm{d}x\,\mathrm{d}\bm{q}\,\mathrm{d}t.
\end{equation*}
By H\"{o}lder's inequality we obtain that
\begin{equation}\label{eq135.1}
\begin{split}
&\left| \int_{\mathcal{O}} M^m [\zeta(\rho^{m,n})(s+h) - \zeta(\rho^{m,n})(s)] \hat{\psi}^{m,n} \varphi^m_i \,\mathrm{d}x\,\mathrm{d}\bm{q} \right| \\
&\leq C(M, \zeta_{\max}) h^{\frac{1}{2}} \| \bm{v}^{m,n} \|_{L^2(s,s+h; L^q(\Omega))} \| \nabla_x (\hat{\psi}^{m,n}(s) \varphi^m_i) \|_{L^{\frac{q}{q-1}}(\mathcal{O})} \\
&\leq C(M,\zeta_{\max}) h^{\frac{1}{2}} \| \bm{v}^{m,n} \|_{L^2(s,s+h; L^q(\Omega))} \bigg(\| \nabla_x \hat{\psi}^{m,n}(s) \|_{L^2(\mathcal{O})} \| \varphi^m_i \|_{L^{\frac{2q}{q-2}}(\mathcal{O})} \\
&\quad + \| \nabla_x \varphi^m_i \|_{L^2(\mathcal{O})} \| \hat{\psi}^{m,n}(s) \|_{L^{\frac{2q}{q-2}}(\mathcal{O})} \bigg),
\end{split}
\end{equation}
where $q \in \left(2, 2(K+1) \right]$ when $d=2$ and $q \in (3, 6]$ when $d=3$. Since
\begin{align*}
(\zeta(\rho^{m,n}) \hat{\psi}^{m,n})(s+h) &- (\zeta(\rho^{m,n}) \hat{\psi}^{m,n})(s) \\
&= \zeta(\rho^{m,n})(s+h) [\hat{\psi}^{m,n}(s+h) - \hat{\psi}^{m,n}(s)] + [\zeta(\rho^{m,n})(s+h) - \zeta(\rho^{m,n})(s)] \hat{\psi}^{m,n}(s),
\end{align*}
it follows from (\ref{eq132}) and (\ref{eq135.1}) that
\begin{align*}
\left| \int_{\mathcal{O}} M^m \zeta(\rho^{m,n})(s+h) [\hat{\psi}^{m,n}(s+h) - \hat{\psi}^{m,n}(s)] \varphi^m_i \,\mathrm{d}x\,\mathrm{d}\bm{q} \right| \leq C(h^{\frac{1}{2}} + h^{\frac{q-d}{2q}}) \| \nabla_{x, \bm{q}} \varphi^m_i \|_{L^2(\mathcal{O})} \\
\qquad + Ch^{\frac{1}{2}} \bigg(\| \nabla_x \hat{\psi}^{m,n}(s) \|_{L^2(\mathcal{O})} \| \varphi^m_i \|_{L^{\frac{2q}{q-2}}(\mathcal{O})} + \| \nabla_x \varphi^m_i \|_{L^2(\mathcal{O})} \| \hat{\psi}^{m,n}(s) \|_{L^{\frac{2q}{q-2}}(\mathcal{O})} \bigg),
\end{align*}
where $q \in (2, 2(K+1)]$ when $d=2$ and $q \in (3,6]$ when $d=3$. Taking $\varphi^m_i = \hat{\psi}^{m,n}(s+h) - \hat{\psi}^{m,n}(s)$ in the above inequality and integrating with respect to $s$ we obtain that
\begin{equation*}
\begin{split}
&\int_0^{T-h} \int_{\mathcal{O}} M^m \zeta(\rho^{m,n})(s+h) [\hat{\psi}^{m,n}(s+h) - \hat{\psi}^{m,n}(s)]^2 \,\mathrm{d}x\,\mathrm{d}\bm{q}\,\mathrm{d}s \\
&\leq 2C(h^{\frac{1}{2}} + h^{\frac{q-d}{2q}}) \left( \int_0^{T}  \| \nabla_{x, \bm{q}} \hat{\psi}^{m,n}(s) \|^2_{L^2(\mathcal{O})} \,\mathrm{d}s\right)^{\frac{1}{2}} \\
&\quad + 2Ch^{\frac{1}{2}} \left( \int_0^{T} \| \nabla_x \hat{\psi}^{m,n}(s) \|^2_{L^2(\mathcal{O})} \,\mathrm{d}s \right)^{\frac{1}{2}} \left(\int_0^{T} \| \hat{\psi}^{m,n}(s) \|^2_{L^{\frac{2q}{q-2}}(\mathcal{O})} \,\mathrm{d}s\right)^{\frac{1}{2}} \\
&\quad + 2Ch^{\frac{1}{2}} \left( \int_0^{T} \| \nabla_x \hat{\psi}^{m,n}(s) \|^2_{L^2(\mathcal{O})} \,\mathrm{d}s \right)^{\frac{1}{2}} \left(\int_0^{T} \| \hat{\psi}^{m,n}(s) \|^2_{L^{\frac{2q}{q-2}}(\mathcal{O})} \,\mathrm{d}s\right)^{\frac{1}{2}} \\
&\leq Ch^{\frac{q-d}{2q}},
\end{split}
\end{equation*}
where $q \in (2, 2(K+1)]$ when $d=2$ and $q \in (3,6]$ when $d=3$. The above inequality follows from H\"{o}lder's inequality, the uniform estimate (\ref{eq201}) and the Sobolev embedding $W^{1,2}(\mathcal{O}) \hookrightarrow L^p(\mathcal{O})$, $p \in \left[1,\frac{2(K+1)d}{(K+1)d - 2} \right]$. Since $\zeta(\rho^{m,n}) \geq \zeta_{\min}$ and $M^m \geq 1/m$, we have that
\begin{equation*}
\| \hat{\psi}^{m,n}(\cdot+h) - \hat{\psi}^{m,n}(\cdot) \|_{L^2(0, T-h; L^2(\mathcal{O}))} \leq Ch^\gamma,
\end{equation*}
where $0 < \gamma \leq \frac{K}{4(K+1)}$ when $d=2$ and $0 < \gamma \leq 1/8$ when $d=3$. We have therefore shown the following Nikolski\u{\i} norm estimate:
\begin{equation}\label{eq227.1}
\| \hat{\psi}^{m,n} \|_{N^\gamma_2(0,T; L^2(\mathcal{O}))} \coloneqq \sup_{0 < h < T} h^{-\gamma} \| \hat{\psi}^{m,n}(\cdot+h) - \hat{\psi}^{m,n}(\cdot) \|_{L^2(0, T-h; L^2(\mathcal{O}))} \leq C,
\end{equation}
where $0 < \gamma \leq \frac{K}{4(K+1)}$ when $d=2$ and $0 < \gamma \leq 1/8$ when $d=3$.

From (\ref{eq201}), (\ref{eq226}) and (\ref{eq227.1}) and by using the Aubin--Lions Lemma we deduce that there exists a subsequence (not relabelled) such that
\begin{align}
\label{eq208} \hat{\psi}^{m,n} &\rightharpoonup \hat{\psi}^{m} \qquad \qquad \qquad \quad \quad \text{weakly in $L^2(0,T; W^{1,2}(\mathcal{O}))$}, \\
\label{eq209} \partial_t (M^m \zeta(\rho^{m,n}) \hat{\psi}^{m,n}) &\rightharpoonup \partial_t (M^m \zeta(\rho^m) \hat{\psi}^m) \qquad \, \text{weakly in $L^2(0,T; (W^{1,2}(\mathcal{O}))')$}, \\
\label{eq210} \hat{\psi}^{m,n} &\to \hat{\psi}^m \qquad \qquad \qquad \quad \quad \text{strongly in $L^2(0,T; L^2(\mathcal{O}))$}.
\end{align}
It then follows from \eqref{eq196}, \eqref{eq210} and the global Lipschitz continuity of the mapping $s \in \mathbb{R} \mapsto [s]_+ \in \mathbb{R}_{\geq 0}$ that
\[\varrho^{m,n} \rightarrow \varrho^m \quad \text{strongly in $L^2(0,T; L^2(\Omega))$},\]
where 
\begin{align*}
\varrho^m = \zeta(\rho^m) \, \int_D M^m \, [\hat{\psi}^m]_+ \, \mathrm{d}\bm{q},
\end{align*}
and therefore, thanks to \eqref{rhomn-strong} and the assumption (\ref{eq26}) on $\mu$, the Dominated Convergence Theorem implies that
\begin{align*}
\mu(\rho^{m,n},\varrho^{m,n}) &\to \mu(\rho^m,\varrho^m) \qquad \text{strongly in $L^1(Q)$}.
\end{align*}
Hence, because $0 \leq \mu_{\min} \leq \mu(\cdot,\cdot) \leq \mu_{\max} < \infty$, it follows that
\begin{align}
\label{eq195} \mu(\rho^{m,n},\varrho^{m,n}) &\to \mu(\rho^m,\varrho^m) \qquad \text{strongly in $L^p(Q)$},
\end{align}
for any $p \in [1,\infty)$.

\par
Now the above convergence results (\ref{eq189}), (\ref{eq191}) and (\ref{eq194}) for $\rho^{m,n}$, (\ref{eq195}) for $\mu(\rho^{m,n},\varrho^{m,n})$, (\ref{eq196}) for $\zeta(\rho^{m,n})$, (\ref{eq208}), (\ref{eq209}) and (\ref{eq210}) for $\hat{\psi}^{m,n}$, (\ref{eq186}) and (\ref{eq187}) for $c^{m,n}$, (\ref{eq188}) for $\bm{v}^{m,n}$ and (\ref{eq171}) for $\tau^{m,n}$ enable us to pass to the limit $n \to \infty$ in (\ref{eq48})--(\ref{eq49}) to obtain the following:
\begin{align}
\label{eq211} &\langle \partial_t \rho^{m}, \eta \rangle - (\bm{v}^{m}\rho^{m}, \nabla_x \eta) = 0,\quad \text{for all $\eta \in C^{0,1}(\overline\Omega)$ and a.e. $t \in (0,T)$},\\
\begin{split}
\label{eq212} &\langle \partial_t (\rho^{m} \bm{v}^{m}), \bm{w}_i \rangle - (\rho^{m} \bm{v}^{m} \otimes \bm{v}^{m}, \nabla_x \bm{w}_i) + (\mu(\rho^{m},\varrho^m)D(\bm{v}^{m}), \nabla_x \bm{w}_i)\\
&\quad = -(\tau^{m}, \nabla_x \bm{w}_i) + (\rho^{m}\bm{f}^m, \bm{w}_i) \quad \text{for all $i = 1,\dots,m$ and a.e. $t \in (0,T)$}, 
\end{split}
\end{align}
and
\begin{align}
\begin{aligned}
\label{eq213} 
&\hspace{-6mm}\left\langle \partial_t(M^m \zeta(\rho^{m})\hat{\psi}^{m}), \varphi \right\rangle_{\mathcal{O}} -  \left(M^m \zeta(\rho^{m}) \bm{v}^{m} \hat{\psi}^{m}, \nabla_x \varphi \right)_{\mathcal{O}} -  \left(M \zeta(\rho^{m}) \Lambda_\ell(\hat{\psi}^{m})(\nabla_x \bm{v}^{m}) \bm{q}, \nabla_{\bm{q}} \varphi \right)_{\mathcal{O}} \\
&\quad  + (M^m \nabla_x \hat{\psi}^{m}, \nabla_x \varphi)_{\mathcal{O}} + \left( M^m \mathbb{A}(\nabla_{\bm{q}} \hat{\psi}^{m}), \nabla_{\bm{q}} \varphi \right)_{\mathcal{O}} = 0 \quad  \text{for all $\varphi \in W^{1,2}(\mathcal{O})$ and a.e. $t \in (0,T)$}.
\end{aligned}
\end{align}
It is obvious that $\bm{v}^m(x,0) = \bm{v}^m_0(x)$. We can deduce using standard calculations that
\begin{equation*}
\lim_{t \to 0+} \| \hat{\psi}^m(\cdot, t) - T_\ell(\hat{\psi}^m_0(\cdot)) \|_{L^2(\mathcal{O})} = 0.
\end{equation*}
Also, from (\ref{eq196}) we deduce that
\begin{equation}
\zeta(\rho^{m,n}) \to \zeta(\rho^m) \qquad \text{a.e. in $Q$}.
\end{equation}
Similarly, from (\ref{eq210}) we deduce that
\begin{equation*}
\hat{\psi}^{m,n} \to \hat{\psi}^m \qquad \text{a.e. in $\mathcal{O} \times (0,T)$}.
\end{equation*}
Therefore,
\begin{equation*}
\zeta(\rho^{m,n}) \hat{\psi}^{m,n} \to \zeta(\rho^m) \hat{\psi}^m \qquad \text{a.e. in $\mathcal{O} \times (0,T)$}.
\end{equation*}
Finally, the Dominated Convergence Theorem gives that
\begin{equation}\label{eq218}
\tau^{m} = -k \int_D \left[ K M \zeta(\rho^{m})T_\ell(\hat{\psi}^{m})I + \sum_{j=1}^K \zeta(\rho^{m}) T_\ell(\hat{\psi}^{m})\nabla_{\bm{q}^j} M \otimes \bm{q}^j \right] \,\mathrm{d}\bm{q} \quad \text{a.e. in $Q$}.
\end{equation}

\subsection{Passage to the limit with $m$}
%
%
%
%
Since $\bm{v}^m \in L^2(0,T; W^{1,2}(\Omega; \mathbb{R}^d)) \hookrightarrow L^1(0,T; W^{1,1}(\Omega; \mathbb{R}^d))$ and the initial data $\rho^m(0) = \rho_0^m \in C^1(\overline\Omega) \subset L^\infty(\Omega)$, we deduce from Theorem VI.1.6 in \cite{MR2986590} that there exists a unique weak solution $\rho^m \in L^\infty(\Omega \times (0,T))$ that solves (\ref{eq211}). Moreover, by Theorem VI.1.3 in \cite{MR2986590} we have that $\rho^m \in C([0,T]; L^p(\Omega))$, where $1 \leq p < \infty$. The solution satisfies
\begin{equation}\label{eq220}
\sup_{t \in (0,T)} \| \rho^m(t) \|_{L^\infty(\Omega)} \leq \| \rho_0^m \|_{L^\infty(\Omega)} \leq \rho_{\max},
\end{equation}
and $\rho^m \geq \rho_{\min}$. By the renormalization property, for any $\beta \in C^1(\mathbb{R})$, $\beta(\rho^m)$ satisfies
\begin{equation}\label{eq222a}
\frac{\partial \beta(\rho^{m})}{\partial t} + \divergence_x (\beta(\rho^{m}) \bm{v}^{m}) = 0
\end{equation}
in the distributional sense. We note that $\zeta(\rho^{m})$ is a renormalized solution in the sense that \eqref{eq222a} is satisfied with $\beta=\zeta$.

We would now like to show that $\hat{\psi}^m \geq 0$ a.e. in $\mathcal{O} \times (0,T)$. Before doing so, we shall revisit the question of regularity of $\rho^m$ and will show that, thanks to the regularity of $\bm{v}^m$ and $\rho_0^m$, in fact $\partial_t \rho^m \in L^\infty(Q)$ and $\nabla_x \rho^m \in L^\infty(Q)$, and therefore also $\partial_t \zeta(\rho^m) \in L^\infty(Q)$ and $\nabla_x \zeta(\rho^m) \in L^\infty(Q)$, whereby the renormalized equation \eqref{eq222a} does in fact hold in the sense that 
\begin{align}\label{eq222}
 \partial_t \zeta(\rho^m) + \nabla_x \cdot (\bm{v}^m\zeta(\rho^m)) = 0 \qquad \mbox{a.e. on $Q$}.
 \end{align}

To this end we note that after passing to the limit $n \rightarrow \infty$ in \eqref{eq-vmn}, 
\begin{align}\label{eq-vmm}
\bm{v}^{m}(x,t) \coloneqq \sum_{i=1}^m c^{m}_i(t)\, \bm{w}_i(x),
\end{align}
where $c^{m,n}_i \rightharpoonup c_i^m \in W^{1,\infty}((0,T))$ for each $m \geq 1$ as $n \rightarrow \infty$ by \eqref{eq186}
and $\bm{w}_i \in C^1(\overline\Omega;\mathbb{R}^d)$ for each $i=1,\ldots,m$. Thanks to the orthonormality of $(\bm{w}_i)_{i=1}^m$ in $L^2(\Omega;\mathbb{R}^d)$ it therefore follows that
\[ \|\bm{v}^m(\cdot)\|_{L^2(\Omega)} = \left(\sum_{i=1}^m|c_i^m(\cdot)|^2\right)^{\frac{1}{2}} \in W^{1,\infty}((0,T)).\]
For $x \in \overline\Omega$ and $t \in [0,T]$, consider the following initial-value problem for a system of ordinary differential equations:
\begin{align}\label{eq-ode}
\begin{aligned}
\frac{\mathrm{d}X^m}{\mathrm{d}s}(x,t;s) &= \bm{v}^m(X^m(x,t;s),s),\qquad s \in [0,T],\\
X^m(x,t;t) &=x.
\end{aligned}
\end{align}
By the Cauchy--Lipschitz theorem, whose assumptions are satisfied thanks to the regularity properties of $c_i^m$ and $\bm{w}_i$, $i=1,\ldots,m$, we deduce that, for each $(x,t) \in \overline \Omega \times [0,T]$, the system \eqref{eq-ode} has a unique solution $s \in [0,T] \mapsto X^m(x,t;s) \in C^1([0,T];\mathbb{R}^d)$. Furthermore, thanks to the specific form of \eqref{eq-vmm} and the fact that $\bm{w}_i \in C^1(\overline\Omega;\mathbb{R}^n)$ for $i=1,\ldots,m$, it follows that $X^m$ is a differentiable function of $x$, and, for every $x \in \overline{\Omega}$ and $t \in [0,T]$, 
\begin{align*} 
\frac{\mathrm{d}}{\mathrm{d}s} (\nabla_x X^m(x,t;s)) &= \nabla_x \bm{v}^m(X^m(x,t;s),s)\, \nabla_x X^m(x,t;s), \qquad s \in [0,T],\\
\nabla_x X^m(x,t;t) &= I,
\end{align*}
where $I$ is the $d \times d$ identity matrix. Thus, upon integration and recalling the form of $\bm{v}^m$ from \eqref{eq-vmm}, we deduce that
\begin{align*} 
\nabla_x X^m(x,t;s) &= \nabla_x X^m(x,t;t)\\
&\quad + \int_t^s \sum_{i=1}^m c_i^m(\sigma)\, (\nabla_x \bm{w}_i)(X^m(x,t;\sigma))\, \nabla_x X^m(x,t;\sigma)\, \mathrm{d}\sigma.
\end{align*}
By Liouville's formula for matrix-differential equations, and recalling that the functions $\bm{w}_i$, $i=1,2,\ldots$ are divergence-free, whereby the matrix $\nabla_x \bm{v}^m$ is trace-free, we have that
\[ \mathrm{det}\left(\nabla_x X^m(x,t;s) \right)
= \mathrm{det}\left(\nabla_x X^m(x,t;t)\right)
\mathrm{exp}\left(\int_t^s \mathrm{tr}[(\nabla_x \bm{v}^m)(X^m(x,t;\sigma),\sigma)]\,\mathrm{d}\sigma\right) = \mathrm{det}(I)\, \mathrm{exp}(0)=1\]
for all $x \in \overline{\Omega}$ and all $t,s \in [0,T]$. Thus, the mapping $x \in \overline\Omega \mapsto X^m(x,t;s) \in \overline\Omega$ has a nonvanishing Jacobian for all $x \in \overline{\Omega}$ and all $t,s \in [0,T]$
and is therefore an invertible $C^1$ bijection from $\overline\Omega$ onto $\overline\Omega$, with $C^1$ inverse $y \in \overline\Omega \mapsto X^m(y;s;t)$ for all $t,s \in [0,T]$. In what follows we require a bound on $\nabla_x X^m$.

To this end, let $\|\cdot\|$ be a matrix norm on $\mathbb{R}^{d \times d}$ induced by a vector norm on $\mathbb{R}^d$. Then, $\|I\|=1$ and therefore, because $X^m(x,t;t)=I$, also
\begin{align*} 
\|\nabla_x X^m(x,t;s)\| &\leq  1 + \left|\int_t^s \sum_{i=1}^m |c_i^m(\sigma)|\, \|(\nabla_x \bm{w}_i)(X^m(x,t;\sigma))\|\, \|\nabla_x X^m(x,t;\sigma)\|\, \mathrm{d}\sigma \right|\\
&\leq  1 + \left|\int_t^s \sum_{i=1}^m |c_i^m(\sigma)|\, \|\nabla_x \bm{w}_i\|_{C(\overline\Omega)}\, \|\nabla_x X^m(x,t;\sigma)\|\, \mathrm{d}\sigma \right|\\
&\leq  1 + \left|\int_t^s \left(\sum_{i=1}^m |c_i^m(\sigma)|^2 \right)^{\frac{1}{2}}\, \left(\sum_{i=1}^m \|\nabla_x \bm{w}_i\|^2_{C(\overline\Omega)}\right)^{\frac{1}{2}}\, \|\nabla_x X^m(x,t;\sigma)\|\, \mathrm{d}\sigma \right|\\
& =  1 + \left(\sum_{i=1}^m \|\nabla_x \bm{w}_i\|^2_{C(\overline\Omega)}\right)^{\frac{1}{2}} \left|\int_t^s \|\bm{v}^m(\sigma)\|_{L^2(\Omega)}\, \|\nabla_x X^m(x,t;\sigma)\|\, \mathrm{d}\sigma \right|
\\
& \leq   1 + \left(\sum_{i=1}^m \|\nabla_x \bm{w}_i\|^2_{C(\overline\Omega)}\right)^{\frac{1}{2}} \|\bm{v}^m\|_{L^\infty(0,T;L^2(\Omega))}\,\left|\int_t^s  \|\nabla_x X^m(x,t;\sigma)\|\, \mathrm{d}\sigma \right|.
\end{align*}
Thus, by Gronwall's lemma, we have that
\[ \|\nabla_x X^m(x,t;s)\| \leq \mathrm{exp}\left(|s-t| \left(\sum_{i=1}^m \|\nabla_x \bm{w}_i\|^2_{C(\overline\Omega)}\right)^{\frac{1}{2}} \|\bm{v}^m \|_{L^\infty(0,T;L^2(\Omega))}\right).\]
Consequently,
\[ \max_{x \in \overline\Omega,\, t,s \in [0,T]}\|\nabla_x X^m(x,t;s)\| \leq C_m:= \mathrm{exp}\left(T\left(\sum_{i=1}^m \|\nabla_x \bm{w}_i\|^2_{C(\overline\Omega)}\right)^{\frac{1}{2}} \|\bm{v}^m \|_{L^\infty(0,T;L^2(\Omega))}\right).\]
We shall show that the unique weak solution $\rho^m \in L^\infty(\Omega \times (0,T))$
of (\ref{eq211}) is given by the expression
\[\rho^m(x,t) = \rho_0^m(X^m(x,t;0)), \qquad x \in \overline\Omega,\, t \in [0,T].\]
Indeed, for any $\varphi \in \mathcal{D}(\Omega):=C^\infty_0(\Omega)$, we have as equalities in $\mathcal{D}'(0,T)$, that
\begin{align*}
&~_{~~_{\mathcal{D}'(\Omega)}\!\!}\left \langle \partial_t \rho^m(\cdot,t), \varphi \right\rangle _{_{\!\mathcal{D}(\Omega)}} = \frac{\mathrm{d}}{\mathrm{d}t}\hspace{-3mm} ~_{~~_{\mathcal{D}'(\Omega)}\!\!}\left \langle \rho^m(\cdot,t), \varphi \right\rangle _{_{\!\mathcal{D}(\Omega)}}\\ 
&\qquad\qquad= \frac{\mathrm{d}}{\mathrm{d}t} \int_\Omega \rho^m(x,t)\,\varphi(x)\,\mathrm{d}x = \frac{\mathrm{d}}{\mathrm{d}t} \int_\Omega \rho^m_0(X^m(x,t;0))\,\varphi(x)\,\mathrm{d}x\\
&\qquad\qquad=\frac{\mathrm{d}}{\mathrm{d}t} \int_\Omega \rho^m_0(y)\,\varphi(X^m(y,0;t))\,
\mathrm{det}[(\nabla_y X^m)(y,0;t)]\,\mathrm{d}y\\
&\qquad\qquad=\frac{\mathrm{d}}{\mathrm{d}t} \int_\Omega \rho^m_0(y)\,\varphi(X^m(y,0;t))\,\mathrm{d}y = \int_\Omega \rho^m_0(y)\,(\nabla_x \varphi)(X^m(y,0;t))\cdot \frac{\mathrm{d}X^m}{\mathrm{d}t}(y,0;t)\,\mathrm{d}y.
\end{align*}
Hence,
\begin{align*}
&~_{~~_{\mathcal{D}'(\Omega)}\!\!}\left \langle \partial_t \rho^m, \varphi \right\rangle _{_{\!\mathcal{D}(\Omega)}} = \int_\Omega \rho^m_0(y)\,(\nabla_x \varphi)(X^m(y,0;t))\cdot \bm{v}^m(X^m(y,0;t))\,\mathrm{d}y\\
&\qquad\qquad  =\int_\Omega \rho^m_0(X^m(x,t,0))\,\nabla_x \varphi(x)\cdot\bm{v}^m(x)\,\mathrm{det}[(\nabla_x X^m)(x,t;0)]\,\mathrm{d}x\\
&\qquad\qquad=\int_\Omega \rho^m_0(X^m(x,t,0))\,\nabla_x \varphi(x)\cdot \bm{v}^m(x)\,\mathrm{d}x\\
&\qquad\qquad=\int_\Omega \rho^m(x,t)\,\nabla_x \varphi(x)\cdot \bm{v}^m(x)\,\mathrm{d}x\\
&\qquad\qquad=\int_\Omega \rho^m(x,t)\,\bm{v}^m(x) \cdot \nabla_x \varphi(x)\,\mathrm{d}x\\ 
&\qquad\qquad = \!\!\!\!\!\!~_{~~_{\mathcal{D}'(\Omega)}\!\!}\left \langle 
\rho^m(\cdot,t)\,\bm{v}^m ,  \nabla_x \varphi 
\right\rangle _{_{\!\mathcal{D}(\Omega)}}.
\end{align*}
Thus, 
\[~_{~~_{\mathcal{D}'(\Omega)}\!\!}\left \langle \partial_t \rho^m, \varphi \right\rangle _{_{\!\mathcal{D}(\Omega)}} =   \!\!\!\!\!\!~_{~~_{\mathcal{D}'(\Omega)}\!\!}\left \langle 
\rho^m(\cdot,t)\,\bm{v}^m ,  \nabla_x \varphi 
\right\rangle _{_{\!\mathcal{D}(\Omega)}} =  \!\!\!\!\!\!~_{~~_{\mathcal{D}'(\Omega)}\!\!}\left \langle -
\nabla_x \cdot (\rho^m(\cdot,t)\,\bm{v}^m) , \varphi 
\right\rangle _{_{\!\mathcal{D}(\Omega)}},\]
as equalities in $\mathcal{D}'(0,T)$. Thus we have shown that $\rho^m$ defined by 
$\rho^m(x,t) = \rho_0^m(X^m(x,t;0))$ for $x \in \overline\Omega$ and $t \in [0,T]$ satisfies the equation 
\[ \frac{\partial\rho^m}{\partial t} + \nabla_x \cdot (\bm{v}^m \rho^m) = 0\]
in $\mathcal{D}'(\Omega \times (0,T))$; in addition $\rho^m(x,0)=\rho_0^m(x)$ for all $x \in \Omega$. By the uniqueness of the weak solution to the
problem (\ref{eq211}) asserted by Theorem VI.1.3 in \cite{MR2986590}, it follows that the weak solution in question is given by $\rho^m(x,t) = \rho_0^m(X^m(x,t;0))$ for $x \in \overline\Omega$ and $t \in [0,T]$.

Using this representation of $\rho^m$ we are now in a position to show that, in addition to the regularity $\rho^m \in C([0,T]; L^p(\Omega))$, where $1 \leq p < \infty$, guaranteed by Theorem VI.1.3 in \cite{MR2986590}, in our case, thanks to the regularity of $\bm{v}^m$ and $\rho_0^m$, $\rho^m$ has additional regularity. 

Indeed, because $\rho_0^m \in C^1(\overline\Omega)$ and the mapping $x \in \overline \Omega \mapsto X^m(x,t;0) \in \overline\Omega$ is a $C^1$ mapping, it follows by the chain rule that 
\[ \nabla_x \rho^m(x,t) =\nabla_x\, [\rho_0^m (X^m(x,t;0))] =  [(\nabla_x X^m)(x,t;0)]^{\rm T}\, (\nabla \rho_0^m)(X^m(x,t;0))\]
for all $x \in \overline\Omega$ and all $t \in [0,T]$,  whereby
\[ \|\nabla_x \rho^m\|_{L^\infty(Q)} \leq C_m \|\nabla \rho_0^m\|_{L^\infty(\Omega)} \leq C(m).\]
As $\bm{v}^m \in W^{1,\infty}(Q;\mathbb{R}^d)$, it then also follows that
\[ \left\|\frac{\partial \rho^m}{\partial t}\right\|_{L^\infty(Q)} \leq \|\bm{v}^m\|_{L^\infty(Q)} \|\nabla_x \rho^m\|_{L^\infty(Q)} \leq 
C_m \|\nabla_x \rho^m_0\|_{L^\infty(\Omega)} \|\bm{v}^m\|_{L^\infty(Q)} \leq C(m).\]
With these bounds it follows that 
\[ \frac{\partial \rho^m}{\partial t} + \nabla_x\cdot  (\bm{v}^m \rho^m) = \frac{\partial \rho^m}{\partial t} + \bm{v}^m \cdot \nabla_x \rho^m = 0\qquad \mbox{a.e. in $Q$}, \] 
and $\rho^m(x,0)=\rho_0^m(x)$ for all $x \in \Omega$. Hence, the renormalized equation also holds in the sense that
\begin{align}\label{eq-renorm-new} 
\frac{\partial \beta(\rho^m)}{\partial t} + \nabla_x\cdot  (\bm{v}^m \beta(\rho^m)) = \frac{\partial \beta(\rho^m)}{\partial t} + \bm{v}^m \cdot \nabla_x \beta(\rho^m) = 0\qquad \mbox{a.e. in $Q$}, 
\end{align} 
for any $\beta \in C^1(\mathbb{R})$,
and $\beta(\rho^m(x,0))=\beta(\rho_0^m(x))$
for all $x \in \Omega$.

We are now ready to return to the proof of the nonnegativity of $\hat\psi^m$. By following a similar procedure as in (\ref{eq178})--(\ref{eq200}), but now using the renomalization property in the form \eqref{eq-renorm-new}, and
bearing in mind that 
\begin{align}\label{eq-good} 
\nabla_x \rho^m \in L^\infty(Q) \quad \mbox{and} \quad \frac{\partial \rho^m}{\partial t}  \in L^\infty(Q),
\end{align}
we obtain that
\begin{align*}
\frac{\mathrm{d}}{\,\mathrm{d}t} \int_{\mathcal{O}} M^m \zeta(\rho^{m,n}) ((\hat{\psi}^{m})_-)^2 \,\mathrm{d}x\,\mathrm{d}\bm{q} + \int_{\mathcal{O}} M^m |\nabla_{x,\bm{q}} (\hat{\psi}^{m})_-|^2 \,\mathrm{d}x\,\mathrm{d}\bm{q} \\
\leq C(m,\ell) \int_{\mathcal{O}} M^m \zeta(\rho^{m,n}) ((\hat{\psi}^{m})_-)^2 \,\mathrm{d}x\,\mathrm{d}\bm{q}.
\end{align*}


Since $\hat{\psi}^m(0) = T_\ell(\hat{\psi}^m_0) \geq 0$ and $\zeta(\cdot) \geq \zeta_{\min} > 0$, Gronwall's inequality implies that $(\hat{\psi}^m)_- \equiv 0$ in $\mathcal{O} \times (0,T)$. Thus we deduce that
\begin{equation*}
\hat{\psi}^m \geq 0 \quad \text{a.e. in $\mathcal{O} \times (0,T)$}.
\end{equation*}
\par
Next, we will derive $m$-independent estimates for $\bm{v}^m$ and $\hat{\psi}^m$. We define
\begin{align*}
&\mathcal{F}_{\delta}(s) \coloneqq (s + \delta) \log(s + \delta) + 1, \quad \mathcal{F}(s) \coloneqq s \log s + 1, \\
&T_{\delta, \ell}(s) \coloneqq \int_0^s \frac{\Lambda_{\ell}(t)}{t + \delta} \,\mathrm{d}t = \int_0^s \frac{t \Gamma_\ell(t)}{t + \delta} \,\mathrm{d}t.
\end{align*}
We set $\varphi \coloneqq \log(\hat{\psi}^m + \delta) + 1$ in (\ref{eq213}), where $\delta > 0$ is arbitrary, to obtain the following identity:
\begin{equation}\label{eq227}
\begin{split}
&\frac{\mathrm{d}}{\,\mathrm{d}t} \int_{\mathcal{O}} M^m \zeta(\rho^m) \mathcal{F}_{\delta}(\hat{\psi}^m) \,\mathrm{d}x\,\mathrm{d}\bm{q} - \int_{\mathcal{O}} M^m \partial_t \zeta(\rho^m) (\delta \log(\hat{\psi}^m + \delta) + 1 - \hat{\psi}^m) \,\mathrm{d}x\,\mathrm{d}\bm{q} \\
&\quad- \int_{\mathcal{O}} M^m \zeta(\rho^m) \frac{\hat{\psi}^m}{\hat{\psi}^m + \delta} (\bm{v}^m \cdot \nabla_x \hat{\psi}^m) \,\mathrm{d}x\,\mathrm{d}\bm{q} + \int_{\mathcal{O}} \frac{M^m}{\hat{\psi}^m + \delta} \nabla_x \hat{\psi}^m : \nabla_x \hat{\psi}^m \,\mathrm{d}x\,\mathrm{d}\bm{q} \\
&\quad+ \int_{\mathcal{O}} \frac{M^m}{\hat{\psi}^m + \delta} \mathbb{A}(\nabla_{\bm{q}} \hat{\psi}^m) : \nabla_{\bm{q}} \hat{\psi}^m \,\mathrm{d}x\,\mathrm{d}\bm{q} = \left( M\zeta(\rho^m) (\nabla_x \bm{v}^m) \bm{q}, \nabla_{\bm{q}} T_{\delta, \ell}(\hat{\psi}^m) \right)_{\mathcal{O}}.
\end{split}
\end{equation}
Since $\hat{\psi}^m \in W^{1,2}(\mathcal{O})$, we can test the equation (\ref{eq222})  with the function
\begin{equation*}
\eta = \int_D M^m (\delta \log(\hat{\psi}^m + \delta) + 1 - \hat{\psi}^m) \,\mathrm{d}\bm{q}, \quad \nabla_x \eta = - \int_D \frac{M^m \hat{\psi}^m \nabla_x \hat{\psi}^m}{\hat{\psi}^m + \delta} \,\mathrm{d}\bm{q}
\end{equation*}
to deduce that the second term and the third term in (\ref{eq227}) add up to $0$. Next, we integrate the resulting equation with respect to time over $(0,t)$ and use the assumption (\ref{eq1.14}), to obtain that
\begin{equation*}
\begin{split}
&\int_{\mathcal{O}} M^m \zeta(\rho^m(\cdot, t)) \mathcal{F}_{\delta}(\hat{\psi}^m(\cdot,t)) \,\mathrm{d}x\,\mathrm{d}\bm{q} + C_1\int_0^t \int_{\mathcal{O}} \frac{M^m}{\hat{\psi}^m + \delta} |\nabla_{x,\bm{q}} \hat{\psi}^m|^2 \,\mathrm{d}x\,\mathrm{d}\bm{q}\,\mathrm{d}s \\
&\quad\leq \int_{\mathcal{O}} M^m \zeta(\rho_0^m) \mathcal{F}_{\delta}(T_\ell(\hat{\psi}^m_0)) \,\mathrm{d}x\,\mathrm{d}\bm{q} + \int_0^t \left( M \zeta(\rho^m) (\nabla_x \bm{v}^m) \bm{q}, \nabla_{\bm{q}} T_{\delta, \ell}(\hat{\psi}^m) \right)_{\mathcal{O}} \,\mathrm{d}s.
\end{split}
\end{equation*}
Taking the limit $\delta \to 0_+$ in the above inequality, the first term on the left-hand side and the first-term on the right-hand side can be easily dealt with. For the second term on the left-hand side we apply the Monotone Convergence Theorem. Therefore we get
\begin{equation}\label{eq230}
\begin{split}
&\int_{\mathcal{O}} M^m \zeta(\rho^m(\cdot, t)) \mathcal{F}(\hat{\psi}^m(\cdot,t)) \,\mathrm{d}x\,\mathrm{d}\bm{q} + 4C_1\int_0^t \int_{\mathcal{O}} M^m \left| \nabla_{x,\bm{q}} \sqrt{\hat{\psi}^m} \right|^2 \,\mathrm{d}x\,\mathrm{d}\bm{q}\,\mathrm{d}s \\
&\quad \leq \int_{\mathcal{O}} M^m \zeta(\rho_0^m) \mathcal{F}(T_\ell(\hat{\psi}^m_0)) \,\mathrm{d}x\,\mathrm{d}\bm{q} + \limsup_{\delta \to 0+} \int_0^t \left( M \zeta(\rho^m) (\nabla_x \bm{v}^m) \bm{q}, \nabla_{\bm{q}} T_{\delta, \ell}(\hat{\psi}^m) \right)_{\mathcal{O}} \,\mathrm{d}s.
\end{split}
\end{equation}
For the last term on the right-hand side we use integration by parts (the boundary term vanishes since $M = 0$ on $\partial D$) and note that $\divergence_x \bm{v}^m = 0$. We obtain
\begin{equation}\label{eq231a}
\begin{split}
\int_0^t \left( M \zeta(\rho^m) (\nabla_x \bm{v}^m) \bm{q}, \nabla_{\bm{q}} T_{\delta, \ell}(\hat{\psi}^m) \right)_{\mathcal{O}} \,\mathrm{d}s &= -\int_0^t \left( \divergence_{\bm{q}} \left(M \zeta(\rho^m)(\nabla_x \bm{v}^m) \bm{q}\right), T_{\delta, \ell}(\hat{\psi}^m) \right)_{\mathcal{O}} \,\mathrm{d}s \\
&= - \sum_{j=1}^K \int_0^t \left( \zeta(\rho^m) (\nabla_x \bm{v}^m), T_{\delta, \ell}(\hat{\psi}^m) \nabla_{\bm{q}^j} M \otimes \bm{q}^j \right)_{\mathcal{O}} \,\mathrm{d}s.
\end{split}
\end{equation}
Since $|T_{\delta,\ell}(s)- T_\ell(s)| \leq \delta \log\left(1 + \frac{\ell}{\delta}\right)$ for all $s \in [0,\infty)$, using \eqref{eq231a} in \eqref{eq230} we can pass to the limit in (\ref{eq230}) to obtain
\begin{equation}\label{eq232}
\begin{split}
&\int_{\mathcal{O}} M^m \zeta(\rho^m(\cdot, t)) \mathcal{F}(\hat{\psi}^m(\cdot,t)) \,\mathrm{d}x\,\mathrm{d}\bm{q} + 4C_1\int_0^t \int_{\mathcal{O}} M^m \left| \nabla_{x,\bm{q}} \sqrt{\hat{\psi}^m} \right|^2 \,\mathrm{d}x\,\mathrm{d}\bm{q}\,\mathrm{d}s \\
&\quad \leq \int_{\mathcal{O}} M^m \zeta(\rho_0^m) \mathcal{F}(T_\ell(\hat{\psi}^m_0)) \,\mathrm{d}x\,\mathrm{d}\bm{q} - \sum_{j=1}^K \int_0^t \left( \zeta(\rho^m) (\nabla_x \bm{v}^m), T_{\ell}(\hat{\psi}^m) \nabla_{\bm{q}^j} M \otimes \bm{q}^j \right)_{\mathcal{O}} \,\mathrm{d}s.
\end{split}
\end{equation}
\par
We multiply the $i$-th equation in (\ref{eq212}) by $c_i^m(t)$ and sum with respect to $i = 1,\dots,m$ to deduce the following energy identity:
\begin{equation}\label{eq233}
\frac{1}{2}\frac{\mathrm{d}}{\,\mathrm{d}t} \int_\Omega \rho^{m}|\bm{v}^{m}|^2 \,\mathrm{d}x + \int_\Omega \mu(\rho^{m},\varrho^m) |D(\bm{v}^{m})|^2 \,\mathrm{d}x = - \int_\Omega \tau^{m} : \nabla_x \bm{v}^m \,\mathrm{d}x+ (\rho^{m} \bm{f}^m, \bm{v}^{m}).
\end{equation}
Using $\divergence_x \bm{v}^m = 0$ and (\ref{eq218}) we have that
\begin{equation*}
\int_\Omega \tau^m : \nabla_x \bm{v}^m \,\mathrm{d}x = -k \sum_{j=1}^K \left( \zeta(\rho^m) (\nabla_x \bm{v}^m), T_{\ell}(\hat{\psi}^m) \nabla_{\bm{q}^j} M \otimes \bm{q}^j \right)_{\mathcal{O}}.
\end{equation*}
Integrating (\ref{eq233}) with respect to time over $(0,t)$ and multiplying (\ref{eq232}) by $k$ and adding the results  we get
\begin{equation}\label{eq255}
\begin{split}
&k\int_{\mathcal{O}} M^m \zeta(\rho^m(\cdot, t)) \mathcal{F}(\hat{\psi}^m(\cdot,t)) \,\mathrm{d}x\,\mathrm{d}\bm{q} + \frac{1}{2}\int_\Omega \rho^m(\cdot,t) |\bm{v}^m(\cdot,t)|^2 \,\mathrm{d}x \\
&\quad + \int_0^t \int_\Omega \mu(\rho^{m},\varrho^{m}) |D(\bm{v}^{m})|^2 \,\mathrm{d}x\,\mathrm{d}s + 4kC_1\int_0^t \int_{\mathcal{O}} M^m \left| \nabla_{x,\bm{q}} \sqrt{\hat{\psi}^m} \right|^2 \,\mathrm{d}x\,\mathrm{d}\bm{q}\,\mathrm{d}s \\
&\leq k\int_{\mathcal{O}} M^m \zeta(\rho_0^m) \mathcal{F}(T_\ell(\hat{\psi}^m_0)) \,\mathrm{d}x\,\mathrm{d}\bm{q} + \frac{1}{2}\int_\Omega \rho_0^m |\bm{v}_0|^2 \,\mathrm{d}x + \int_0^t \int_\Omega \rho^m \bm{f}^m \cdot \bm{v}^m \,\mathrm{d}x \,\mathrm{d}s.
\end{split}
\end{equation}
Using the assumption (\ref{eq26}), Korn's inequality (\ref{eq1.26}) and Gronwall's inequality we arrive at the following uniform estimate:
\begin{equation}\label{eq236}
\begin{split}
&k\zeta_{\min} \sup_{t \in (0,T)} \| M^m \mathcal{F}(\hat{\psi}^m(\cdot,t)) \|_{L^1(\mathcal{O})} + \frac{\rho_{\min}}{2} \sup_{t \in (0,T)} \| \bm{v}^m(\cdot, t) \|^2_{L^2(\Omega; \mathbb{R}^d)} \\
&\quad + \frac{\mu_{\min}c_0}{2} \int_0^T \| \bm{v}^m \|^2_{W^{1,2}(\Omega; \mathbb{R}^d)}\,\mathrm{d}t + 4kC_1\int_0^T \int_{\mathcal{O}} M^m \left| \nabla_{x,\bm{q}} \sqrt{\hat{\psi}^m} \right|^2 \,\mathrm{d}x\,\mathrm{d}\bm{q}\,\mathrm{d}t \\
&\leq k\zeta_{\max} \int_{\mathcal{O}} M^m 
\mathcal{F}(T_\ell(\hat{\psi}^m_0)) \,\mathrm{d}x\,\mathrm{d}\bm{q} 
+ \frac{\rho_{\max}}{2} \| \bm{v}_0 \|^2_{L^2(\Omega; \mathbb{R}^d)} 
+ \frac{\rho_{\max}^2}{2\mu_{\min}c_0} \int_0^T \| \bm{f}^m \|^2_{L^2(\Omega; \mathbb{R}^d)} \,\mathrm{d}t \\
&\leq C(\ell, \zeta_{\max}, \rho_{\max}, \hat{\psi}_0, \bm{v}_0, \bm{f}).
\end{split}
\end{equation}
By interpolation we have the following estimate for $\bm{v}^m$:
\begin{equation}\label{eq237}
\| \bm{v}^m \|_{L^{\frac{2(d+2)}{d}}(Q; \mathbb{R}^d)} \leq C(\ell).
\end{equation} 
It directly follows from (\ref{eq26}), (\ref{eq212}), (\ref{eq236}) and (\ref{eq237}) that
\begin{equation}\label{eq258}
\int_0^T \| \partial_t (\rho^m \bm{v}^m) \|^2_{W^{-1,2}(\Omega; \mathbb{R}^d)} \,\mathrm{d}t \leq C(\ell).
\end{equation}
By a similar calculation as in (\ref{eq108}) we have that
\begin{equation}\label{eq259.1}
\left\| \partial_t \rho^{m}\right\|_{L^2(0,T; W^{-1,p}(\Omega))} \leq C(\ell), \\
\end{equation}
where $p \in [1,\infty)$ when $d=2$ and $p \in [1,6]$ when $d=3$. Thanks to the presence of the truncation function $T_\ell(\cdot)$ we have that
\begin{equation}\label{eq244}
|\tau^{m}| \leq C(\ell).
\end{equation}

Next, we shall focus on the fractional time derivative of $\bm{v}^m$. Integrating (\ref{eq212}) with respect to time over $(s, s+h)$, with $s < T-h$, we have that
\begin{equation}\label{eq1.216}
\begin{split}
\int_\Omega [(\rho^m\bm{v}^m)(s+h) - (\rho^m \bm{v}^m)(s)] \cdot \bm{\omega}_i \,\mathrm{d}x &= \int_s^{s+h} \int_\Omega \rho^m \bm{v}^m \otimes \bm{v}^m : \nabla_x \bm{\omega}_i \,\mathrm{d}x\,\mathrm{d}t \\
&\quad - \int_s^{s+h} \int_\Omega \mu(\rho^m,\varrho^m) D(\bm{v}^m):D(\bm{\omega}_i) \,\mathrm{d}x\,\mathrm{d}t \\
&\quad - \int_s^{s+h} \int_\Omega \tau^m : \nabla_x \bm{\omega}_i \,\mathrm{d}x\,\mathrm{d}t + \int_s^{s+h} \rho^m \bm{f}^m \cdot \bm{\omega}_i \,\mathrm{d}x\,\mathrm{d}t \\
&\eqqcolon I_1 + I_2 + I_3 + I_4.
\end{split}
\end{equation}
For $I_1$, we have by H\"{o}lder's inequality, the Gagliardo--Nirenberg inequality (\ref{eq1.25}) and Korn's inequality (\ref{eq1.26}) that 
\begin{equation*}
\begin{split}
|I_1| &\leq \rho_{\max} \left(\int_s^{s+h} \| \bm{v}^m \|_{L^q(\Omega; \mathbb{R}^d)} \| \bm{v}^m \|_{L^{\frac{2q}{q-2}}(\Omega; \mathbb{R}^d)} \,\mathrm{d}t \right) \| \nabla_x \bm{\omega}_i \|_{L^2(\Omega; \mathbb{R}^{d \times d})} \\
&\leq C\rho_{\max}\left( \int_s^{s+h} \| \bm{v}^m \|_{L^q(\Omega;\mathbb{R}^d)} \| \bm{v}^m \|^{1 - \frac{\mathrm{d}}{q}}_{L^2(\Omega; \mathbb{R}^d)} \| \bm{v}^m \|^{\frac{\mathrm{d}}{q}}_{W^{1,2}(\Omega; \mathbb{R}^d)} \,\mathrm{d}t \right)  \| \nabla_x \bm{\omega}_i \|_{L^2(\Omega; \mathbb{R}^{d \times d})} \\
&\leq C\rho_{\max}  \| \nabla_x \bm{\omega}_i \|_{L^2(\Omega; \mathbb{R}^{d \times d})} \| \bm{v}^m \|^{1 - \frac{\mathrm{d}}{q}}_{L^\infty(0,T; L^2(\Omega; \mathbb{R}^d))} \left( \int_s^{s+h} \| \bm{v}^m \|^2_{W^{1,2}(\Omega; \mathbb{R}^d)} \,\mathrm{d}t \right)^{\frac{\mathrm{d}}{2q}} \\
&\quad \times \left( \int_s^{s+h} \| \bm{v}^m \|^{\frac{2q}{2q-d}}_{L^q(\Omega; \mathbb{R}^d)} \,\mathrm{d}t \right)^{\frac{2q-d}{2q}} \\
&\leq C\rho_{\max}  \| \nabla_x \bm{\omega}_i \|_{L^2(\Omega; \mathbb{R}^{d \times d})} \| \bm{v}^m \|^{1 - \frac{\mathrm{d}}{q}}_{L^\infty(0,T; L^2(\Omega; \mathbb{R}^d))} \| \bm{v}^m \|^{\frac{\mathrm{d}}{q}}_{L^2(s,s+h; W^{1,2}(\Omega; \mathbb{R}^d))} \| \bm{v}^m \|_{L^2(s, s+h; L^q(\Omega; \mathbb{R}^d))} \\
&\quad \times \left( \int_s^{s+h} 1^{\frac{2q-d}{q-d}}\,\mathrm{d}t \right)^{\frac{q-d}{2q}} \\
&\leq Ch^{\frac{q-d}{2q}} \| D(\bm{\omega}_i) \|_{L^2(\Omega; \mathbb{R}^{d \times d})},
\end{split}
\end{equation*}
where $q \in (2,\infty)$ when $d=2$ and $q \in (3,6]$ when $d=3$. For $I_2$, we have by (\ref{eq236}) that
\begin{equation*}
|I_2| \leq \mu_{\max}h^{\frac{1}{2}} \| D(\bm{v}^m) \|_{L^2(s,s+h; L^2(\Omega; \mathbb{R}^{d \times d}))} \| D(\bm{\omega}_i) \|_{L^2(\Omega; \mathbb{R}^{d \times d})} \leq Ch^{\frac{1}{2}} \| D(\bm{\omega}_i) \|_{L^2(\Omega; \mathbb{R}^{d \times d})}.
\end{equation*}
For $I_3$, we have by (\ref{eq244}) that
\begin{equation*}
|I_3| \leq Ch^{\frac{1}{2}} \| D(\bm{\omega}_i) \|_{L^2(\Omega; \mathbb{R}^{d \times d})}.
\end{equation*}
For $I_4$, we have by H\"{o}lder's inequality and Korn's inequality (\ref{eq1.26}) that
\begin{equation*}
|I_4| \leq \rho_{\max} h^{\frac{1}{2}} \| \bm{f}^m \|_{L^2(s,s+h; L^2(\Omega; \mathbb{R}^d))} \| \bm{\omega}_i \|_{L^2(\Omega; \mathbb{R}^{d \times d})} \leq Ch^{\frac{1}{2}} \| D(\bm{\omega}_i) \|_{L^2(\Omega; \mathbb{R}^{d \times d})}.
\end{equation*}
Hence,
\begin{equation}\label{eq117.1}
\left| \int_\Omega [(\rho^m\bm{v}^m)(s+h) - (\rho^m \bm{v}^m)(s)] \cdot \bm{\omega}_i \,\mathrm{d}x \right| \leq C(h^{\frac{1}{2}} + h^{\frac{q-d}{2q}}) \| D(\bm{\omega}_i) \|_{L^2(\Omega; \mathbb{R}^{d \times d})},
\end{equation}
where $q \in (2,\infty)$ when $d=2$ and $q \in (3,6]$ when $d=3$. Next, we take the test function $\eta = \chi_{[s, s+h]} (\bm{v}^m(s) \cdot \bm{\omega}_i)$ in 
\begin{equation*}
\int_0^T \int_\Omega \left( \frac{\partial \rho^m}{\partial t} \eta - \rho^m \bm{v}^m \cdot \nabla_x \eta \right) \,\mathrm{d}x\,\mathrm{d}t = 0,
\end{equation*}
to deduce that
\begin{equation*}
\int_\Omega [\rho^m(s+h) - \rho^m(s)] (\bm{v}^m(s) \cdot \bm{\omega}_i) \,\mathrm{d}x = \int_s^{s+h} \int_\Omega \rho^m \bm{v}^m \cdot \nabla_x (\bm{v}^m(s) \cdot \bm{\omega}_i) \,\mathrm{d}x\,\mathrm{d}t.
\end{equation*}
By H\"{o}lder's inequality we have that
\begin{equation}\label{eq120.1}
\begin{split}
\left| \int_\Omega [\rho^m(s+h) - \rho^m(s)] (\bm{v}^m(s) \cdot \bm{\omega}_i) \,\mathrm{d}x \right| &\leq \rho_{\max} h^{\frac{1}{2}} \| \bm{v}^m \|_{L^2(s,s+h; L^q(\Omega; \mathbb{R}^d))} \| \nabla_x (\bm{v}^m(s) \cdot \bm{\omega}_i) \|_{L^{\frac{q}{q-1}}(\Omega; \mathbb{R}^d)} \\
&\leq \rho_{\max} h^{\frac{1}{2}} \| \bm{v}^m \|_{L^2(s,s+h; L^q(\Omega; \mathbb{R}^d))} \\
&\quad \times \bigg(\| \nabla_x \bm{v}^m(s) \|_{L^2(\Omega; \mathbb{R}^{d \times d})} \| \bm{\omega}_i \|_{L^{\frac{2q}{q-2}}(\Omega; \mathbb{R}^d)} \\
&\quad \quad + \| \nabla_x \bm{\omega}_i \|_{L^2(\Omega; \mathbb{R}^{d \times d})} \| \bm{v}^m(s) \|_{L^{\frac{2q}{q-2}}(\Omega; \mathbb{R}^d)} \bigg).
\end{split}
\end{equation}
Since
\begin{equation*}
(\rho^m \bm{v}^m)(s+h) - (\rho^m \bm{v}^m)(s) = \rho^m(s+h)[\bm{v}^m(s+h) - \bm{v}^m(s)] + [\rho^m(s+h) - \rho^m(s)]\bm{v}^m(s),
\end{equation*}
it follows from (\ref{eq117.1}) and (\ref{eq120.1}) that
\begin{align*}
\begin{aligned}
&\left| \int_\Omega \rho^m(s+h)[\bm{v}^m(s+h) - \bm{v}^m(s)] \cdot \bm{\omega}_i \,\mathrm{d}x \right| \leq C(h^{\frac{1}{2}} + h^{\frac{q-d}{2q}}) \| D(\bm{\omega}_i) \|_{L^2(\Omega; \mathbb{R}^{d \times d})} \\
&\quad + Ch^{\frac{1}{2}} \bigg(\| \nabla_x \bm{v}^m(s) \|_{L^2(\Omega; \mathbb{R}^{d \times d})} \| \bm{\omega}_i \|_{L^{\frac{2q}{q-2}}(\Omega; \mathbb{R}^d)} + \| \nabla_x \bm{\omega}_i \|_{L^2(\Omega; \mathbb{R}^{d \times d})} \| \bm{v}^m(s) \|_{L^{\frac{2q}{q-2}}(\Omega; \mathbb{R}^d)} \bigg),
\end{aligned}
\end{align*}
where $q \in (2,\infty)$ when $d=2$ and $q \in (3,6]$ when $d=3$. Taking $\bm{\omega}_i = \bm{v}^m(s+h) - \bm{v}^m(s)$ in the above inequality and integrating with respect to $s$ we obtain that
\begin{equation*}
\begin{split}
&\left| \int_0^{T-h} \int_\Omega \rho^m(s+h) [\bm{v}^m(s+h) - \bm{v}^m(s)]^2 \,\mathrm{d}x\,\mathrm{d}s \right| \\
&\leq 2C(h^{\frac{1}{2}} + h^{\frac{q-d}{2q}}) \left( \int_0^{T} \| D(\bm{v}^m(s)) \|^2_{L^2(\Omega; \mathbb{R}^{d \times d})} \,\mathrm{d}s \right)^{\frac{1}{2}} \\
&\quad + 2Ch^{\frac{1}{2}} \left( \int_0^{T} \| \nabla_x \bm{v}^m(s) \|^2_{L^2(\Omega; \mathbb{R}^{d \times d})} \,\mathrm{d}s\right)^{\frac{1}{2}} \left( \int_0^{T} \| \bm{v}^m(s) \|^2_{L^{\frac{2q}{q-2}}(\Omega; \mathbb{R}^d)} \,\mathrm{d}s \right)^{\frac{1}{2}} \\
&\quad + 2Ch^{\frac{1}{2}}\left( \int_0^{T} \| \nabla_x \bm{v}^m(s) \|^2_{L^2(\Omega; \mathbb{R}^{d \times d})} \,\mathrm{d}s\right)^{\frac{1}{2}} \left( \int_0^{T} \| \bm{v}^m(s) \|^2_{L^{\frac{2q}{q-2}}(\Omega; \mathbb{R}^d)} \,\mathrm{d}s \right)^{\frac{1}{2}} \\
&\leq Ch^{\frac{q-d}{2q}},
\end{split}
\end{equation*}
where $q \in (2,\infty)$ when $d=2$ and $q \in (3,6]$ when $d=3$. The above inequality follows from H\"{o}lder's inequality, the uniform estimate (\ref{eq236}) and the Sobolev embedding $W^{1,2}(\Omega) \hookrightarrow L^p(\Omega)$, $p \in [1,\infty)$ when $d=2$ and $p \in [1,6]$ when $d=3$. Since $\rho^m \geq \rho_{\min}$, we have
\begin{equation*}
\| \bm{v}^m(\cdot+h) - \bm{v}^m(\cdot) \|_{L^2(0, T-h; L^2(\Omega; \mathbb{R}^d))} \leq Ch^\gamma,
\end{equation*}
where $0 < \gamma < 1/4$ when $d=2$ and $0 < \gamma \leq 1/8$ when $d=3$. Therefore, we obtain the following Nikolski\u{\i} norm estimate:
\begin{equation}\label{eq261}
\| \bm{v}^m \|_{N^\gamma_2(0,T; L^2(\Omega; \mathbb{R}^d))} \coloneqq \sup_{0 < h < T} h^{-\gamma} \| \bm{v}^m(\cdot+h) - \bm{v}^m(\cdot) \|_{L^2(0, T-h; L^2(\Omega; \mathbb{R}^d))} \leq C,
\end{equation}
where $0 < \gamma < 1/4$ when $d=2$ and $0 < \gamma \leq 1/8$ when $d=3$.
Using the $m$-independent estimates (\ref{eq220}), (\ref{eq236}), (\ref{eq261}), (\ref{eq237}), (\ref{eq258}), (\ref{eq259.1}) and (\ref{eq244}) and the Aubin--Lions Lemma we deduce the existence of subsequences (not relabelled) such that, as $m \to \infty$,
\begin{align}\label{eq263.0}
\rho^{m} &\rightharpoonup \rho \qquad \quad \qquad \text{weak* in $L^\infty(\Omega \times (0,T))$},\\
\label{eq263.1} \partial_t \rho^m &\rightharpoonup \partial_t \rho \qquad \qquad \, \text{weakly in $L^2(0,T; W^{-1,p}(\Omega))$}, \\
\tau^{m} &\rightharpoonup \tau \qquad \quad \qquad \text{weak* in $L^\infty(\Omega \times (0,T); \mathbb{R}^{d \times d})$}, \label{eq263.2}\\
\bm{v}^m &\rightharpoonup \bm{v} \qquad \quad \qquad \text{weak* in $L^\infty(0,T; L^2(\Omega; \mathbb{R}^d))$}, \label{eq263.3}\\
\label{eq248} \bm{v}^m &\rightharpoonup \bm{v} \qquad \quad \qquad \text{weakly in $L^2(0,T; W^{1,2}(\Omega; \mathbb{R}^d))$}, \\
\partial_t (\rho^m \bm{v}^m) &\rightharpoonup \partial_t (\rho \bm{v}) \qquad \ \, \, \, \text{weakly in $L^2(0,T; W^{-1,2}(\Omega;\mathbb{R}^d))$}, \label{eq248a}\\
\label{eq250} \bm{v}^m &\to \bm{v} \qquad \quad \qquad \text{strongly in $L^2(0,T; L^p(\Omega;\mathbb{R}^d))$},
\end{align}
where $p \in [1,\infty)$ when $d=2$ and $p \in [1,6)$ when $d=3$. By the strong convergence of $\bm{v}^m$, we can perform a similar argument as in (\ref{eq1.125})--(\ref{eq194}) to deduce that,
\begin{equation}\label{eq269.1}
\rho^{m} \to \rho \qquad \text{strongly in $L^p(\Omega \times (0,T))$}, 
\end{equation}
for any $p \in [1,\infty)$. With the convergence result (\ref{eq250}) for $\bm{v}^m$ we can perform a similar argument as in Theorem VI.1.9 in \cite{MR2986590} and strengthen the convergence above to get
\begin{equation}\label{rhom-strong}
\rho^m \to \rho \qquad \text{strongly in $C([0,T]; L^p(\Omega))$},
\end{equation}
for any $p \in [1,\infty)$. Therefore,
\begin{align}
\label{eq253} \zeta(\rho^{m}) &\to \zeta(\rho) \qquad \text{strongly in $C([0,T]; L^p(\Omega))$},
\end{align}
for any $p \in [1,\infty)$. \par
Next, we shall focus on the convergence properties of $\hat{\psi}^m$. First, we set $\varphi \equiv 1$ in (\ref{eq213}) and deduce that
\begin{equation*}
\begin{split}
0 &\leq \int_{\mathcal{O}} M^m(\bm{q})\zeta(\rho^m(x,t)) \hat{\psi}^m(x,\bm{q},t) \,\mathrm{d}x\,\mathrm{d}\bm{q} = \int_{\mathcal{O}} M^m\zeta(\rho_0^m) T_\ell(\hat{\psi}^m_0) \,\mathrm{d}x\,\mathrm{d}\bm{q} \\
&\leq \zeta_{\max} \int_{\mathcal{O}} M^m \hat{\psi}^m_0 \,\mathrm{d}x\,\mathrm{d}\bm{q} = \zeta_{\max} \int_{\mathcal{O}} M \hat{\psi}_0 \,\mathrm{d}x\,\mathrm{d}\bm{q} \leq C.
\end{split}
\end{equation*}
As $\hat\psi^m \geq 0$ a.e. on $\mathcal{O} \times (0,T)$, it follows that
\begin{equation}\label{varrho}
\varrho^m(x,t) = \zeta(\rho^{m}(x,t)) \int_D M^m(\bm{q})\, [\hat{\psi}^m(x,\bm{q},t)]_+ \,\mathrm{d}\bm{q} = \zeta(\rho^m(x,t)) \int_D M^m(\bm{q}) \hat{\psi}^m(x,\bm{q},t) \,\mathrm{d}\bm{q} \geq 0.
\end{equation}

By defining
\begin{equation}\label{rhomplus}
\lambda^m(x,t) \coloneqq \int_D M^m(\bm{q}) \hat{\psi}^m(x,\bm{q},t) \,\mathrm{d}\bm{q}
\end{equation}
and setting $\varphi(x,\bm{q},t) \coloneqq \bar{\varphi}(x,t)$ in (\ref{eq213}), we have the following equation satisfied by $\lambda^m$:
\begin{equation}\label{eq256}
\left\langle \partial_t (\zeta(\rho^m) \lambda^m), \bar{\varphi} \right\rangle - (\zeta(\rho^m) \lambda^m \bm{v}^m, \nabla_x \bar{\varphi}) + (\nabla_x \lambda^m, \nabla_x \bar{\varphi}) = 0 \quad \text{for all $\bar{\varphi} \in W^{1,2}(\Omega)$ and a.e. $t \in (0,T)$},
\end{equation}
supplemented by the initial condition $\lambda^m(0) \coloneqq \lambda^m_0$, where, thanks to \eqref{eq25},
\begin{equation*}
0 \leq \lambda^m_0(x) \coloneqq \int_D M^m T_\ell(\hat{\psi}^m_0) \,\mathrm{d}\bm{q} \leq \int_D M \hat{\psi}_0 \,\mathrm{d}\bm{q} = \frac{1}{\zeta(\rho_0)} \int_D \psi_0 \, \mathrm{d}\bm{q} \leq \frac{\varrho_{\max}}{\zeta_{\min}}.
\end{equation*}
Let $\omega \coloneqq \sup_{x \in \Omega} \lambda^m_0(x)$, test (\ref{eq222}) with the function $\eta = \omega \bar{\varphi}$, and subtract (\ref{eq256}) from the resulting equation; this gives that
\begin{equation}\label{eq256a}
\left\langle \partial_t (\zeta(\rho^m) (\omega-\lambda^m)), \bar{\varphi} \right\rangle - (\zeta(\rho^m) (\omega-\lambda^m) \bm{v}^m, \nabla_x \bar{\varphi}) + (\nabla_x (\omega-\lambda^m), \nabla_x \bar{\varphi}) = 0\qquad \forall\, \bar\varphi \in W^{1,2}(\Omega)
\end{equation}
for a.e. $t \in (0,T)$. Note that $\hat\psi^m \in L^2(0,T;W^{1,2}(\mathcal{O}))$
by \eqref{eq208}. It therefore follows from \eqref{rhomplus} that $\lambda^m \in L^2(0,T;W^{1,2}(\Omega))$. Hence, because 
$\zeta(\rho^m) \in W^{1,\infty}(Q)$, which follows from \eqref{eq-good}, we have that 
\[\zeta(\rho^m) (\omega-\lambda^m)
\in L^2(0,T;W^{1,2}(\Omega)).\]  
Furthermore, 
because $\zeta(\rho^m) (\omega-\lambda^m) \bm{v}^m \in L^2(0,T;L^2(\Omega))$ and $\nabla_x(\omega-\lambda^m) \in L^2(0,T;L^2(\Omega))$, it follows from \eqref{eq256} that 
\[ \partial_t(\zeta(\rho^m) (\omega-\lambda^m)) \in L^2(0,T;(W^{1,2}(\Omega))').\]
where $(W^{1,2}(\Omega))'$ is the dual space of $W^{1,2}(\Omega)$. In addition, since $\zeta(\rho^m) \in W^{1,\infty}(Q)$ and $\lambda^m \in L^2(0,T; W^{1,2}(\Omega))$, we have that $\varrho^m \in L^2(0,T; W^{1,2}(\Omega))$. By \eqref{eq209} and \eqref{eq-good} we also have that
$\partial_t(M^m \zeta(\rho^m)\hat\psi^m) \in L^2(0,T;(W^{1,2}(\mathcal{O}))')$, and therefore $\partial_t \varrho^m \in  L^2(0,T;(W^{1,2}(\Omega))')$. Thus, in summary, 
\begin{align}\label{eq-varrhobd} \varrho^m \in L^2(0,T;W^{1,2}(\Omega))\quad \mbox{and} \quad \partial_t \varrho^m \in L^2(0,T;(W^{1,2}(\Omega))').
\end{align}
We shall now show that $\partial_t \lambda^m \in L^2(0,T; (W^{1,2}(\Omega))')$. For any $\phi \in W^{1,2}(\Omega)$ and $\theta \in C_0^\infty((0,T))$, we have
\begin{align*}
 \int_0^T \langle \partial_t \lambda^m, \phi \rangle \, \theta \, \mathrm{d}t &= - \int_0^T\int_\Omega \lambda^m \, \phi \, \partial_t \theta \, \mathrm{d}x \, \mathrm{d}t \\
 &= - \int_0^T\int_\Omega \varrho^m \, \frac{\phi}{\zeta(\rho^m)} \, \partial_t \theta \, \mathrm{d}x \, \mathrm{d}t \\
 &= - \int_0^T \int_\Omega \varrho^m \, \phi \left[ \partial_t \left( \frac{\theta}{\zeta(\rho^m)} \right) - \theta \, \partial_t \left( \frac{1}{\zeta(\rho^m)} \right) \right] \mathrm{d}x \, \mathrm{d}t \\
 &= \int_0^T \left\langle \partial_t \varrho^m, \frac{\phi}{\zeta(\rho^m)} \right\rangle \, \theta \, \mathrm{d}t - \int_0^T \left( \int_\Omega \varrho^m \, \phi \, \frac{\zeta^\prime(\rho^m) \, \partial_t \rho^m}{\zeta(\rho^m)^2} \, \mathrm{d}x \right)\, \theta \, \mathrm{d}t.
\end{align*}
Note that $\partial_t \rho^m$ is integrable since $\rho^m \in W^{1,\infty}(Q)$ by \eqref{eq-good}. From the last equality, we deduce that
\begin{align*}
 \langle \partial_t \lambda^m, \phi \rangle = \left\langle \partial_t \varrho^m, \frac{\phi}{\zeta(\rho^m)} \right\rangle - \int_\Omega \varrho^m \, \phi \, \frac{\zeta^\prime(\rho^m) \, \partial_t \rho^m}{\zeta(\rho^m)^2} \, \mathrm{d}x \qquad \mbox{a.e. on $\Omega$}
\end{align*}
for all $\phi \in W^{1,2}(\Omega)$. It follows that
\begin{align*}
|\langle \partial_t \lambda^m, \phi \rangle| \leq \| \partial_t \varrho^m \|_{(W^{1,2}(\Omega))'} \left\| \frac{\phi}{\zeta(\rho^m)} \right\|_{W^{1,2}(\Omega)} + \| \varrho^m \|_{L^2(\Omega)} \left\| \phi \, \frac{\zeta^\prime(\rho^m) \, \partial_t \rho^m}{\zeta(\rho^m)^2} \right\|_{L^2(\Omega)}.
\end{align*}
Note that by \eqref{eq-varrhobd} we have that
$\| \partial_t \varrho^m \|_{(W^{1,2}(\Omega))'} \leq C(m)$ and $\| \varrho^m \|_{L^2(\Omega)} \leq C(m)$. 
Also, by using $\| \nabla_x \rho^m \|_{L^\infty(Q)} \leq C(m)$ and $\| \partial_t \rho^m \|_{L^\infty(Q)} \leq C(m)$, we obtain the following bounds
\begin{align*}
 \left\| \frac{\phi}{\zeta(\rho^m)} \right\|_{W^{1,2}(\Omega)} &\leq C(m) \| \phi \|_{W^{1,2}(\Omega)}, \\
 \left\| \phi \, \frac{\zeta^\prime(\rho^m) \, \partial_t \rho^m}{\zeta(\rho^m)^2} \right\|_{L^2(\Omega)} &\leq C(m) \| \phi \|_{W^{1,2}(\Omega)}.
\end{align*}
Hence,
\begin{align*}
| \langle \partial_t \lambda^m, \phi \rangle| \leq C(m) \| \phi \|_{W^{1,2}(\Omega)},
\end{align*}
which then implies that
\begin{align*}
\| \partial_t \lambda^m \|^2_{L^2(0,T; (W^{1,2}(\Omega))')} = \int_0^T \| \partial_t \lambda^m \|^2_{(W^{1,2}(\Omega))'} dt \leq C(m).
\end{align*}
We have shown that 
\begin{align*}
 \lambda^m \in L^2(0,T;W^{1,2}(\Omega))\quad \mbox{and} \quad \partial_t \lambda^m \in L^2(0,T;(W^{1,2}(\Omega))'),
\end{align*}
whereby, upon defining $\alpha^m:= \omega - \lambda^m$, also
\begin{align}\label{eq-alpha}
  \alpha^m \in L^2(0,T;W^{1,2}(\Omega))\quad \mbox{and} \quad \partial_t \alpha^m \in L^2(0,T;(W^{1,2}(\Omega))').
\end{align}
It then follows from \eqref{eq256a} that, for any $\overline\varphi \in W^{1,2}(\Omega)$ and any $\theta \in C^\infty_0((0,T))$, we have
\begin{align*}
 0&=-\int_0^T \int_\Omega \zeta(\rho^m)\,\alpha^m\,\overline\varphi \,\partial_t\theta \, \mathrm{d}x \mathrm{d}t - \int_0^T \int_\Omega \zeta(\rho^m)\, \alpha^m \,\bm{v}^m\cdot (\nabla_x \, \overline\varphi)\, \theta \,\mathrm{d}x \, \mathrm{d}t 
 + \int_0^T \int_\Omega \nabla_x \alpha^m\cdot (\nabla_x \overline\varphi)\, \theta\, \mathrm{d}x\, \mathrm{d}t\\
 &=
 -\int_0^T \int_\Omega \alpha^m \left[\partial_t(\zeta(\rho^m)\theta) - (\partial_t \zeta(\rho^m))\,\theta\right]\overline\varphi\,\mathrm{d}x\, \mathrm{d}t\\
&\quad - \int_0^T\int_\Omega \alpha^m \left[\nabla_x\cdot (\zeta(\rho^m)\,\bm{v}^m\, \overline\varphi) - (\nabla_x\cdot (\zeta(\rho^m)\,\bm{v}^m))\,\overline\varphi\right]\theta \, \mathrm{d}x\, \mathrm{d}t\\
&\quad  + \int_0^T \int_\Omega \nabla_x \alpha^m\cdot (\nabla_x \overline\varphi)\, \theta\, \mathrm{d}x\, \mathrm{d}t\\
&= -\int_0^T \int_\Omega \alpha^m \left[\partial_t(\zeta(\rho^m)\, \theta)\, \overline\varphi + \nabla_x \cdot(\zeta(\rho^m)\,\bm{v}^m\,\overline\varphi)\,\theta \right]\mathrm{d}x\, \mathrm{d}t\\
&\quad + \int_0^T \int_\Omega\alpha^m\left[\partial_t\zeta(\rho^m) +\nabla_x\cdot (\zeta(\rho^m)\,\bm{v}^m)\right]\overline\varphi\,\theta \,\mathrm{d}x\, \mathrm{d}t\\ 
&\quad + \int_0^T \int_\Omega \nabla_x\alpha^m\cdot (\nabla_x\overline\varphi)\,\theta\, \mathrm{d}x\, \mathrm{d}t\\
&= -\int_0^T \int_\Omega \alpha^m \left[\partial_t(\zeta(\rho^m)\, \theta)\, \overline\varphi + \nabla_x \cdot(\zeta(\rho^m)\,\bm{v}^m\,\overline\varphi)\,\theta \right]\mathrm{d}x\, \mathrm{d}t + \int_0^T \int_\Omega \nabla_x\alpha^m\cdot (\nabla_x\overline\varphi)\,\theta\, \mathrm{d}x\, \mathrm{d}t,
 \end{align*}
where in the transition to the last line we made use the fact that the renormalized equation \eqref{eq-renorm-new} satisfied by $\rho^m$, with $\beta=\zeta$, holds almost everywhere on $Q = \Omega \times (0,T)$. We then deduce from the last equality that 
\begin{align*}
\int_0^T \langle \partial_t \alpha^m , \zeta(\rho^m)\,\overline\varphi\rangle\, \theta \,\mathrm{d}t + \int_0^T \left(\int_\Omega \nabla_x \alpha^m \cdot \zeta(\rho^m)\,\bm{v}^m\,\overline\varphi\,\mathrm{d}x\right)\theta\,\mathrm{d}t
 + \int_0^T \left(\int_\Omega \nabla_x\alpha^m\cdot (\nabla_x\overline\varphi) \mathrm{d}x\right) \theta \,\mathrm{d}t = 0, 
\end{align*}
for all $\overline\varphi \in W^{1,2}(\Omega)$ and all $\theta \in C^\infty_0((0,T))$. 
Hence, also
\begin{align}\label{eq-alpha-rev}
\langle \partial_t \alpha^m , \zeta(\rho^m)\,\overline\varphi\rangle\, + \int_\Omega \nabla_x \alpha^m \cdot \zeta(\rho^m)\,\bm{v}^m\,\overline\varphi\,\mathrm{d}x 
 + \int_\Omega \nabla_x\alpha^m\cdot \nabla_x\overline\varphi\, \mathrm{d}x  = 0\qquad \mbox{a.e. on $Q$}
\end{align}
for all $\overline\varphi \in W^{1,2}(\Omega)$. Our objective is to show that $\alpha^m = \omega-\varrho^m \geq 0$ a.e. on $Q$. To this end, it seems tempting to take $\overline\varphi=[\alpha^m]_{-}$ in \eqref{eq-alpha-rev}; however, the calculation that would by use of the renormalized equation \eqref{eq-renorm-new} satisfied by $\rho^m$  with $\beta=\zeta$ result in the desired inequality
\[ \int_\Omega \zeta(\rho)\,([\alpha^m(x,t)]_{-})^2\, \mathrm{d}x \leq  \int_\Omega \zeta(\rho)\,([\alpha^m(x,0)]_{-})^2\, \mathrm{d}x = \int_\Omega \zeta(\rho)\,([\omega-\varrho^m_0(x)]_{-})^2\, \mathrm{d}x = 0, \]
which would then ultimately imply that $[\omega -\varrho^m(x,t)]_{-} = 0$, a.e. on $Q$, is difficult to justify rigorously. The main obstacle in the approach is the limited regularity of $\alpha^m$ in conjunction with the fact that the function $s \in \mathbb{R} \mapsto [s]_{+} \in \mathbb{R}_{+}$ is only Lipschitz continuous. An equivalent way of phrasing this would be to say that we define $G(s):=\frac{1}{2}([s]_+)^2$ for $s \in \mathbb{R}$, and we take as our test function in \eqref{eq-alpha-rev} the function $\overline\varphi = G'(\alpha^m)$. We shall overcome this difficulty by making a different choice of the function $G$. 

For $\delta \in (0,1)$, let
\[ G_\delta(s) := \left\{ \begin{array}{cr} 
                          \frac{s^2-\delta^2}{2\delta} + s (\log \delta -1) + 1 & \mbox{for $s \leq \delta$},\\
                          s (\log s - 1) + 1  & \mbox{for $s \geq \delta$}.
                         \end{array}
                         \right.\]
It then follows that
\[ G'_\delta(s) = \left\{ \begin{array}{cr} 
                          \frac{s}{\delta} + \log \delta -1 & \mbox{for $s \leq \delta$},\\
                          \log s & \mbox{for $s \geq \delta$},
                         \end{array}
                         \right.\] 
and
\[ G''_\delta(s) = \left\{ \begin{array}{cr} 
                          1/\delta& \mbox{for $s \leq \delta$},\\
                          1/s  & \mbox{for $s \geq \delta$}.
                         \end{array}
                         \right.\] 
Clearly $G_\delta \in C^{2,1}(\mathbb{R})$, $G_\delta(s) \geq 0$ for all $s \in \mathbb{R}$, $G_\delta(1)=0$, $G_\delta$ is strictly convex, and in addition
\[ G_\delta(s) \geq \left\{ \begin{array}{cr} 
                          \frac{s^2}{2\delta} & \mbox{for $s \leq 0$},\\
                          0  & \mbox{for $s \geq 0$}.
                         \end{array}
                         \right.\]
We shall choose $\overline\varphi = G_\delta'(\alpha^m)$ in \eqref{eq-alpha-rev}, 
and will at the end of the calculation pass to the limit $\delta \rightarrow 0_+$.
Hence, a.e. on $(0,T)$, 
\begin{align*}
\langle \partial_t \alpha^m , \zeta(\rho^m)\,G_\delta'(\alpha^m)\rangle  + \int_\Omega \nabla_x \alpha^m \cdot \zeta(\rho^m)\,\bm{v}^m\,G_\delta'(\alpha^m)\,\mathrm{d}x 
 + \int_\Omega \nabla_x\alpha^m\cdot \nabla_x G_\delta'(\alpha^m)\, \mathrm{d}x  = 0.
\end{align*}
Equivalently,
\begin{align*}
\langle \partial_t \alpha^m , \,G_\delta'(\alpha^m)\, \zeta(\rho^m)\rangle  + \int_\Omega \zeta(\rho^m)\,\bm{v}^m\cdot \nabla_x G_\delta(\alpha^m)\,\mathrm{d}x 
 + \int_\Omega G_\delta''(\alpha^m)\,|\nabla_x \alpha^m|^2\, \mathrm{d}x  = 0\qquad \mbox{a.e. on $(0,T)$}.
\end{align*}
Thus, after partial integration in the second term on the left-hand side noting that $\bm{v}^m|_{\partial\Omega \times (0,T)} = 0$ and $\mbox{div}\,\bm{v}^m=0$ on $Q$ we have that
\begin{align*}
\langle \partial_t \alpha^m , \,G_\delta'(\alpha^m)\, \zeta(\rho^m)\rangle  - \int_\Omega \left(\bm{v}^m \cdot \nabla_x \zeta(\rho^m)\right)G_\delta(\alpha^m)\,\mathrm{d}x 
 + \int_\Omega G_\delta''(\alpha^m)\,|\nabla_x \alpha^m|^2\, \mathrm{d}x  = 0\qquad \mbox{a.e. on $(0,T)$}.
\end{align*}
Next, we invoke the renormalized equation \eqref{eq-renorm-new} to transform the second integral on the left-hand side further, resulting in
\begin{align}\label{eq-Gdelta}
\langle \partial_t \alpha^m , \,G_\delta'(\alpha^m)\, \zeta(\rho^m)\rangle + \int_\Omega \left(\partial_t \zeta(\rho^m)\right)G_\delta(\alpha^m)\,\mathrm{d}x 
 + \int_\Omega G_\delta''(\alpha^m)\,|\nabla_x \alpha^m|^2\, \mathrm{d}x  = 0\qquad \mbox{a.e. on $(0,T)$}.
\end{align}
Our objective is now to show that the first term of the left-hand side can be rewritten as follows: 
\[ \langle \partial_t \alpha^m , G_\delta'(\alpha^m)\, \zeta(\rho^m)\rangle
 = \langle \partial_t G_\delta(\alpha^m),  \zeta(\rho^m)\rangle\qquad \mbox{a.e. on $(0,T)$}.
\]
To this end, we extend $\alpha^m$ by $0$ from $\Omega \times [0,T]$ to $\Omega \times \mathbb{R}$, and we mollify the resulting function, still denoted by $\alpha^m$, with respect to $t$; to be more specific, for a nonnegative function $\chi \in C^\infty_0((0,T))$ such that $\int_{\mathbb{R}} \chi(t)\, \mathrm{d}t = 1$ and $\varepsilon \in (0,1)$, we let $\chi_\varepsilon(t)=\frac{1}{\varepsilon}\chi(\frac{t}{\varepsilon})$, and define
\begin{align*} \alpha^m_\varepsilon := \alpha^m \ast_t \chi_\varepsilon \in C^\infty_0(\mathbb{R};W^{1,2}(\Omega)),
\end{align*}
where $\ast_t$ denotes convolution with respect to $t$. If then follows from \eqref{eq-alpha} that
\begin{align}\label{eq-ameps1} 
\|\alpha^m_\varepsilon -\alpha^m\|_{L^2(0,T;W^{1,2}(\Omega))} \rightarrow 0 \qquad \mbox{as $\varepsilon \rightarrow 0_+$},
\end{align}
and, by Young's inequality for convolutions, 
\begin{align}\label{eq-ameps2} \|\partial_t\alpha^m_\varepsilon\|_{L^2(0,T;W^{-1,2}(\Omega))} \leq \|\partial_t\alpha^m\|_{L^2(0,T;W^{-1,2}(\Omega))}.
\end{align}

Now, because $\alpha^m_\varepsilon \in C^\infty_0(\mathbb{R}; W^{1,2}(\Omega))$, it follows by the chain rule and an application of H\"older's inequality that $\partial_t G_\delta(\alpha^m_\varepsilon) = G'_\delta(\alpha^m_\varepsilon)\,\partial_t \alpha^m_\varepsilon \in C([0,T];L^2(\Omega))$, and therefore, for any $\theta \in C^\infty_0((0,T))$, 
\begin{align}\label{eq-el000}
 \begin{aligned}
\int_0^T \langle \partial_t G_\delta(\alpha^m_\varepsilon),  \zeta(\rho^m)\rangle \, \theta(t)\,\mathrm{d}t &= \int_0^T \int_\Omega \partial_t G_\delta(\alpha^m_\varepsilon)\, \zeta(\rho^m)\, \theta(t)\, \mathrm{d}x\,\mathrm{d}t\\
&=\int_0^T \int_\Omega G'_\delta(\alpha^m_\varepsilon)\,\partial_t\alpha^m_\varepsilon\, \zeta(\rho^m)\, \theta(t)\, \mathrm{d}x\,\mathrm{d}t\\
&= \int_0^T \int_\Omega G'_\delta(\alpha^m)\, \partial_t\alpha^m_\varepsilon\, \zeta(\rho^m)\, \theta(t)\, \mathrm{d}x\,\mathrm{d}t\\
&\quad + \int_0^T \int_\Omega (G'_\delta(\alpha^m_\varepsilon) - G'_\delta(\alpha^m))\,\partial_t\alpha^m_\varepsilon\, \zeta(\rho^m)\, \theta(t)\, \mathrm{d}x\,\mathrm{d}t\\
&=:\mathrm{T}_{1,\varepsilon} + \mathrm{T}_{2,\varepsilon}.
\end{aligned}
\end{align}
We shall now pass to the limit $\varepsilon \rightarrow 0_+$ in this equality. First, because, by \eqref{eq-ameps2}, $\partial_t \alpha^m_\varepsilon$ is, for each $m \geq 1$, uniformly bounded in $L^2(0,T;(W^{1,2}(\Omega))')$ and because $L^2(0,T;(W^{1,2}(\Omega))')$ is, by definition, the dual space of the normed linear space $L^2(0,T;W^{1,2}(\Omega))$, thanks to the Banach--Alaoglu theorem we can extract a weak* convergent subsequence, for which (without indicating the subsequence in our notation) we have by the uniqueness of the weak* limit that 
\[ \partial_t \alpha^m_\varepsilon \rightharpoonup^* \partial_t \alpha^m\qquad \mbox{in $L^2(0,T;(W^{1,2}(\Omega))')$}. \]
Because $G'_\delta(\alpha^m)\,\zeta(\rho^m)\,\theta \in L^2(0,T;W^{1,2}(\Omega))$, the predual of $L^2(0,T;(W^{1,2}(\Omega))')$ it therefore follows that
\[ \lim_{\varepsilon \rightarrow 0_+} \mathrm{T}_{1,\varepsilon} = \int_0^T \langle \partial_t\alpha^m, G'_\delta(\alpha^m)\, \zeta(\rho^m)\rangle \, \theta(t)\, \mathrm{d}t.\]
Next, we will show that $\mathrm{T}_{2,\varepsilon} \rightarrow 0$ as $\varepsilon \rightarrow 0_+$. We begin by noting that 
\begin{align} \label{eq-el00}
\begin{aligned}
 |\mathrm{T}_{2,\varepsilon}| &\leq \|\partial_t\alpha^m_\varepsilon\|_{L^2(0,T;(W^{1,2}(\Omega))')}\, \|\zeta(\rho^m)\, \theta\, (G'_\delta(\alpha^m_\varepsilon) - G'_\delta(\alpha^m))\|_{L^2(0,T;W^{1,2}(\Omega))}\\
 &\leq \sqrt{2} \left(\|\zeta(\rho^m)\,\theta\|_{L^\infty(Q)}^2 + \|\nabla_x(\zeta(\rho^m))\,\theta\|_{L^\infty(Q)}^2\right)^{\frac{1}{2}}\|G'_\delta(\alpha^m_\varepsilon) - G'_\delta(\alpha^m)\|_{L^2(0,T;W^{1,2}(\Omega))}.
\end{aligned}
\end{align}
It therefore remains to show that 
\begin{align} \label{eq-el0}
\|G'_\delta(\alpha^m_\varepsilon) - G'_\delta(\alpha^m)\|_{L^2(0,T;W^{1,2}(\Omega))} \rightarrow 0\quad \mbox{as $\varepsilon \rightarrow 0_+$}.
\end{align}
First, observe that 
\begin{align}\label{eq-el1} 
\begin{aligned}
\|G'_\delta(\alpha^m_\varepsilon) - G'_\delta(\alpha^m)\|_{L^2(0,T;L^2(\Omega))} &\leq \|G''_\delta\|_{L^\infty(\mathbb{R})}
 \|\alpha^m_\varepsilon- \alpha^m\|_{L^2(0,T;L^2(\Omega))}\\
 & \leq \frac{1}{\delta}  \|\alpha^m_\varepsilon- \alpha^m\|_{L^2(0,T;L^2(\Omega))} \rightarrow 0\qquad \mbox{as $\varepsilon \rightarrow 0_+$}.
\end{aligned}
\end{align}
Next we shall show that $\|\nabla_x(G'_\delta(\alpha^m_\varepsilon)) - \nabla_x(G'_\delta(\alpha^m))\|_{L^2(0,T;L^{2}(\Omega))} \rightarrow 0$. We begin by observing that
\[ \nabla_x(G'_\delta(\alpha^m_\varepsilon)) = G''_\delta(\alpha^m_\varepsilon)\,\nabla_x \alpha^m_\varepsilon.\]
Thanks to the strong conference \eqref{eq-ameps1}, for $m$ fixed, there exists a subsequence (with respect to $\varepsilon$, not indicated) such that $\alpha^m_\varepsilon \rightarrow \alpha^m$ a.e. on $Q$. Because $G''_\delta \in C^{0,1}(\mathbb{R})$, it then follows that $G''_\delta(\alpha^m_\varepsilon) \rightarrow G''_\delta(\alpha^m)$ a.e. on $Q$. Furthermore, $0 \leq G''_\delta(\alpha^m_\varepsilon) \leq 1/\delta$ a.e. on $Q$. In addition, thanks to a `converse' of the Dominated Convergence Theorem, which asserts that each subsequence of a strongly convergent sequence in $L^1(Q)$ contains a dominated sub-subsequence (c.f., for example, Theorem 1 in \cite{MR0140657}), because by
\eqref{eq-ameps1} $|\alpha^m_\varepsilon - \alpha^m|^2 + |\nabla_x \alpha^m_\varepsilon - \nabla_x \alpha^m|^2 \rightarrow 0$ in $L^1(Q)$ as $\varepsilon \rightarrow 0_+$, there exists a nonnegative function $g \in L^1(Q)$ such that for a particular sub-subsequence (not indicated)
\begin{align}\label{eq-subsub} |\alpha^m_\varepsilon(x,t) - \alpha^m(x,t)|^2 + |\nabla_x \alpha^m_\varepsilon(x,t) - \nabla_x \alpha^m(x,t)|^2 \leq g(x,t)\qquad\mbox{for a.e. $(x,t) \in Q$}.
\end{align}
For this same sub-subsequence, 
\[ G''_\delta(\alpha^m_\varepsilon)\,\nabla_x \alpha^m_\varepsilon \rightarrow G''_\delta(\alpha^m)\,\nabla_x \alpha^m  \qquad \mbox{a.e. on $Q$}\]
and 
\[ |G''_\delta(\alpha^m_\varepsilon(x,t))\,\nabla_x \alpha^m_\varepsilon(x,t) -  G''_\delta(\alpha^m(x,t))\,\nabla_x \alpha^m(x,t)|^2 \leq \frac{2}{\delta^2}\,g(x,t) + \frac{4}{\delta^2}\,|\nabla_x \alpha^m(x,t)|^2\qquad \mbox{a.e. on $Q$}.
 \]
Thus, we can pass to the limit over this sub-subsequence to deduce by the Dominated Convergence Theorem that
\begin{align} \label{eq-el2}
\begin{aligned}
 &\|\nabla_x(G'_\delta(\alpha^m_\varepsilon)) - \nabla_x(G'_\delta(\alpha^m))\|_{L^2(0,T;L^{2}(\Omega))}^2 \\
 &\qquad = \int_Q  |G''_\delta(\alpha^m_\varepsilon(x,t))\,\nabla_x \alpha^m_\varepsilon(x,t) -  G''_\delta(\alpha^m(x,t))\,\nabla_x \alpha^m(x,t)|^2\, \mathrm{d}x\,\mathrm{d}t \rightarrow 0\quad \mbox{as $\varepsilon \rightarrow 0_+$.}
\end{aligned}
\end{align}
By passing to the limit over this sub-subsequence, \eqref{eq-el1} and \eqref{eq-el2} imply \eqref{eq-el0}, and then using \eqref{eq-el0} in \eqref{eq-el00}, again by passage to the limit with $\varepsilon \rightarrow 0_+$ over this sub-subsequence implies that $\lim_{\varepsilon \rightarrow 0_+} |\mathrm{T}_{2,\varepsilon}|=0$. Hence, from \eqref{eq-el000} we have, by passage to the limit over this sub-subsequence,
\begin{align}\label{eq-el3}
\lim_{\varepsilon \rightarrow 0_+} \int_0^T \langle \partial_t G_\delta(\alpha^m_\varepsilon),  \zeta(\rho^m)\rangle \, \theta(t)\,\mathrm{d}t
&= \int_0^T \langle \partial_t\alpha^m, G'_\delta(\alpha^m)\, \zeta(\rho^m)\rangle \, \theta(t)\, \mathrm{d}t.
\end{align}
On the other hand, thanks to the smoothness of $\alpha^m_\varepsilon$ with respect to $t$, use of the chain rule and since $\partial_t G_\delta(\alpha^m_\varepsilon) \in C([0,T]; L^2(\Omega))$, the expression on the left-hand side of \eqref{eq-el3} can be rewritten as follows:
\begin{align}\label{eq-el4}
\begin{aligned}
\int_0^T \langle \partial_t G_\delta(\alpha^m_\varepsilon),  \zeta(\rho^m)\rangle \, \theta(t)\,\mathrm{d}t
&= \int_0^T \int_\Omega  \partial_t G_\delta(\alpha^m_\varepsilon)\,  \zeta(\rho^m) \, \theta(t)\,\mathrm{d}x\, \mathrm{d}t\\
&= - \int_0^T \int_\Omega  G_\delta(\alpha^m_\varepsilon)\,  \partial_t(\zeta(\rho^m) \, \theta(t))\,\mathrm{d}x\, \mathrm{d}t.
\end{aligned}
\end{align}
Now, with the sub-subsequence under consideration $G_\delta(\alpha^m_\varepsilon) \rightarrow G_\delta(\alpha^m)$ as $\varepsilon \rightarrow 0_+$ a.e. on $Q$. Also, because $0 \leq G_\delta(s) \leq C_\delta (s^2 + 1)$, it follows from \eqref{eq-subsub} that 
\[ 0 \leq G_\delta(\alpha^m_\varepsilon(x,t)) \leq C_\delta (|\alpha^m_\varepsilon(x,t)|^2 + 1) \leq C_\delta + 2C_\delta \, g(x,t) + 2C_\delta\, |\alpha^m(x,t)|^2.\] 
As the right-hand side of this is a function in $L^1(Q)$, by the Dominated Convergence Theorem we can pass to the limit over the sub-subsequence under consideration to deduce strong convergence of $G_\delta(\alpha_\varepsilon^m)$ to $G_\delta(\alpha^m)$ in $L^1(Q)$ as $\varepsilon \rightarrow 0_+$. It then follows from \eqref{eq-el3} and \eqref{eq-el4} that
\begin{align}\label{eq-el5}
 - \int_0^T \int_\Omega  G_\delta(\alpha^m)\,  \partial_t(\zeta(\rho^m) \, \theta(t))\,\mathrm{d}x\, \mathrm{d}t = \int_0^T \langle \partial_t\alpha^m, G'_\delta(\alpha^m)\, \zeta(\rho^m)\rangle \, \theta(t)\, \mathrm{d}t.
\end{align}
Next, we multiply \eqref{eq-Gdelta} by $\theta$, integrate this over $(0,T)$, and use \eqref{eq-el5} to rewrite the resulting first term on the left-hand side; hence. 
\begin{align}\label{eq-el6}
\begin{aligned}
 - \int_0^T \int_\Omega  G_\delta(\alpha^m)\,  \partial_t(\zeta(\rho^m) \, \theta(t))\,\mathrm{d}x\, \mathrm{d}t + \int_0^T \int_\Omega \left(\partial_t \zeta(\rho^m)\right)G_\delta(\alpha^m)\,\theta(t)\, \mathrm{d}x\, \mathrm{d}t \\
\qquad  + \int_0^T \int_\Omega G_\delta''(\alpha^m)\,|\nabla_x \alpha^m|^2\, \theta(t) \,\mathrm{d}x \,\mathrm{d}t = 0.
\end{aligned}
\end{align}
As $\partial_t(\zeta(\rho^m) \, \theta) = \partial_t(\zeta(\rho^m))\, \theta + \zeta(\rho^m)\, \partial_t \theta$, the first integral on the left-hand side of \eqref{eq-el6} can be written as a sum of two integrals, the first of which cancels with the penultimate integral on the left-hand side of \eqref{eq-el6}; hence, 
\begin{align*}
\begin{aligned}
 - \int_0^T \int_\Omega  G_\delta(\alpha^m)\, \zeta(\rho^m) \, \partial_t\theta \,\mathrm{d}x\, \mathrm{d}t + \int_0^T \int_\Omega G_\delta''(\alpha^m)\,|\nabla_x \alpha^m|^2\, \theta  \,\mathrm{d}x \,\mathrm{d}t = 0.
\end{aligned}
\end{align*}
Equivalently, since $\theta$ is independent of $x$, 
\begin{align}\label{eq-el7}
\begin{aligned}
 - \int_0^T \left(\int_\Omega  G_\delta(\alpha^m)\, \zeta(\rho^m) \,\mathrm{d}x \right) \partial_t \theta \, \mathrm{d}t + \int_0^T \left(\int_\Omega G_\delta''(\alpha^m)\,|\nabla_x \alpha^m|^2 \,\mathrm{d}x \right) \theta \,\mathrm{d}t = 0
\end{aligned}
\end{align}
for all $\theta \in C^\infty_0((0,T))$. Consequently, 
\begin{align*}
\frac{\mathrm{d}}{\mathrm{d}t}\int_\Omega G_\delta(\alpha^m) \, \zeta(\rho^m)\,\mathrm{d}x 
 + \int_\Omega G_\delta''(\alpha^m)\,|\nabla_x \alpha^m|^2\, \mathrm{d}x  = 0\qquad \mbox{a.e. on $(0,T)$}.
\end{align*}
This implies, upon integration with respect to $t$ and thanks to the properties of $G_\delta$ that
\begin{align*}
&\int_\Omega G_\delta(\alpha^m(x,t)) \, \zeta(\rho^m(x,t))\,\mathrm{d}x 
 + \frac{1}{\delta} \int_0^t \int_\Omega |\nabla_x \alpha^m(x,s)|^2\, \mathrm{d}x \, \mathrm{d}s \\
&\quad  \leq \int_\Omega G_\delta(\alpha^m(x,0)) \, \zeta(\rho^m(x,0))\,\mathrm{d}x = \int_\Omega G_\delta(\omega - \lambda^m_0(x)) \, \zeta(\rho^m_0(x))\,\mathrm{d}x\qquad \mbox{for all $t \in (0,T)$}.
\end{align*}
Thus, 
\begin{align*}
&\int_\Omega G_\delta(\alpha^m(x,t)) \, \zeta(\rho^m(x,t))\,\mathrm{d}x 
 \leq  \int_\Omega G_\delta(\omega - \lambda^m_0(x)) \, \zeta(\rho^m_0(x))\,\mathrm{d}x\qquad \mbox{for all $t \in (0,T)$}.
\end{align*}
Let us now denote by $\Omega_{-}(t)$, for $t \in (0,T)$, 
the set of all $x \in \Omega$ such that $\alpha^m(x,t)\leq 0$. Once again, appealing to the properties of $G_\delta$, we have that $G_\delta(\alpha^m(s,t)) \geq |\alpha^m(s,t)|^2/(2\delta)$ for all $x \in \Omega_{-}(t)$. Therefore, 
\begin{align}\label{eq-finald}
& \zeta_{\min}\int_{\Omega_{-}(t)} |\alpha^m(x,t)|^2 \,\mathrm{d}x 
 \leq  2\delta \int_\Omega G_\delta(\omega - \lambda^m_0(x)) \, \zeta(\rho^m_0(x))\,\mathrm{d}x\qquad \mbox{for all $t \in (0,T)$}.
\end{align}
On the other hand, because $\omega - \lambda^m_0(x) \geq 0$ on $\Omega$, by passing to the limit $\delta \rightarrow 0_+$ it follows that 
\[ \lim_{\delta \rightarrow 0_+}\int_\Omega G_\delta(\omega - \lambda^m_0(x)) \, \zeta(\rho^m_0(x))\,\mathrm{d}x = \int_\Omega [(\omega - \lambda^m_0(x))\,(\log(\omega-\lambda^m_0(x)) -1) + 1]  \, \zeta(\rho^m_0(x))\,\mathrm{d}x.\]
Hence, by passing to the limit $\delta \rightarrow 0_+$ in \eqref{eq-finald} it follows that 
\[ \int_{\Omega_{-}(t)} |\alpha^m(x,t)|^2 \,\mathrm{d}x 
 \leq  0 \qquad \mbox{for all $t \in (0,T)$}.\]
Therefore $\mbox{meas}(\Omega_{-}(t))=0$ for all $t \in (0,T)$. In other words, $0 \leq \lambda^m(x,t) \leq \omega$ for a.e. $(x,t) \in Q$. Consequently,
\begin{equation}\label{eq263}
\| \lambda^m \|_{L^\infty(\Omega \times (0,T))} \leq \| \lambda^m_0 \|_{L^\infty(\Omega \times (0,T))} \leq \frac{\varrho_{\max}}{\zeta_{\min}} \leq C.
\end{equation}
Noting that $\zeta(\cdot) \leq \zeta_{\max}$, we obtain from \eqref{varrho} and \eqref{eq263} that
\begin{equation}\label{eq-varrho}
\| \varrho^m \|_{L^\infty(\Omega \times (0,T))} \leq \zeta_{\max} \| \lambda^m \|_{L^\infty(\Omega \times (0,T))} \leq C.
\end{equation}

By setting $\bar\varphi=\lambda^m$ in \eqref{eq256} and using that $\zeta(\rho^m)$ satisfies \eqref{eq222} we further deduce that
\[ \frac{1}{2} \frac{\mathrm{d}}{{\mathrm d}t}\int_\Omega 
\zeta(\rho^m(x,t)) [\lambda^m(x,t)]^2 \,\mathrm{d}x + 
\int_\Omega |\nabla_x \lambda^m(x,t)|^2 \,\mathrm{d}x = 0,\]
and therefore, upon integration with respect to $t$, also
\begin{equation*}
\int_0^T \int_\Omega |\nabla_x \lambda^m|^2 \,\mathrm{d}x\,\mathrm{d}t \leq C.
\end{equation*}
From the uniform estimate (\ref{eq236}) we deduce that
\begin{equation*}
\sup_{t \in (0,T)} \int_{\mathcal{O}} \tilde{\psi}^m(x,\bm{q},t) \log(1 + \tilde{\psi}^m(x,\bm{q},t)) \,\mathrm{d}x\,\mathrm{d}\bm{q} \leq C(\ell),
\end{equation*}
where $\tilde{\psi}^m := M^m \hat{\psi}^m$. By de la Vall\'{e}e Poussin's Theorem (Theorem 2.29 in \cite{MR2341508}) the sequence $\tilde{\psi}^m$ is uniformly equi-integrable. Hence, by the Dunford--Pettis Theorem (Theorem 2.54 in \cite{MR2341508}), the sequence $( \tilde{\psi}^m)_{m\geq 1}$ is weakly relatively compact in $L^1(\mathcal{O} \times (0,T))$, which implies the existence of a subsequence (not relabelled) such that
\begin{equation*}
\tilde{\psi}^m \rightharpoonup \tilde{\psi} \qquad \text{weakly in $L^1(\mathcal{O} \times (0,T))$}.
\end{equation*}
Since $M^m$ converges to $M$ uniformly in $C(\overline{D})$, we deduce that
\begin{equation*}
\hat{\psi}^m \rightharpoonup \frac{\tilde \psi}{M}=:\hat{\psi} \qquad \text{weakly in $L^1_{loc}(\mathcal{O} \times (0,T))$}.
\end{equation*}
\par
Next, we shall show that 
\begin{equation*}
\hat{\psi}^m \to \hat{\psi} \qquad \text{a.e. in $\mathcal{O} \times (0,T)$}.
\end{equation*}
Let $\mathcal{O}_0$ be a Lipschitz subdomain of $\mathcal{O}$ such that $\mathcal{O}_0 \subset \overline{\mathcal{O}_0} \subset \mathcal{O}$. Since $\mathcal{F}(s) = s \log s + 1 \geq s$ for all $s \in \mathbb{R}_{\geq 0}$ and $M^m$ is bounded below on $\mathcal{O}_0$ by a positive constant (which may depend on $\mathcal{O}_0$), we have from (\ref{eq236}) that
\begin{equation}\label{eq269}
\sup_{t \in (0,T)} \| \sqrt{\hat{\psi}^m(\cdot,t)} \|^2_{L^2(\mathcal{O}_0)} + \int_0^T \| \sqrt{\hat{\psi}^m} \|^2_{W^{1,2}(\mathcal{O}_0)} \,\mathrm{d}t \leq C(\mathcal{O}_0).
\end{equation}
Since $\mathcal{O}_0 \subset \mathcal{O} \subset \mathbb{R}^{(K+1)d}$, standard interpolation gives that
\begin{equation}\label{eq270}
\int_0^T \int_{\mathcal{O}_0} |\hat{\psi}^m|^{\frac{(K+1)d + 2}{d(K+1)}} \,\mathrm{d}x\,\mathrm{d}\bm{q}\,\mathrm{d}t = \int_0^T \int_{\mathcal{O}_0} \left|\sqrt{\hat{\psi}^m} \right|^{\frac{2((K+1)d + 2)}{d(K+1)}} \,\mathrm{d}x\,\mathrm{d}\bm{q}\,\mathrm{d}t \leq C(\mathcal{O}_0).
\end{equation}
The application of H\"{o}lder's inequality gives that
\begin{equation*}
\begin{split}
\int_0^T \int_{\mathcal{O}_0} |\nabla_{x,\bm{q}} \hat{\psi}^m|^p \,\mathrm{d}x\,\mathrm{d}\bm{q}\,\mathrm{d}t &= 2^p \int_0^T \int_{\mathcal{O}_0} \left| \nabla_{x,\bm{q}} \sqrt{\hat{\psi}^m} \right|^p \left| \sqrt{\hat{\psi}^m} \right|^p \,\mathrm{d}x\,\mathrm{d}\bm{q}\,\mathrm{d}t \\
&\leq 2^p \left( \int_0^T \int_{\mathcal{O}_0} \left| \nabla_{x,\bm{q}} \sqrt{\hat{\psi}^m} \right|^2 \,\mathrm{d}x\,\mathrm{d}\bm{q}\,\mathrm{d}t \right)^{\frac{p}{2}} \left( \int_0^T \int_{\mathcal{O}_0} \left| \sqrt{\hat{\psi}^m} \right|^{\frac{2p}{2-p}} \,\mathrm{d}x\,\mathrm{d}\bm{q}\,\mathrm{d}t \right)^{\frac{2-p}{2}} \\
&\leq C(\mathcal{O}_0),
\end{split}
\end{equation*}
provided that
\begin{equation}\label{eq272}
\frac{2p}{2-p} \leq \frac{2((K+1)d + 2)}{d(K+1)}.
\end{equation}
By selecting $p = \frac{d(K+1) + 2}{d(K+1)+1}$, which is the largest value satisfying (\ref{eq272}), we obtain
\begin{equation}\label{eq273}
\int_0^T \int_{\mathcal{O}_0} |\nabla_{x,\bm{q}} \hat{\psi}^m|^{\frac{d(K+1) + 2}{d(K+1)+1}} \,\mathrm{d}x\,\mathrm{d}\bm{q}\,\mathrm{d}t \leq C(\mathcal{O}_0).
\end{equation}
It directly follows from (\ref{eq263}) that
\begin{equation}\label{eq273a}
\| \tilde{\psi}^m \|_{L^\infty(Q; L^1(D))} \leq C,
\end{equation}
and therefore,
\begin{equation}\label{eq275}
\| \hat{\psi}^m \|_{L^\infty(\Omega_0 \times (0,T); L^1(D_0))} \leq C,
\end{equation}
where $\mathcal{O}_0 = \Omega_0 \times D_0$. Interpolating between (\ref{eq270}), (\ref{eq273}) and (\ref{eq275}) we see that for any two real numbers $q_1$ and $q_2$, with
\begin{equation*}
\frac{(K+1)d + 2}{(K+1)d} \leq q_1 < \infty, \qquad 1 < q_2 \leq \frac{(K+1)d + 2}{(K+1)d},
\end{equation*}
and satisfying the relation
\begin{equation*}
q_1 \left(1- \frac{1}{q_2} \right) = \frac{2}{(K+1)d},
\end{equation*}
we have that
\begin{equation*}
\| \hat{\psi}^m \|_{L^{q_1}(\Omega_0 \times(0,T); L^{q_2}(D_0))} \leq C(\mathcal{O}_0).
\end{equation*}
Since $\zeta(\cdot) \leq \zeta_{\max}$, using (\ref{eq237}) and H\"{o}lder's inequality we deduce that there exists a $\delta > 0$ such that
\begin{equation}\label{eq279}
\| \zeta(\rho^m) \bm{v}^m \hat{\psi}^m \|_{L^{1+\delta}(\mathcal{O}_0 \times (0,T))} \leq C(\mathcal{O}_0).
\end{equation}
Similarly, from (\ref{eq236}), we deduce that
\begin{equation}\label{eq280}
\| \zeta(\rho^m) \Lambda_\ell(\hat{\psi}^m)(\nabla_x \bm{v}^m) \bm{q} \|_{L^{1+\delta}(\mathcal{O}_0 \times (0,T))} \leq C(\mathcal{O}_0).
\end{equation}
To apply the Div-Curl Lemma, let us first define the following sequences of $(1 + d + Kd)$-component vector fields:
\begin{align}\nonumber
H^m &\coloneqq (M^m\zeta(\rho^m)\hat{\psi}^m, M^m \zeta(\rho^m) \hat{\psi}^m\bm{v}^m - M^m \nabla_x \hat{\psi}^m, M \zeta(\rho^m) \Lambda_\ell(\hat{\psi}^m)(\nabla_x \bm{v}^m) \bm{q} - M^m \mathbb{A}(\nabla_{\bm{q}} \hat{\psi}^m)), \\
\nonumber Q^m &\coloneqq ( (1 + \hat{\psi}^m)^\alpha, \underbrace{0, \dots,0}_\text{$(d+Kd)$-times} ),
\end{align}
for some $\alpha \in (0, 1/2)$. Consequently, using (\ref{eq236}), (\ref{eq279}) and (\ref{eq280}), we deduce the existence of subsequences (not relabelled) such that
\begin{align*}
H^m &\rightharpoonup H \qquad \text{weakly in $L^{1+\delta}(\mathcal{O}_0 \times (0,T); \mathbb{R}^{1 + d +Kd})$}, \\
Q^m &\rightharpoonup Q \qquad \, \text{weakly in $L^{\frac{1}{\alpha}}(\mathcal{O}_0 \times (0,T); \mathbb{R}^{1 + d +Kd})$},
\end{align*}
where, noting the uniform convergence of $M^m$ and the strong convergences of $\zeta(\rho^m)$ and $\bm{v}^m$,
\begin{align*}
H &\coloneqq (M\zeta(\rho)\hat{\psi}, M \zeta(\rho) \hat{\psi}\bm{v} - M \nabla_x \hat{\psi}, M \zeta(\rho) \overline{\Lambda_\ell(\hat{\psi})}(\nabla_x \bm{v}) \bm{q} - M \mathbb{A}(\nabla_{\bm{q}} \hat{\psi})), \\
\nonumber Q &\coloneqq ( \overline{(1 + \hat{\psi})^\alpha}, 0, \dots,0).
\end{align*}
It follows from (\ref{eq213}) that
\begin{equation*}
\divergence_{t, x, \bm{q}} H^m = 0 \qquad \text{in $\mathcal{O}_0 \times (0,T)$}.
\end{equation*}
Since $\alpha \in (0, 1/2)$, we obtain by using (\ref{eq269}) that
\begin{equation*}
\begin{split}
\int_0^T \int_{\mathcal{O}_0} | \curl_{t,x,\bm{q}} Q^m |^2 \,\mathrm{d}x\,\mathrm{d}\bm{q}\,\mathrm{d}t &= \int_0^T \int_{\mathcal{O}_0} | \nabla_{t,x,\bm{q}} Q^m - (\nabla_{t,x,\bm{q}} Q^m)^{\rm T} |^2 \,\mathrm{d}x\,\mathrm{d}\bm{q}\,\mathrm{d}t \\
&\leq C\int_0^T \int_{\mathcal{O}_0} |\nabla_{x,\bm{q}} (1 + \hat{\psi}^m)^\alpha|^2 \,\mathrm{d}x\,\mathrm{d}\bm{q}\,\mathrm{d}t \\
&\leq C\int_0^T \int_{\mathcal{O}_0} \left| \nabla_{x,\bm{q}} \sqrt{\hat{\psi}^m} \right|^2 \,\mathrm{d}x\,\mathrm{d}\bm{q}\,\mathrm{d}t \\
&\leq C(\mathcal{O}_0).
\end{split}
\end{equation*}
Hence, $(\divergence_{t,x,\bm{q}} H^m)_{m=1}^\infty$ is precompact in $W^{-1,2}(\mathcal{O}_0 \times (0,T))$ and $(\curl_{t,x,\bm{q}} Q^m)_{m=1}^\infty$ is precompact in $W^{-1,2}(\mathcal{O}_0 \times (0,T))$. By choosing $\alpha < \frac{\delta}{1+\delta}$ we deduce using the Div-Curl Lemma that
\begin{equation*}
H^m \cdot Q^m \rightharpoonup H \cdot Q \qquad \text{weakly in $L^1(\mathcal{O}_0 \times (0,T))$}.
\end{equation*}
In particular, we have that
\begin{equation*}
M^m \zeta(\rho^m) \hat{\psi}^m (1+\hat{\psi}^m)^\alpha \rightharpoonup M \zeta(\rho) \hat{\psi} \overline{(1+\hat{\psi})^\alpha}.
\end{equation*}
Since $M^m$ converges to $M$ uniformly and $\zeta(\rho^m)$ converges to $\zeta(\rho)$ strongly in $L^\infty(0,T; L^p(\Omega))$, $1 < p < \infty$, the above implies that
\begin{equation}\label{eq287}
\hat{\psi}^m (1+\hat{\psi}^m)^\alpha \rightharpoonup \hat{\psi} \overline{(1+\hat{\psi})^\alpha}.
\end{equation}
Since $(1+\hat{\psi}^m)^\alpha \rightharpoonup \overline{(1+\hat{\psi})^\alpha}$ in $L^1(\mathcal{O}_0 \times (0,T))$, we can add this to (\ref{eq287}), which gives
\begin{equation*}
(1 + \hat{\psi}^m)^{\alpha + 1} = (1 + \hat{\psi}^m)(1 + \hat{\psi}^m)^{\alpha} \rightharpoonup (1 + \hat{\psi})\overline{(1+\hat{\psi})^\alpha}.
\end{equation*}
Thanks to the weak lower-semicontinuity of the continuous convex function $s \in [0,\infty) \mapsto s^{\alpha+1} \in [0,\infty)$ we have that
\begin{equation*}
(1 + \hat{\psi})^{\alpha + 1} \leq (1 + \hat{\psi})\overline{(1+\hat{\psi})^\alpha},
\end{equation*}
which is equivalent to 
\begin{equation*}
(1 + \hat{\psi})^{\alpha } \leq \overline{(1+\hat{\psi})^\alpha}.
\end{equation*}
On the other hand, the function $s \in [0,\infty) \mapsto s^{\alpha} \in [0,\infty)$ is concave. Again, by the weak lower-semicontinuity of the continuous convex function $s \in [0,\infty) \mapsto -s^{\alpha} \in [0,\infty)$, we deduce that
\begin{equation*}
-(1 + \hat{\psi})^{\alpha} \leq - \overline{(1+\hat{\psi})^\alpha}.
\end{equation*}
Therefore,
\begin{equation*}
(1 + \hat{\psi})^{\alpha} = \overline{(1+\hat{\psi})^\alpha}.
\end{equation*}
Since $s \in [0,\infty) \mapsto s^{\alpha} \in [0,\infty)$ is strictly concave, thanks to Theorem 10.20 in \cite{MR2499296}, there exists a subsequence (not relabelled), such that
\begin{equation*}
\hat{\psi}^m \to \hat{\psi} \qquad \text{a.e. in $\mathcal{O}_0 \times (0,T)$}.
\end{equation*}
Since $M^m$ converges uniformly to $M$, we have that
\begin{equation*}
\tilde{\psi}^m \to \tilde{\psi} \qquad \text{a.e. in $\mathcal{O}_0 \times (0,T)$}.
\end{equation*}
Now we want to extend the pointwise convergence result of $\tilde{\psi}^m$ to the whole of our domain $\mathcal{O} \times (0,T)$. For this purpose we choose a nondecreasing sequence of nested sets $( \mathcal{O}_0^k)_{k=1}^\infty$, i.e., $\mathcal{O}^1_0 \subset \mathcal{O}^2_0 \subset \cdots \subset \mathcal{O}^k_0 \subset \cdots$, satisfying $\cup_{k=1}^\infty \mathcal{O}^k_0 = \mathcal{O}$. For each $k \in \mathbb{N}$, we deduce the existence of a subsequence of $(\tilde{\psi}^m)_{m=1}^\infty$ that is pointwise convergent to $\tilde{\psi}$ a.e. in $\mathcal{O}^k_0 \times (0,T)$. Arguing by a diagonal procedure we deduce that there exists a subsequence such that
\begin{equation}\label{eq313}
\tilde{\psi}^m \to \tilde{\psi} \qquad \text{a.e. in $\mathcal{O} \times (0,T)$}.
\end{equation}
Since $\tilde{\psi}^m$ is uniformly equi-integrable, using Vitali's Convergence Theorem (Theorem 2.24 in \cite{MR2341508}), we obtain that
\begin{equation}\label{eq296}
\tilde{\psi}^m \to \tilde{\psi} \qquad \text{strongly in $L^1(0,T; L^1(\mathcal{O}))$}.
\end{equation}
Interpolating between (\ref{eq273a}) and (\ref{eq296}) gives that
\begin{equation}\label{eq297}
M^m \hat\psi^m=\tilde{\psi}^m \to \tilde{\psi}= M \hat\psi \qquad \text{strongly in $L^p(Q; L^1(D))$, for all $p \in [1,\infty)$}.
\end{equation}
Thus, by recalling \eqref{eq253} and \eqref{rhomplus} it follows that
\[ \varrho^m \rightarrow \varrho \quad \mbox{strongly in $L^p(Q)$ for all $p \in [1,\infty)$},\quad \mbox{where}\quad \varrho(x,t)=\zeta(\rho) \int_D M(\bm{q})\hat\psi(x,\bm{q},t)\,\mathrm{d}\bm{q}.\]
Hence, and by recalling \eqref{rhom-strong}, we have that
\begin{align}\label{eq270.1}
\mu(\rho^{m},\varrho^{m}) &\to \mu(\rho,\varrho) \qquad \text{strongly in $L^p(Q)$ for all $p \in [1,\infty)$}.
\end{align}

For any measurable $U \subset (\mathcal{O} \times (0,T))$, with $|U| \leq \delta$, we use H\"{o}lder's inequality to deduce that
\begin{equation}\label{eq316}
\begin{split}
\int_U M^m |\nabla_{x,\bm{q}} \hat{\psi}^m| \,\mathrm{d}x\,\mathrm{d}\bm{q}\,\mathrm{d}t &= 2 \int_U M^m \left| \nabla_{x,\bm{q}} \sqrt{\hat{\psi}^m} \right| \left| \sqrt{\hat{\psi}^m} \right| \,\mathrm{d}x\,\mathrm{d}\bm{q}\,\mathrm{d}t \\
&\leq 2 \left( \int_U M^m \left| \nabla_{x,\bm{q}} \sqrt{\hat{\psi}^m} \right|^2 \,\mathrm{d}x\,\mathrm{d}\bm{q}\,\mathrm{d}t \right)^{\frac{1}{2}} \left( \int_U \tilde{\psi}^m \,\mathrm{d}x\,\mathrm{d}\bm{q}\,\mathrm{d}t \right)^{\frac{1}{2}} \\
&\leq C\varepsilon^{\frac{1}{2}},
\end{split}
\end{equation}
which follows from the uniform estimate (\ref{eq236}) and the uniform equi-integrability of $\tilde{\psi}^m$. By the Dunford--Pettis Theorem we can extract a subsequence such that
\begin{equation}\label{eq299}
M^m \nabla_{x,\bm{q}} \hat{\psi}^m \rightharpoonup \overline{M^m \nabla_{x,\bm{q}} \hat{\psi}^m} \qquad \text{weakly in $L^1(\mathcal{O} \times (0,T); \mathbb{R}^{d(K+1)})$}.
\end{equation}
From (\ref{eq273}) we deduce that $\nabla_{x,\bm{q}} \hat{\psi}^m$ weakly converges to $\nabla_{x,\bm{q}} \hat{\psi}$ locally in $L^1$. Since $M^m$ converges uniformly to $M$, we can identify the weak limit $\overline{M^m \nabla_{x,\bm{q}} \hat{\psi}^m}$ as $M \nabla_{x,\bm{q}} \hat{\psi}$. Analogously as in (\ref{eq299}) it also follows from (\ref{eq273}) that
\begin{equation}\label{eq300}
\sqrt{M^m} \nabla_{x,\bm{q}} \sqrt{\hat{\psi}^m} \rightharpoonup \sqrt{M} \nabla_{x,\bm{q}} \sqrt{\hat{\psi}} \qquad \text{weakly in $L^2(\mathcal{O} \times (0,T); \mathbb{R}^{d(K+1)})$}.
\end{equation}
Since $M^m \nabla_{x,\bm{q}} \hat{\psi}^m =2 \sqrt{M^m}  \sqrt{\hat{\psi}^m} \sqrt{M^m} \nabla_{x,\bm{q}} \sqrt{\hat{\psi}^m} = 2 \sqrt{\tilde{\psi}^m} \sqrt{M^m} \nabla_{x,\bm{q}} \sqrt{\hat{\psi}^m}$, by using a similar calculation as in (\ref{eq316}) we also see that
\begin{equation*}
\begin{split}
\int_Q \left( \int_D M^m |\nabla_{x,\bm{q}} \hat{\psi}^m| \,\mathrm{d}\bm{q} \right)^2 \,\mathrm{d}x\,\mathrm{d}t &= 4 \int_Q \left( \int_D \sqrt{\tilde{\psi}^m} \left|\sqrt{M^m} \nabla_{x,\bm{q}} \sqrt{\hat{\psi}^m} \right| \,\mathrm{d}\bm{q} \right)^2 \,\mathrm{d}x\,\mathrm{d}t \\
&\leq 4\int_Q \| \tilde{\psi}^m \|^2_{L^1(D)} \| \sqrt{M^m} \nabla_{x,\bm{q}} \sqrt{\hat{\psi}^m} \|^2_{L^2(D)} \,\mathrm{d}x\,\mathrm{d}t \\
&\leq 4 \| \lambda^m \|^2_{L^\infty(\Omega \times (0,T))} \int_Q \| \sqrt{M^m} \nabla_{x,\bm{q}} \sqrt{\hat{\psi}^m} \|^2_{L^2(D)} \,\mathrm{d}x\,\mathrm{d}t \\
&\leq C,
\end{split}
\end{equation*}
where we have used the estimates (\ref{eq236}) and (\ref{eq263}). Therefore, we can strengthen (\ref{eq299}) as follows:
\begin{equation}\label{eq301}
M^m \nabla_{x,\bm{q}} \hat{\psi}^m \rightharpoonup M \nabla_{x,\bm{q}} \hat{\psi} \qquad \text{weakly in $L^2(Q; L^1(D; \mathbb{R}^{d(K+1)}))$}.
\end{equation}
With the strong convergence results (\ref{eq250}), (\ref{eq253}) and (\ref{eq297}) for $\bm{v}^m$, $\zeta(\rho^m)$ and $\tilde{\psi}^m$ we deduce that
\begin{equation}\label{eq302}
M^m \zeta(\rho^m) \bm{v}^m \hat{\psi}^m \to M \zeta(\rho) \bm{v} \hat{\psi} \qquad \text{strongly in $L^p(Q; L^1(D))$},
\end{equation}
where $p \in [1, \frac{2(d+2)}{d})$. Since $\Gamma_\ell \in C_0^\infty((-2\ell, 2\ell))$, there exists a positive constant $C = C(\ell)$ such that
\begin{equation*}
|\Gamma_\ell(\hat{\psi}^m) - \Gamma_\ell(\hat{\psi})| \leq C|\hat{\psi}^m - \hat{\psi}|.
\end{equation*}
Hence, using (\ref{eq253}) and (\ref{eq297}),
\begin{equation*}
\zeta(\rho^m)\Lambda_\ell(\hat{\psi}^m) \to \zeta(\rho)\Lambda_\ell(\hat{\psi})  \qquad \text{strongly in $L^p(Q; L^1(D))$, for all $p \in [1,\infty)$}.
\end{equation*} 
It then follows from (\ref{eq248}) that
\begin{equation}\label{eq305}
M \zeta(\rho^m) \Lambda_\ell(\hat{\psi}^m)(\nabla_x \bm{v}^m) \rightharpoonup M \zeta(\rho) \Lambda_\ell(\hat{\psi})(\nabla_x \bm{v}) \qquad \text{weakly in $L^p(Q; L^1(D; \mathbb{R}^{d \times d}))$, for all $p \in [1,2)$}.
\end{equation}
Using (\ref{eq213}) and the convergence results (\ref{eq301}), (\ref{eq302}) and (\ref{eq305}) we obtain that
\begin{equation}\label{eq325}
\partial_t (M^m \zeta(\rho^m) \hat{\psi}^m) \rightharpoonup \partial_t (M \zeta(\rho) \hat{\psi}) \qquad \text{weakly in $L^p(Q; W^{-1,1}(D))$, for all $p \in [1,2)$}.
\end{equation}
From (\ref{eq253}) we deduce that
\begin{equation*}
\zeta(\rho^m) \to \zeta(\rho) \qquad \text{a.e. in $Q$}.
\end{equation*}
Therefore, using (\ref{eq313}) and the Dominated Convergence Theorem (thanks to the presence of the truncation $T_\ell$), we can let $m \to \infty$ in (\ref{eq218}) to deduce that
\begin{equation}\label{eq326}
\tau^m \to \tau \qquad \text{strongly in $L^1(Q; \mathbb{R}^{d \times d})$},
\end{equation}
where 
\begin{equation}\label{eq327}
\tau = -k \int_D \left[ K M \zeta(\rho)T_\ell(\hat{\psi})I + \sum_{j=1}^K \zeta(\rho) T_\ell(\hat{\psi})\nabla_{\bm{q}^j} M \otimes \bm{q}^j \right] \,\mathrm{d}\bm{q} \quad \text{a.e. in $Q$}.
\end{equation}
\par
Now let us reinstate the superscript $\ell$. Collecting the convergence results (\ref{eq263.1}), (\ref{eq248}), (\ref{eq250}), (\ref{eq269.1}), (\ref{eq270.1}), (\ref{eq253}), (\ref{eq296}), (\ref{eq301}), (\ref{eq302}), (\ref{eq305}), (\ref{eq325}) and (\ref{eq326}), and using the fact that $\bm{f}^m$ converges to $\bm{f}$ in $L^2(0,T; L^2(\Omega; \mathbb{R}^d))$, we can pass to the limit as $m \to \infty$ in (\ref{eq211})--(\ref{eq213}) to obtain the following:
\begin{equation}\label{eq328}
\int_0^T [ \langle \partial_t \rho^\ell, \eta \rangle - (\bm{v}^\ell \rho^\ell, \nabla_x \eta)] \,\mathrm{d}t = 0,\quad \text{for all $\eta \in L^1(0,T; W^{1, \frac{q}{q-1}}(\Omega))$}, 
\end{equation}
where $q \in (2,\infty)$ when $d=2$ and $q \in [3,6]$ when $d=3$,
\begin{equation}\label{eq329}
\begin{split}
&\int_0^T \langle \partial_t (\rho^\ell \bm{v}^\ell), \bm{w} \rangle \,\mathrm{d}t + \int_0^T [-(\rho^\ell \bm{v}^\ell \otimes \bm{v}^\ell, \nabla_x \bm{w}) + (\mu(\rho^\ell,\varrho^\ell)D(\bm{v}^\ell), \nabla_x \bm{w})] \,\mathrm{d}t \\
&\quad = \int_0^T [-(\tau^\ell, \nabla_x \bm{w}) +  (\rho^\ell \bm{f}, \bm{w})] \,\mathrm{d}t \quad \text{for all $\bm{w} \in L^s(0,T; W^{1,s}_{0,\divergence}(\Omega; \mathbb{R}^d))$\quad with $s>2$}, 
\end{split}
\end{equation}
and
\begin{equation}\label{eq330}
\begin{split}
&\int_0^T \left\langle \partial_t(M \zeta(\rho^\ell)\hat{\psi}^\ell), \varphi \right\rangle_{\mathcal{O}} -  \left(M \zeta(\rho^\ell) \bm{v}^\ell \hat{\psi}^\ell, \nabla_x \varphi \right)_{\mathcal{O}} -  \left(M \zeta(\rho^\ell) \Lambda_\ell(\hat{\psi}^\ell)(\nabla_x \bm{v}^\ell) \bm{q}, \nabla_{\bm{q}} \varphi \right)_{\mathcal{O}} \,\mathrm{d}t \\
&\quad  + \int_0^T (M \nabla_x \hat{\psi}^\ell, \nabla_x \varphi)_{\mathcal{O}} + \left( M \mathbb{A}(\nabla_{\bm{q}} \hat{\psi}^\ell), \nabla_{\bm{q}} \varphi \right)_{\mathcal{O}} \,\mathrm{d}t = 0 \quad  \text{for all $\varphi \in L^\infty(0,T; W^{1,\infty}(\mathcal{O}))$}.
\end{split}
\end{equation}
Letting $m \to \infty$ in (\ref{eq255}), we deduce the following energy inequality:
\begin{equation}\label{eq331}
\begin{split}
&k\int_{\mathcal{O}} M \zeta(\rho^\ell(\cdot, t)) \mathcal{F}(\hat{\psi}^\ell(\cdot,t)) \,\mathrm{d}x\,\mathrm{d}\bm{q} + \frac{1}{2}\int_\Omega \rho^\ell(\cdot,t) |\bm{v}^\ell(\cdot,t)|^2 \,\mathrm{d}x \\
&\quad + \int_0^t \int_\Omega \mu(\rho^\ell) |D(\bm{v}^\ell)|^2 \,\mathrm{d}x\,\mathrm{d}s + 4kC_1\int_0^t \int_{\mathcal{O}} M \left| \nabla_{x,\bm{q}} \sqrt{\hat{\psi}^\ell} \right|^2 \,\mathrm{d}x\,\mathrm{d}\bm{q}\,\mathrm{d}s \\
&\leq k\int_{\mathcal{O}} M \zeta(\rho_0) \mathcal{F}(T_\ell(\hat{\psi}_0)) \,\mathrm{d}x\,\mathrm{d}\bm{q} + \frac{1}{2}\int_\Omega \rho_0 |\bm{v}_0|^2 \,\mathrm{d}x + \int_0^t (\rho^\ell \bm{f}, \bm{v}^\ell) \,\mathrm{d}s.
\end{split}
\end{equation}
Letting $m \to \infty$ in (\ref{eq220}), (\ref{eq-varrho}) and (\ref{eq236}), by the weak lower semicontinuity of norms, we deduce the following a priori estimate:
\begin{equation}\label{eq332}
\begin{split}
&\sup_{t \in (0,T)} \left( \| \rho^\ell(\cdot, t) \|_{L^\infty(\Omega)} + \| \varrho^\ell(\cdot, t) \|_{L^\infty(\Omega)} \right) + 2k\zeta_{\min} \sup_{t \in (0,T)} \| \mathcal{F}(\hat{\psi}^\ell(\cdot,t)) \|_{L^1_M(\mathcal{O})} \\
&\quad + \rho_{\min}  \sup_{t \in (0,T)} \| \bm{v}^\ell(\cdot, t) \|^2_{L^2(\Omega; \mathbb{R}^d)} + \mu_{\min}c_0 \int_0^T \| \bm{v}^\ell \|^2_{W^{1,2}(\Omega; \mathbb{R}^d)}\,\mathrm{d}t + 8kC_1\int_0^T \left\| \sqrt{\hat{\psi}^\ell} \right\|^2_{W^{1,2}_M(\mathcal{O})} \,\mathrm{d}t \\
&\leq 2k\zeta_{\max} \int_{\mathcal{O}} M \mathcal{F}(T_\ell(\hat{\psi}_0)) \,\mathrm{d}x\,\mathrm{d}\bm{q} + \rho_{\max} \| \bm{v}_0 \|^2_{L^2(\Omega; \mathbb{R}^d)} + \frac{\rho_{\max}^2}{\mu_{\min}c_0} \int_0^T \| \bm{f} \|^2_{L^2(\Omega; \mathbb{R}^d)} \,\mathrm{d}t \\
&\leq C,
\end{split}
\end{equation}
where $C$ is a positive constant independent of $\ell$ and
\begin{equation}\label{varrhol}
\varrho^\ell(x,t) = \zeta(\rho^\ell(x,t)) \int_D M(\bm{q}) \hat{\psi}^\ell(x,\bm{q},t) \,\mathrm{d}\bm{q}.
\end{equation}
We are now ready to pass to the final limit, $\ell \rightarrow \infty$. 

\subsection{Passage to the limit with $\ell$}
From the uniform estimate (\ref{eq332}) (noting that the constant $C$ does not depend on $\ell$) we deduce the existence of a subsequence (not relabelled) such that, as $\ell \to \infty$,
\begin{align}
\begin{aligned}
\bm{v}^\ell &\rightharpoonup \bm{v} \qquad \text{weak* in $L^\infty(0,T; L^2_{0,\divergence}(\Omega; \mathbb{R}^d))$}, \\
\label{eq335} \bm{v}^\ell &\rightharpoonup \bm{v} \qquad \text{weakly in $L^2(0,T; W^{1,2}_{0,\divergence}(\Omega; \mathbb{R}^d))$}.
\end{aligned}
\end{align}
By standard interpolation
\begin{equation*}
\bm{v}^\ell \rightharpoonup \bm{v} \qquad \text{weakly in $L^{\frac{2(d+2)}{d}}(Q; \mathbb{R}^d)$}.
\end{equation*}
From the definition (\ref{eq327}) of $\tau^\ell$ and the definition of the truncation $T_\ell$ we deduce that
\begin{equation}\label{eq337}
\begin{split}
\| \tau^\ell \|^2_{L^2(0,T; L^2(\Omega))} &\leq C(M, \zeta_{\max}) \int_0^T \int_\Omega \left( \int_D T_\ell(\hat{\psi}^\ell) \,\mathrm{d}\bm{q} \right)^2 \,\mathrm{d}x\,\mathrm{d}t \\
&\leq C(M,\zeta_{\max}) \int_0^T \int_\Omega \left( \int_D \hat{\psi}^\ell \,\mathrm{d}\bm{q} \right)^2 \,\mathrm{d}x\,\mathrm{d}t \\
&\leq C,
\end{split}
\end{equation}
where $C$ is a positive constant independent of $\ell$, which then implies the existence of a subsequence (not relabelled) such that
\begin{equation}\label{eq338}
\tau^\ell \rightharpoonup \tau \qquad \text{weakly in $L^2(0,T; L^2(\Omega; \mathbb{R}^{d \times d}))$}.
\end{equation}
We can perform a similar argument as in (\ref{eq1.216})--(\ref{eq261}) to deduce that
\begin{equation*}
\| \bm{v}^\ell \|_{N^\gamma_2(0,T; L^2(\Omega; \mathbb{R}^d))} \leq C,
\end{equation*}
where $0 < \gamma < 1/4$ when $d=2$ and $0 < \gamma \leq 1/8$ when $d=3$. By the Aubin--Lions Lemma it follows that
\begin{equation}\label{eq340}
\bm{v}^\ell \to \bm{v} \qquad \text{strongly in $L^2(0,T; L^p(\Omega; \mathbb{R}^d))$},
\end{equation}
where $p \in [1,\infty)$ when $d=2$ and $p \in [1,6)$ when $d=3$.
\par
From the uniform estimate (\ref{eq332}) we deduce the existence of a subsequence (not relabelled) such that, as $\ell \to \infty$,
\begin{equation*}
\rho^\ell \rightharpoonup \rho \qquad \text{weak* in $L^\infty(Q)$}.
\end{equation*}
Using (\ref{eq328}), (\ref{eq332}) and Sobolev embedding we have that
\begin{equation*}
\| \partial_t \rho^\ell \|_{L^2(0,T; W^{-1,p}(\Omega))} \leq C,
\end{equation*}
where $p \in [1,\infty)$ when $d=2$ and $p \in [1,6]$ when $d=3$. Therefore, we have that
\begin{equation}\label{eq343}
\partial_t \rho^\ell \rightharpoonup \partial_t \rho \qquad \text{weakly in $L^2(0,T; W^{-1,p}(\Omega))$},
\end{equation}
where $p \in (1,\infty)$ when $d=2$ and $p \in (1,6]$ when $d=3$. We can also show that
\begin{equation}\label{eq344}
\rho^\ell \to \rho \qquad \text{strongly in $L^p(Q)$, for any $p \in [1,\infty)$};
\end{equation}
the proof proceeds similarly as in (\ref{eq1.125})--(\ref{eq194}). With the convergence result (\ref{eq340}) for $\bm{v}^\ell$ we can perform a similar argument as  in Theorem VI.1.9 in \cite{MR2986590} and strengthen the above convergence to get
\begin{equation}\label{eq344a}
\rho^\ell \to \rho \qquad \text{strongly in $C([0,T]; L^p(\Omega))$, for any $p \in [1,\infty)$}.
\end{equation}
From the assumption (\ref{eq26}) we further deduce that
\begin{align}
\label{eq345} \zeta(\rho^\ell) &\to \zeta(\rho) \qquad \text{strongly in $C([0,T]; L^p(\Omega))$, for any $p \in [1,\infty)$}.
\end{align}
Using (\ref{eq329}), (\ref{eq332}) and (\ref{eq337}) we have that
\begin{equation*}
\int_0^T \| \partial_t (\rho^\ell \bm{v}^\ell) \|^2_{W^{-1,2}(\Omega)} \,\mathrm{d}t \leq C,
\end{equation*}
where $C$ is a positive constant independent of $\ell$, which then implies that
\begin{equation}\label{eq348}
\partial_t (\rho^\ell \bm{v}^\ell) \rightharpoonup \partial_t (\rho \bm{v}) \qquad \text{weakly in $L^2(0,T; W^{-1,2}(\Omega; \mathbb{R}^d))$}.
\end{equation}
\par
Next, we will show the strong convergence of $\hat{\psi}^\ell$. From (\ref{eq332}) we have that
\begin{equation*}
\sup_{t \in (0,T)} \int_{\mathcal{O}} M(\bm{q}) \hat{\psi}^\ell(x,\bm{q},t) \log(1+ \hat{\psi}^\ell(x,\bm{q},t)) \,\mathrm{d}x\,\mathrm{d}\bm{q} \leq C,
\end{equation*}
where $C$ is a positive constant independent of $\ell$,
which implies the uniform equi-integrability of the sequence $(\hat{\psi}^\ell)_{\ell \geq 0}$. By the Dunford--Pettis Theorem the sequence $( \hat{\psi}^\ell)_{\ell \geq 0}$ is weakly relatively compact in $L^1_M(\mathcal{O} \times (0,T))$. Hence, there exists a $\hat{\psi} \in L^1_M(\mathcal{O} \times (0,T))$ and a subsequence (not relabelled) such that
\begin{equation*}
\hat{\psi}^\ell \rightharpoonup \hat{\psi} \qquad \text{weakly in $L^1_M(\mathcal{O} \times (0,T)$}.
\end{equation*}
The next step is to show that 
\begin{equation*}
\hat{\psi}^\ell \to \hat{\psi} \qquad \text{a.e. in $\mathcal{O} \times (0,T)$}.
\end{equation*}
Let $\mathcal{O}_0$ be a Lipschitz subdomain of $\mathcal{O}$ such that $\mathcal{O}_0 \subset \overline{\mathcal{O}} \subset \mathcal{O}$. Since $M$ is bounded below in $\mathcal{O}_0$, we have that
\begin{equation*}
\sup_{t \in (0,T)} \| \sqrt{\hat{\psi}^\ell(\cdot,t)} \|^2_{L^2(\mathcal{O}_0)} + \int_0^T \| \sqrt{\hat{\psi}^\ell} \|^2_{W^{1,2}(\mathcal{O}_0)} \,\mathrm{d}t \leq C(\mathcal{O}_0).
\end{equation*}
Since $\mathcal{O}_0 \subset \mathcal{O} \subset \mathbb{R}^{(K+1)d}$, standard interpolation gives that
\begin{equation*}
\int_0^T \int_{\mathcal{O}_0} |\hat{\psi}^\ell|^{\frac{(K+1)d + 2}{d(K+1)}} \,\mathrm{d}x\,\mathrm{d}\bm{q}\,\mathrm{d}t = \int_0^T \int_{\mathcal{O}_0} \left|\sqrt{\hat{\psi}^\ell} \right|^{\frac{2((K+1)d + 2)}{d(K+1)}} \,\mathrm{d}x\,\mathrm{d}\bm{q}\,\mathrm{d}t \leq C(\mathcal{O}_0).
\end{equation*}
The application of H\"{o}lder's inequality gives that
\begin{equation}\label{eq353}
\int_0^T \int_{\mathcal{O}_0} |\nabla_{x,\bm{q}} \hat{\psi}^\ell|^{\frac{d(K+1) + 2}{d(K+1)+1}} \,\mathrm{d}x\,\mathrm{d}\bm{q}\,\mathrm{d}t \leq C(\mathcal{O}_0).
\end{equation}
Thanks to (\ref{eq332}) we have that
\begin{equation*}
\| \hat{\psi}^\ell \|_{L^\infty(Q; L^1_M(D))} \leq C,
\end{equation*}
where $C$ is a positive constant independent of $\ell$. 
It then follows that for any $q_1 \in (1,\infty)$ there exists a $q_2 > 1$ such that
\begin{equation}\label{eq355}
\| \hat{\psi}^\ell \|_{L^{q_1}(\Omega_0 \times(0,T); L^{q_2}(D_0))} \leq C(\mathcal{O}_0).
\end{equation}
Hence, using (\ref{eq332}), (\ref{eq355}), the fact that $\zeta(\cdot) \leq \zeta_{\max}$ and H\"{o}lder's inequality, there exists a $\delta > 0$ such that
\begin{equation}\label{eq356}
\| \zeta(\rho^\ell) \bm{v}^\ell \hat{\psi}^\ell \|_{L^{1+\delta}(\mathcal{O}_0 \times (0,T))} + \| \zeta(\rho^\ell) \Lambda_\ell(\hat{\psi}^\ell)(\nabla_x \bm{v}^\ell) \bm{q} \|_{L^{1+\delta}(\mathcal{O}_0 \times (0,T))} \leq C(\mathcal{O}_0).
\end{equation}
To apply the Div-Curl Lemma we define
\begin{align}\nonumber
H^\ell &\coloneqq (M \zeta(\rho^\ell)\hat{\psi}^\ell, M \zeta(\rho^\ell) \hat{\psi}^\ell \bm{v}^\ell - M \nabla_x \hat{\psi}^\ell, M \zeta(\rho^\ell) \Lambda_\ell(\hat{\psi}^\ell)(\nabla_x \bm{v}^\ell) \bm{q} - M \mathbb{A}(\nabla_{\bm{q}} \hat{\psi}^\ell)), \\
\nonumber Q^\ell &\coloneqq ( (1 + \hat{\psi}^\ell)^\alpha, \underbrace{0, \dots,0}_\text{$(d+Kd)$-times} ),
\end{align}
for some $\alpha \in (0, \frac{\delta}{1+\delta})$. Consequently, using (\ref{eq332}), (\ref{eq353}) and (\ref{eq356}), we deduce the existence of subsequences (not relabelled) such that
\begin{align*}
H^\ell &\rightharpoonup H \qquad \text{weakly in $L^{1+\delta}(\mathcal{O}_0 \times (0,T); \mathbb{R}^{1 + d +Kd})$}, \\
Q^\ell &\rightharpoonup Q \qquad \, \text{weakly in $L^{\frac{1}{\alpha}}(\mathcal{O}_0 \times (0,T); \mathbb{R}^{1 + d +Kd})$},
\end{align*}
where, noting the strong convergences of $\zeta(\rho^\ell)$ and $\bm{v}^\ell$,
\begin{align}\nonumber
H &\coloneqq (M\zeta(\rho)\hat{\psi}, M \zeta(\rho) \hat{\psi}\bm{v} - M \nabla_x \hat{\psi}, M \zeta(\rho)\overline{\Lambda_\ell(\hat{\psi})}(\nabla_x \bm{v}) \bm{q} - M \mathbb{A}(\nabla_{\bm{q}} \hat{\psi})), \\
\nonumber Q &\coloneqq ( \overline{(1 + \hat{\psi})^\alpha}, 0, \dots,0).
\end{align}
Similarly as above we can show that $\divergence_{t,x,\bm{q}} H^\ell$ and $\curl_{t,x,\bm{q}} Q^\ell$ are precompact in $W^{-1,2}(\mathcal{O}_0 \times (0,T))$. Thanks to our choice of $\alpha$ we can apply the Div-Curl Lemma to deduce that
\begin{equation*}
H^\ell \cdot Q^\ell \rightharpoonup H \cdot Q \qquad \text{weakly in $L^1(\mathcal{O}_0 \times (0,T))$}.
\end{equation*}
In particular, we have that
\begin{equation*}
M \zeta(\rho^\ell) \hat{\psi}^\ell (1+\hat{\psi}^\ell)^\alpha \rightharpoonup M \zeta(\rho) \hat{\psi} \overline{(1+\hat{\psi})^\alpha}.
\end{equation*}
Thanks to the strong convergence of $\zeta(\rho^\ell)$ we deduce that
\begin{equation*}
\hat{\psi}^\ell (1+\hat{\psi}^\ell)^\alpha \rightharpoonup  \hat{\psi} \overline{(1+\hat{\psi})^\alpha},
\end{equation*}
which then implies that
\begin{equation*}
\hat{\psi}^\ell \to \hat{\psi} \qquad \text{a.e. in $\mathcal{O}_0 \times (0,T)$}.
\end{equation*}
By selecting a nondecreasing sequence of nested set $( \mathcal{O}_0^k)_{k\geq 1}$ such that $\cup_{k=1}^\infty \mathcal{O}_0^k = \mathcal{O}$ and deducing pointwise convergence on each $\mathcal{O}_0^k$, we can argue following a diagonal procedure to deduce that there exists a subsequence such that
\begin{equation*}
\hat{\psi}^\ell \to \hat{\psi} \qquad \text{a.e. in $\mathcal{O} \times (0,T)$}.
\end{equation*}
By Vitali's Convergence Theorem we obtain that
\begin{equation*}
\hat{\psi}^\ell \to \hat{\psi} \qquad \text{strongly in $L^1(0,T; L^1_M(\mathcal{O}))$}.
\end{equation*}
By standard interpolation with (\ref{eq332}) we deduce that
\begin{equation}\label{eq365}
\hat{\psi}^\ell \to \hat{\psi} \qquad \text{strongly in $L^p(Q; L^1_M(D))$ for all $p \in [1,\infty)$}.
\end{equation}
We can now use \eqref{eq345} and \eqref{eq365} to pass to the limit $\ell \rightarrow \infty$ in \eqref{varrhol} to deduce that
\[ \varrho^\ell \rightarrow \varrho\qquad \text{strongly in $L^p(Q)$ for all $p \in [1,\infty)$,\quad where $\varrho(x,t) = \zeta(\rho(x,t))\int_D M(\bm{q}) \hat\psi(x,\bm{q},t)\,\mathrm{d}\bm{q}$}.\]
Thus, by noting \eqref{eq26} and \eqref{eq344a}, we have that
\begin{align}\label{muvarrhol}
\mu(\rho^\ell,\varrho^\ell) &\to \mu(\rho,\varrho) \qquad \text{strongly in $L^p(Q)$  for all $p \in [1,\infty)$}.
\end{align}

Following a similar argument as in (\ref{eq316})--(\ref{eq301}) we deduce the following convergence results:
\begin{align}
\begin{aligned}
\nabla_{x,\bm{q}} \hat{\psi}^\ell &\rightharpoonup \nabla_{x,\bm{q}} \hat{\psi} \qquad \quad \text{weakly in $L^1(0,T; L^1_M(\mathcal{O}; \mathbb{R}^{d(K+1)}))$}, \\
\nabla_{x,\bm{q}} \sqrt{\hat{\psi}^\ell} &\rightharpoonup \nabla_{x,\bm{q}} \sqrt{\hat{\psi}} \qquad \text{weakly in $L^2(0,T; L^2(\mathcal{O}; \mathbb{R}^{d(K+1)}))$}, \\
\label{eq368} \nabla_{x,\bm{q}} \hat{\psi}^\ell &\rightharpoonup \nabla_{x,\bm{q}} \hat{\psi} \qquad \quad \text{weakly in $L^2(Q; L^1_M(D; \mathbb{R}^{d(K+1)}))$}.
\end{aligned}
\end{align}
From the Lipschitz continuity of $\Gamma$, and therefore the Lipschitz continuity of $\Lambda_\ell$, we obtain for any $p \in [1,\infty)$ that
\begin{equation*}
\begin{split}
\| \Lambda_\ell(\hat{\psi}^\ell) - \hat{\psi} \|_{L^p(Q; L^1_M(D))}  &\leq \| \Lambda_\ell(\hat{\psi}^\ell) - \Lambda_\ell(\hat{\psi}) \|_{L^p(Q; L^1_M(D))} + \| \Lambda_\ell(\hat{\psi}) - \hat{\psi} \|_{L^p(Q; L^1_M(D))} \\
&\leq C\| \hat{\psi}^\ell - \hat{\psi} \|_{L^p(Q; L^1_M(D))} + \| \Lambda_\ell(\hat{\psi}) - \hat{\psi} \|_{L^p(Q; L^1_M(D))}.
\end{split}
\end{equation*}
The first term in the above inequality converges to $0$ as $\ell \to \infty$ on noting (\ref{eq365}). The second term in the above inequality also converges to $0$ as $\ell \to \infty$ on noting that $\Lambda_\ell(\hat{\psi})$ converges to $\hat{\psi}$ almost everywhere in $\mathcal{O} \times (0,T)$ and applying the Dominated Convergence Theorem. Therefore,
\begin{equation*}
\Lambda_\ell(\hat{\psi}^\ell) \to \hat{\psi} \qquad \text{strongly in $L^p(Q; L^1_M(D))$ for all $p \in [1,\infty)$}.
\end{equation*}
Moreover, using (\ref{eq335}) and (\ref{eq345}) we deduce that
\begin{equation}\label{eq371}
\zeta(\rho^\ell) \Lambda_\ell(\hat{\psi}^\ell) (\nabla_x \bm{v}^\ell) \rightharpoonup \zeta(\rho) \hat{\psi} (\nabla_x \bm{v}) \qquad \text{weakly in $L^p(Q; L^1_M(D; \mathbb{R}^{d \times d}))$ for any $p \in [1,2)$}.
\end{equation}
With the convergence results (\ref{eq335}), (\ref{eq345}) and (\ref{eq365}) we deduce that
\begin{equation}\label{eq372}
\zeta(\rho^\ell) \bm{v}^\ell \hat{\psi}^\ell \to \zeta(\rho) \bm{v} \hat{\psi} \qquad \text{strongly in $L^p(Q; L^1_M(D; \mathbb{R}^d))$ for any $p \in \left[1, \frac{2(d+2)}{d} \right)$}.
\end{equation}
Consequently, using the identity (\ref{eq330}) and the convergence results (\ref{eq368}), (\ref{eq371}) and (\ref{eq372}), it follows that
\begin{equation}\label{eq374}
\partial_t (M \zeta(\rho^\ell) \hat{\psi}^\ell) \rightharpoonup \partial_t (M \zeta(\rho) \hat{\psi}) \qquad \text{weakly in $L^p(Q; W^{-1,1}(D))$ for any $p \in [1,2)$}.
\end{equation}
\par
Next, we shall deduce the expression for the extra-stress tensor $\tau$. Performing partial integration on (\ref{eq327}) and noting the fact that $M$ has zero trace we rewrite $\tau^\ell$ as follows:
\begin{equation*}
\begin{split}
\tau^\ell &= -k \int_D \left[ K M \zeta(\rho^\ell)T_\ell(\hat{\psi}^\ell)I + \sum_{j=1}^K \zeta(\rho^\ell) T_\ell(\hat{\psi}^\ell)\nabla_{\bm{q}^j} M \otimes \bm{q}^j \right] \,\mathrm{d}\bm{q} \\
& = k \sum_{j=1}^K \int_D M\zeta(\rho^\ell) \nabla_{\bm{q}^j} T_\ell(\hat{\psi}^\ell) \otimes \bm{q}^j \,\mathrm{d}\bm{q}.
\end{split}
\end{equation*}
Using (\ref{eq345}), (\ref{eq365}) and (\ref{eq368}) we pass to the limit as $\ell \to \infty$ in this expression to identify the weak limit $\tau$ in \eqref{eq338} as
\begin{equation*}
\tau = k \sum_{j=1}^K \int_D M\zeta(\rho) \nabla_{\bm{q}^j} \hat{\psi} \otimes \bm{q}^j \,\mathrm{d}\bm{q}.
\end{equation*}
\par
With the convergence results (\ref{eq335}), (\ref{eq338}), (\ref{eq340}), (\ref{eq343}), (\ref{eq344}), (\ref{eq348}),  (\ref{eq365}), (\ref{muvarrhol}), (\ref{eq368}), (\ref{eq371}), (\ref{eq372}) and (\ref{eq374}), we pass to the limit as $\ell \to \infty$ in (\ref{eq328})--(\ref{eq330}) to deduce that
\begin{equation*}
\int_0^T [ \langle \partial_t \rho, \eta \rangle - (\bm{v}\rho, \nabla_x \eta)] \,\mathrm{d}t = 0,\quad \text{for all $\eta \in L^1(0,T; W^{1, \frac{q}{q-1}}(\Omega))$}, 
\end{equation*}
where $q \in (2,\infty)$ when $d=2$ and $q \in [3,6]$ when $d=3$,
\begin{equation*}
\begin{split}
&\int_0^T \langle \partial_t (\rho \bm{v}), \bm{w} \rangle \,\mathrm{d}t + \int_0^T [-(\rho \bm{v} \otimes \bm{v}, \nabla_x \bm{w}) + (\mu(\rho,\varrho)D(\bm{v}), \nabla_x \bm{w})] \,\mathrm{d}t \\
&\quad = \int_0^T [-(\tau, \nabla_x \bm{w}) +  (\rho\bm{f}, \bm{w})] \,\mathrm{d}t \quad \text{for all $\bm{w} \in L^s(0,T; W^{1,s}_{0,\divergence}(\Omega; \mathbb{R}^d))$ \quad where $s>2$}, 
\end{split}
\end{equation*}
and
\begin{equation*}
\begin{split}
&\int_0^T \left\langle \partial_t(M \zeta(\rho)\hat{\psi}), \varphi \right\rangle_{\mathcal{O}} -  \left(M \zeta(\rho) \bm{v} \hat{\psi}, \nabla_x \varphi \right)_{\mathcal{O}} -  \left(M \zeta(\rho) \hat{\psi}(\nabla_x \bm{v}) \bm{q}, \nabla_{\bm{q}} \varphi \right)_{\mathcal{O}} \,\mathrm{d}t \\
&\quad  + \int_0^T (M \nabla_x \hat{\psi}, \nabla_x \varphi)_{\mathcal{O}} + \left( M \mathbb{A}(\nabla_{\bm{q}} \hat{\psi}), \nabla_{\bm{q}} \varphi \right)_{\mathcal{O}} \,\mathrm{d}t = 0 \quad  \text{for all $\varphi \in L^\infty(0,T; W^{1,\infty}(\mathcal{O}))$}.
\end{split}
\end{equation*}
Letting $\ell \to \infty$ in (\ref{eq331}) we deduce the following energy inequality, stated in \eqref{energy}:
\begin{equation*}
\begin{split}
&k\int_{\mathcal{O}} M \zeta(\rho(\cdot, t)) \mathcal{F}(\hat{\psi}(\cdot,t)) \,\mathrm{d}x\,\mathrm{d}\bm{q} + \frac{1}{2}\int_\Omega \rho(\cdot,t) |\bm{v}(\cdot,t)|^2 \,\mathrm{d}x \\
&\quad + \int_0^t \int_\Omega \mu(\rho,\varrho) |D(\bm{v})|^2 \,\mathrm{d}x\,\mathrm{d}s + 4kC_1\int_0^t \int_{\mathcal{O}} M \left| \nabla_{x,\bm{q}} \sqrt{\hat{\psi}} \right|^2 \,\mathrm{d}x\,\mathrm{d}\bm{q}\,\mathrm{d}s \\
&\leq k\int_{\mathcal{O}} M \zeta(\rho_0) \mathcal{F}(\hat{\psi}_0) \,\mathrm{d}x\,\mathrm{d}\bm{q} + \frac{1}{2}\int_\Omega \rho_0 |\bm{v}_0|^2 \,\mathrm{d}x + \int_0^t (\rho \bm{f}, \bm{v}) \,\mathrm{d}s,
\end{split}
\end{equation*}
where $\mathcal{F}(s) = s\log s + 1$ for $s>0$ and $\mathcal{F}(0) \coloneqq \lim_{s \to 0+} \mathcal{F}(s) = 1$.

It remains to prove the weak continuity properties stated in \eqref{weakccont} and the attainment of the initial data asserted in 
\eqref{eqinicond}. First, we set $\bm{w}(x,t) \coloneqq \chi_{[0,t]} \bm{u}(x)$ in (\ref{eq329}), where $\bm{u} \in W^{1,s}_{0,\divergence}(\Omega; \mathbb{R}^d)$ is arbitrary,
with $s>2$, to deduce that
\begin{equation}\label{eq380}
\begin{split}
&((\rho^\ell \bm{v}^\ell)(t), \bm{u}) + \int_0^t [-(\rho^\ell \bm{v}^\ell \otimes \bm{v}^\ell, \nabla_x \bm{u}) + (\mu(\rho^\ell,\varrho^\ell)D(\bm{v}^\ell), \nabla_x \bm{u})] \, \mathrm{d}\tau \\
&\quad = \int_0^t [-(\tau^\ell, \nabla_x \bm{u}) + (\rho^\ell\bm{f}, \bm{w})] \, \mathrm{d}\tau + (\rho_0 \bm{v}_0, \bm{u}).
\end{split}
\end{equation}
Letting $\ell \to \infty$ in (\ref{eq380}) and using the convergence results above we deduce, for almost all $t \in (0,T)$, that, for each $\bm{u} \in W^{1,s}_{0,\divergence}(\Omega; \mathbb{R}^d)$ with $s>2$,
\begin{equation}\label{eq-rv-cont}
\begin{split}
&((\rho\bm{v})(t), \bm{u}) + \int_0^t [-(\rho \bm{v} \otimes \bm{v}, \nabla_x \bm{u}) + (\mu(\rho,\varrho)D(\bm{v}), \nabla_x \bm{u})] \, \mathrm{d} \tau \\
&\quad = \int_0^t [-(\tau, \nabla_x \bm{u}) + \langle \rho \bm{f}, \bm{w} \rangle] \, \mathrm{d} \tau + (\rho_0 \bm{v}_0, \bm{u}).
\end{split}
\end{equation}
After a possible redefinition of $\rho\bm{v}$ on a set of measure zero, the above equation  holds for all $t \in (0,T)$. Thus, by letting $t \to 0_+$, we deduce that
\begin{equation}\label{eq382}
\lim_{t \to 0+} ((\rho\bm{v})(t), \bm{u}) = (\rho_0 \bm{v}_0, \bm{u}) \qquad \text{for all $\bm{u} \in W^{1,s}_{0,\divergence}(\Omega; \mathbb{R}^d)$ with $s>2$}.
\end{equation}
Replacing $t$ with $t'$ in \eqref{eq-rv-cont} and subtracting the resulting equality from \eqref{eq-rv-cont} we deduce, for almost every $t,t' \in (0,T)$ that, for each $\bm{u} \in W^{1,s}_{0,\divergence}(\Omega; \mathbb{R}^d)$ with $s>2$, 
\begin{equation}\label{eq-rv-cont1}
\begin{split}
&((\rho \bm{v})(t), \bm{u}) + \int_{t'}^t [-(\rho \bm{v} \otimes \bm{v}, \nabla_x \bm{u}) + (\mu(\rho,\varrho)D(\bm{v}), \nabla_x \bm{u})] \, \mathrm{d} \tau \\
&\quad = \int_{t'}^t [-(\tau, \nabla_x \bm{u}) + \langle \rho \bm{f}, \bm{w} \rangle] \, \mathrm{d} \tau + ((\rho \bm{v})(t'), \bm{u}).
\end{split}
\end{equation}
As the integrands in the integrals appearing on the left-hand side and right-hand side of \eqref{eq-rv-cont1} belong to $L^1(0,T)$, it follows by the Fundamental Theorem of Calculus for Lebesgue integral, again, after a possible of redefinition of $\rho\bm{v}$ on a set of measure zero, that $t \mapsto ((\rho \bm{v})(t),\bm{u})$ is, for each $\bm{u}\in W^{1,s}_{0,\divergence}(\Omega; \mathbb{R}^d)$ with $s>2$, absolutely continuous on $[0,T]$.

Similarly, setting $\varphi \coloneqq \chi_{[0,t]} \phi(x,\bm{q})$ in (\ref{eq330}), where $\phi \in W^{1,\infty}(\mathcal{O})$ is arbitrary, we deduce that
\begin{equation}\label{eq383}
\begin{split}
&( M (\zeta(\rho^\ell) \hat{\psi}^\ell)(t), \phi)_{\mathcal{O}} - \int_{0}^t \left[ \left(M \zeta(\rho^\ell) \bm{v}^\ell \hat{\psi}^\ell, \nabla_x \phi \right)_{\mathcal{O}} -  \left(M \zeta(\rho^\ell) \Lambda_\ell(\hat{\psi}^\ell)(\nabla_x \bm{v}^\ell) \bm{q}, \nabla_{\bm{q}} \phi \right)_{\mathcal{O}} \right] \, \mathrm{d} \tau \\
&\quad  + \int_{0}^t \left[ (M \nabla_x \hat{\psi}^\ell, \nabla_x \phi)_{\mathcal{O}} + \left( M \mathbb{A}(\nabla_{\bm{q}} \hat{\psi}^\ell), \nabla_{\bm{q}} \phi \right)_{\mathcal{O}} \right] \, \mathrm{d} \tau = (M\zeta(\rho_0) \hat{\psi}_0, \phi )_{\mathcal{O}}.
\end{split}
\end{equation}
Letting $\ell \to \infty$ in (\ref{eq383}) and using the convergence results above we deduce, for almost all $t \in (0,T)$, that, for each $\phi \in W^{1,\infty}(\mathcal{O})$, 
\begin{equation*}
\begin{split}
&( M (\zeta(\rho) \hat{\psi})(t), \phi)_{\mathcal{O}} - \int_0^t \left[ \left(M \zeta(\rho) \bm{v} \hat{\psi}, \nabla_x \phi \right)_{\mathcal{O}} -  \left(M \zeta(\rho)\hat{\psi})(\nabla_x \bm{v}) \bm{q}, \nabla_{\bm{q}} \phi \right)_{\mathcal{O}} \right] \, \mathrm{d} \tau \\
&\quad  + \int_0^t \left[ (M \nabla_x \hat{\psi}, \nabla_x \phi)_{\mathcal{O}} + \left( M \mathbb{A}(\nabla_{\bm{q}} \hat{\psi}), \nabla_{\bm{q}} \phi \right)_{\mathcal{O}} \right] \, \mathrm{d} \tau = (M\zeta(\rho_0) \hat{\psi}_0, \phi )_{\mathcal{O}}.
\end{split}
\end{equation*}
After a possible redefinition of $\zeta(\rho)\hat{\psi}$ on a set of measure zero, the above equation  holds for all $t \in (0,T)$. Letting $t \to 0_+$, we deduce that
\begin{equation}\label{eq385}
\lim_{t \to 0+} ( M (\zeta(\rho)\hat{\psi})(t), \phi)_{\mathcal{O}} = (M\zeta(\rho_0) \hat{\psi}_0, \phi )_{\mathcal{O}} \qquad \text{for all $\phi \in W^{1,\infty}(\mathcal{O})$}.
\end{equation}
Similarly as in the case of $\rho \bm{v}$ above,
after a possible redefinition on a set of measure zero, the function $t \mapsto ( M (\zeta(\rho)\hat{\psi})(t), \phi)_{\mathcal{O}}$
is absolutely continuous on $[0,T]$ for each $\phi \in W^{1,\infty}(\mathcal{O})$.
This completes the proofs of the assertions \eqref{weakccont} and \eqref{eqinicond}, and thereby the proof of Theorem \ref{thm1} is also complete.

\bibliographystyle{abbrv}
\bibliography{Incompressible}

\end{document}